\documentclass[reqno, 12pt]{amsart}
\usepackage{amsmath,latexsym,amsfonts,amssymb,amsthm}
\usepackage{geometry}
\usepackage[dvipsnames]{xcolor}
\usepackage[colorlinks,citecolor=OliveGreen,linkcolor=Mahogany,urlcolor=Plum,pagebackref]{hyperref}
\usepackage{mathrsfs}
\usepackage{graphicx,color}
\usepackage[alphabetic]{amsrefs}
\usepackage[all,cmtip]{xy}
\usepackage{bbm, stmaryrd}
\usepackage{tabu}
\usepackage{enumitem}
\usepackage{tikz}
\usepackage{tikz-cd}
\usepackage{comment}
\usetikzlibrary{calc, decorations.markings}

\newtheorem{prop}{Proposition}[section]
\newtheorem{thm}[prop]{Theorem}

\newtheorem{cor}[prop]{Corollary}
\newtheorem{conj}[prop]{Conjecture}
\newtheorem{lem}[prop]{Lemma}
\newtheorem{que}[prop]{Question}

\theoremstyle{definition}

\newtheorem{defn}[prop]{Definition}
\newtheorem{expl}[prop]{Example}
\newtheorem{rem}[prop]{\it Remark}

\newtheorem*{claim*}{Claim}

\newcommand{\bP}{\mathbb{P}}

\newcommand{\bR}{\mathbb{R}}
\newcommand{\bA}{\mathbb{A}}
\newcommand{\bQ}{\mathbb{Q}}
\newcommand{\bZ}{\mathbb{Z}}
\newcommand{\bN}{\mathbb{N}}

\newcommand{\bG}{\mathbb{G}}
\newcommand{\bT}{\mathbb{T}}
\newcommand{\bS}{\mathbb{S}}
\newcommand{\bF}{\mathbb{F}}
\newcommand{\bk}{\mathbbm{k}}

\newcommand{\oX}{\overline{X}}

\newcommand{\tX}{\widetilde{X}}

\newcommand{\tD}{\widetilde{D}}

\newcommand{\tw}{\widetilde{w}}

\newcommand{\tW}{\widetilde{W}}

\newcommand{\cX}{\mathcal{X}}

\newcommand{\cO}{\mathcal{O}}
\newcommand{\cL}{\mathcal{L}}
\newcommand{\cI}{\mathcal{I}}

\newcommand{\cF}{\mathcal{F}}
\newcommand{\cG}{\mathcal{G}}

\newcommand{\cE}{\mathcal{E}}
\newcommand{\cD}{\mathcal{D}}

\newcommand{\cR}{\mathcal{R}}
\newcommand{\cH}{\mathcal{H}}

\newcommand{\cS}{\mathcal{S}}

\newcommand{\fa}{\mathfrak{a}}

\newcommand{\fg}{\mathfrak{g}}
\newcommand{\fS}{\mathfrak{S}}

\newcommand{\hX}{\widehat{X}}

\newcommand{\Spec}{\mathrm{Spec}}
\newcommand{\Supp}{\mathrm{Supp}}
\newcommand{\Hom}{\mathrm{Hom}}

\newcommand{\Aut}{\mathrm{Aut}}

\newcommand{\ord}{\mathrm{ord}}

\newcommand{\Diff}{\mathrm{Diff}}

\newcommand{\gr}{\mathrm{gr}}

\newcommand{\wt}{\mathrm{wt}}
\newcommand{\QM}{\mathrm{QM}}

\newcommand{\Trop}{\mathrm{Trop}}
\newcommand{\trop}{\mathrm{trop}}

\newcommand{\triv}{\mathrm{triv}}

\newcommand{\tcX}{\widetilde{\mathcal{X}}}

\newcommand{\Proj}{\mathrm{Proj}}

\newcommand{\fX}{\mathfrak{X}}
\newcommand{\fD}{\mathfrak{D}}
\newcommand{\fE}{\mathfrak{E}}

\newcommand{\fH}{\mathfrak{H}}
\newcommand{\fP}{\mathfrak{P}}
\newcommand{\fF}{\mathfrak{F}}

\newcommand{\fL}{\mathfrak{L}}
\newcommand{\fY}{\mathfrak{Y}}
\newcommand{\fG}{\mathfrak{G}}
\newcommand{\fB}{\mathfrak{B}}

\newcommand{\hf}{\hat{f}}
\newcommand{\hD}{\widehat{D}}

\newcommand{\GL}{\mathrm{GL}}
\newcommand{\PGL}{\mathrm{PGL}}

\newcommand{\SV}{\mathrm{SV}}
\newcommand{\PsAut}{\mathrm{PsAut}}
\newcommand{\bmu}{\boldsymbol{\mu}}
\newcommand{\hcX}{\widehat{\mathcal{X}}}
\newcommand{\hcD}{\widehat{\mathcal{D}}}
\newcommand{\hfX}{\widehat{\mathfrak{X}}}

\newcommand{\hfP}{\widehat{\mathfrak{P}}}
\newcommand{\hGamma}{\widehat{\Gamma}}

\newcommand{\oS}{\overline{S}}
\newcommand{\os}{\overline{s}}

\newcommand{\Sk}{\mathrm{Sk}}

\newcommand{\coeff}{\mathrm{coeff}}

\newcommand{\bH}{\mathbb{H}}
\newcommand{\Int}{\mathrm{Int}}
\newcommand{\an}{\mathrm{an}}
\newcommand{\val}{\mathrm{val}}

\newcommand{\rfg}{\mathrm{fg}}
\newcommand{\qm}{\mathrm{qm}}
\newcommand{\ratrank}{\mathrm{rat.rank}}

\newcommand{\LCP}{\mathrm{LCP}}
\newcommand{\rdiv}{\mathrm{div}}
\newcommand{\Rees}{\mathrm{Rees}}
\newcommand{\hDelta}{\widehat{\Delta}}
\newcommand{\Exc}{\mathrm{Exc}}

\newcommand{\tfX}{\widetilde{\mathfrak{X}}}
\newcommand{\Bl}{\mathrm{Bl}}
\newcommand{\oK}{\overline{K}}
\newcommand{\red}{\mathrm{red}}

\newcommand{\sX}{\mathscr{X}}
\newcommand{\sD}{\mathscr{D}}
\newcommand{\pr}{\mathrm{pr}}
\newcommand{\ux}{\underline{x}}
\newcommand{\oell}{\overline{\ell}}
\newcommand{\ocS}{\overline{\mathcal{S}}}
\newcommand{\ofS}{\overline{\mathfrak{S}}}
\newcommand{\codim}{\mathrm{codim}}
\newcommand{\PL}{\mathrm{PL}}
\newcommand{\Cone}{\mathrm{Cone}}
\newcommand{\VPi}{\mathrm{Vert}(\Pi)}
\newcommand{\Stab}{\mathrm{Stab}}
\newcommand{\Adm}{\mathrm{Adm}}

\numberwithin{equation}{section}

\newcommand{\YL}[1]{{\textcolor{blue}{[Yuchen: #1]}}}

\title{Special valuations and automorphisms of affine log Calabi--Yau varieties}

\date{\today}

\author{Yuchen Liu}
\address{Department of Mathematics, Northwestern University, Evanston, IL 60208, USA.}
\email{yuchenl@northwestern.edu}

\begin{document}

\begin{abstract}
Let $U$ be an affine log Calabi--Yau variety. Although special
and finitely generated valuations are defined using a log CY-Fano
compactification of $U$, we prove that the corresponding skeleta
in the dual complex are independent of this choice, and hence
invariant under $\Aut(U)$. We also show that a finitely generated
valuation is special if and only if it is maximal with respect to
a natural partial order induced by regular functions on $U$.

We then consider the affine log Calabi--Yau threefold obtained as the complement of the Markov cubic surface in $\bA^3$. We
describe the action of the three Vieta involutions on the special
skeleton and relate it to the $(\infty,\infty,\infty)$-triangle
reflection group on the hyperbolic plane. We also show that, for every finite triangulation of the dual complex, the special skeleton fails to be locally closed on some simplex. Via the cone construction, this provides
a counterexample to a conjecture of the author and Xu.
\end{abstract}

\maketitle

\tableofcontents

\section{Introduction}

A recurring theme in birational geometry and K-stability is that degenerations of an algebraic variety can be encoded by valuations. For log Calabi--Yau pairs, these valuations naturally lie in the cone of log canonical places, whose projectivization is the dual complex. The aim of this paper is to study distinguished loci in the dual complex corresponding to particularly well-behaved degenerations, and to show
that these loci are intrinsic to the underlying affine log Calabi--Yau variety.

Let $(X,D)$ be an lc log Calabi--Yau (CY) pair such that $D$ is reduced and fully supports an ample divisor, and assume that $U:=X\setminus D $ is klt. We call $U$ an \emph{affine log Calabi--Yau variety} and $(X,D)$ a \emph{log Calabi--Yau compactification} of $U$. Any two such compactifications are crepant birational; in particular, they determine the same cone of log canonical places. We denote this cone and its projectivization by \[ U^{\trop}:=\LCP(X,D), \qquad   \cD(U):= \bigl(U^{\trop}\setminus\{v_{\triv}\}\bigr)/\bR_{>0}. \]
Then $\cD(U)$ is naturally identified with the dual complex $\cD(X,D)$; see \cite{dFKX17, KX15}. Thus $U^{\trop}$ and $\cD(U)$ depend only on $U$, rather than on the choice of compactification.
The ample boundary condition also equips $X$ with a natural auxiliary log Fano structure: one may choose a $\bQ$-divisor $ 0\leq \Delta\leq D $ such that $(X,\Delta)$ is a klt log Fano pair. We call $(X,D;\Delta)$ a \emph{log CY-Fano compactification} of $U$.

Fix such a compactification. A valuation $v\in U^{\trop}$ is called \emph{finitely generated} with respect to $(X,\Delta)$ if the associated graded ring $\gr_v R$ of the anti-log-canonical section ring of $(X,\Delta)$ is finitely generated. It is called \emph{special} if, in addition, the induced log Fano degeneration $(X_v,\Delta_v)$ is klt.
For divisorial valuations, specialness corresponds precisely to special test
configurations in the sense of Li--Xu \cite{LX14}, who showed that K-stability of a Fano variety can be tested using special test configurations.
More generally, special valuations correspond to special $\bR$-test configurations in the sense of \cite{DS20, HL20,BLXZ23}; see Section~\ref{sec:prelim} for precise definitions.

Special valuations have played an important role in K-stability and in the study of canonical degenerations. In various settings, showing that a minimizing valuation is finitely generated or even special is a crucial step in the study of K-moduli, the singular Yau--Tian--Donaldson conjecture, K\"ahler--Ricci soliton degenerations, and the stable degeneration conjecture for klt singularities \cite{LXZ22,BLXZ23,XZ22,Che25}. In the present paper, however, we shift the focus from a fixed log Fano pair to the underlying affine log CY variety, and study the collection of all special and finitely generated valuations.

Our first main result proves that the subsets of $\cD(U)$ consisting respectively of finitely generated and special valuations are intrinsic to the affine log CY variety $U$, despite their apparent dependence on the choice of log CY-Fano compactification $(X,D;\Delta)$.

\begin{thm}\label{thm:main}
Let $U$ be an affine log CY variety. Then the following sets
\begin{align*}
 \cD^{\SV}(X, D; \Delta)& := \{[v]\in \cD(U)\mid v\textrm{ is special with respect to } (X,\Delta)\},\\
\cD^{\rfg}(X,D;\Delta)& :=\{[v]\in \cD(U)\mid v\textrm{ is finitely generated with respect to } (X,\Delta)\}
\end{align*}
are independent of the choice of log CY-Fano compactification $(X,D;\Delta)$ of $U$. 
\end{thm}

Accordingly, we define intrinsic subsets
\[
\cD^{\SV}(U)\subset \cD^{\rfg}(U)\subset \cD(U),
\]
which we call the \emph{special skeleton} and the \emph{finitely generated skeleton} of $U$, respectively. Since every automorphism of $U$ acts naturally on $\cD(U)$, it follows that both skeleta are preserved by this action.

\begin{cor}\label{cor:aut}
Let $U$ be an affine log CY variety. Then both the special skeleton $\cD^{\SV}(U)$ and the finitely generated skeleton $\cD^{\rfg}(U)$  are invariant under the $\Aut(U)$-action on the dual complex $\cD(U)$. 
\end{cor}

Compactification-independence suggests that specialness should
admit a characterization directly in terms of the affine log CY
variety $U$. We obtain such a characterization using regular
functions on $U$. There is a natural partial order on
$U^{\trop}$: for
$v,w\in U^{\trop}$, we write
\[
v\preceq w
\quad\textrm{if and only if}\quad
v(f)\leq w(f)
\quad\text{for every }f\in \cO_U(U).
\]
Our second main result  characterizes special valuations as the
maximal elements of this order among finitely generated
valuations.

\begin{thm}\label{thm:maximality-intro}
Every special valuation in $U^{\trop}$ is maximal with respect to $\preceq$. Conversely, every finitely generated maximal valuation in $U^{\trop}$ is special.
In particular, a divisorial valuation $v\in U^{\trop}$ is special if and only if it is maximal with respect to $\preceq$.
\end{thm}

We also give an equivalent description of $\preceq$ using the
tropical theta functions arising from the valuatively
independent bases of \cite{BL26}; see Proposition \ref{prop:theta-order}.

The intrinsic nature of the special skeleton makes it natural to study the dynamics of the $\Aut(U)$-action on $\cD^{\SV}(U)$. We carry this out for the complement of the Markov cubic surface
\[
S=(xyz+x^2+y^2+z^2=0)\subset \bA^3.
\]
Let $U:=\bA^3\setminus S$. Then $U$ is an affine log CY threefold, with
 log CY-Fano compactification
$(\bP^3,\overline S+H;0)$,
where $\overline S$ is the projective closure of $S$ and $H$ is the hyperplane at infinity. 
The automorphism group of $U$ contains the subgroup $G$ generated by the three Vieta involutions. Using the tropical dynamics studied by Jang \cite{Jan23}, we obtain an explicit description of the $G$-action on the special skeleton, revealing a striking connection with hyperbolic geometry.

\begin{thm}\label{thm:intro-markov}
Let $U=\bA^3\setminus S$ be as above, and let
$G<\Aut(U)$ be the subgroup generated by the three Vieta involutions.
Then 
\[
\cD^{\SV}(U)=\{\ord_S\}\sqcup G\cdot \Pi,
\]
where $G\cdot\Pi$ is a connected component and $\Pi$ is a rational polyhedral fundamental domain. There is a
$G$-equivariant continuous bijection
\[
\Psi: G\cdot \Pi\to
\bH^2\cup\bP^1(\bQ),
\]
where the target is endowed with the Satake topology and the $G$-action
is the $(\infty,\infty,\infty)$-triangle reflection group action.
Moreover, its restriction $\Psi^\circ:
\operatorname{Int}(G\cdot \Pi)
\xrightarrow{\sim}\bH^2$
is a $G$-equivariant homeomorphism.
\end{thm}

See Figures~\ref{fig:SV-Markov} and~\ref{fig:hyp-reflection}
for the geometry of $G\cdot\Pi$ and the corresponding
$(\infty,\infty,\infty)$-triangle reflection group action.

The hyperbolic description also reveals an unexpected topological property of the special skeleton: with respect to any finite triangulation of $\cD(U)$, its intersection with the open simplices need not be locally closed. This disproves the natural global analogue for special valuations of the simplexwise local closedness conjectures for Koll\'ar valuations proposed in \cite[Conjecture~1.8(1),(1')]{LX24}. Moreover, via the cone construction, we obtain a counterexample to the original local conjectures.

\begin{thm}\label{thm:counterex-LX24}
Let $U=\bA^3\setminus S$ be as in Theorem \ref{thm:intro-markov}. Then for any finite triangulation of $\cD(U)$, there exists an open simplex $C^\circ$ such that $\cD^{\SV}(U)\cap C^\circ$ is not open in its closure in $C^{\circ}$.
Moreover, the cone construction yields a $4$-dimensional counterexample to \cite[Conjecture~1.8(1),(1')]{LX24}.
\end{thm}

\begin{rem}
During the preparation of this paper, we learned that Peng \cite{Pen25} studied special valuations for log CY pairs $(X,D)$ with $X$ a del Pezzo surface and $D$ a nodal anticanonical curve. In particular, \cite[Theorem~1.3(d)]{Pen25} proves that specialness is preserved under certain contractions of boundary components, giving a compactification-independence result that may be viewed as a special case of Theorem~\ref{thm:main} for special valuations in dimension two.
\end{rem}

\subsection*{Organization of the paper}

In Section~\ref{sec:prelim}, we review background on log CY-Fano triples,
valuations, test configurations, and multi-degenerations, and
establish several auxiliary results used later in the paper.
We also recall the non-Archimedean and tropical framework
needed for the Markov cubic complement.

In Section~\ref{sec:intrinsic}, we prove the intrinsic results on special and
finitely generated valuations. The key idea is to interpret
finite generation and specialness in terms of the central
fibers of multi-degenerations and compare these degenerations
under crepant birational changes of compactification. This
yields Theorem~\ref{thm:main} and Corollary~\ref{cor:aut}. We then relate the
partial order induced by regular functions to log canonical
places of the corresponding multi-degenerations, leading to
the maximality characterization in Theorem~\ref{thm:maximality-intro}.

Section~\ref{sec:surfaces} presents three examples of affine log CY surfaces
with infinite discrete automorphism groups, illustrating
different possible behaviors of the special and finitely
generated skeleta.

In Section~\ref{sec:Markov}, we study the complement of the Markov cubic.
After describing its dual complex and the action of the Vieta
involutions, we relate this action to the tropical dynamics
studied by Jang \cite{Jan23}. We then determine the special
and finitely generated skeleta, establish the hyperbolic
description in Theorem~\ref{thm:intro-markov}, and prove the failure of
simplexwise local closedness and the resulting counterexample
in Theorem~\ref{thm:counterex-LX24}.

Finally, Section~\ref{sec:discussions} discusses a possible non-Archimedean
analogue of the cone conjecture, connections with cluster
complexes and mirror symmetry, and questions concerning finite
generation and maximality.

\subsection*{AI disclosure}
All mathematical results and arguments in this paper were developed by the author. ChatGPT (OpenAI) was used to assist with drafting and polishing the exposition of several proofs and examples from the author's arguments (notably in Sections \ref{sec:surfaces} and \ref{sec:topology}), for editorial assistance, notation consistency checks, and Python code for the figures in Section \ref{sec:Markov}. Claude (Anthropic) was used for feedback on the draft. All AI-assisted content was independently checked and revised by the author, who takes full responsibility for the paper.

\subsection*{Acknowledgements} 

We would like to thank Eduardo Alves da Silva, Harold Blum, Steven Cutkosky, Rankeya Datta, Fernando Figueroa, Simion Filip, Kento Fujita, Paul Hacking, Andres H\"oring, Mattias Jonsson, Sean Keel, J\'anos Koll\'ar, Wendelin Lutz, Junyao Peng, Chenyang Xu, Eric Zaslow, and Ziquan Zhuang for helpful discussions. The author was partially supported by NSF CAREER Grant DMS-2237139 and an AT\&T Research Fellowship from Northwestern University.

\section{Preliminaries and auxiliary results}\label{sec:prelim}

In this section, we collect the preliminaries and auxiliary results needed throughout the paper.
We first collect the basic background on log CY-Fano triples, valuations, and degenerations. Then  Sections \ref{sec:max-rank} and \ref{sec:geodesic-normal-cone}  establish auxiliary results that will be used in Section \ref{sec:Markov}. Finally, Section \ref{sec:an-trop} introduces the non-Archimedean and tropical language needed for the study of the Markov cubic example in Section \ref{sec:Markov}.

We work over an algebraically closed field $\bk$ of characteristic zero. We follow the standard notation and convention from \cite{KM98, Kol13}.

\subsection{Log CY-Fano triples}

\begin{defn}
    A \emph{pair} $(X,D)$ consists of a demi-normal variety $X$ and an effective $\bQ$-divisor $D$ where $\Supp(D)$ does not contain any codimension $1$ singular locus of $X$, such that $K_X + D$ is $\bQ$-Cartier.
\end{defn}

\begin{defn}
Let $(X,D)$ be a normal pair. A \emph{divisor $E$ over} $X$ is a prime divisor on a normal variety $Y$ together with a birational morphism $\mu: Y\to X$. We define the \emph{log discrepancy} of $E$ with respect to $(X,D)$ by
\[
A_{X,D}(E):= 1 + \coeff_E(K_Y - \mu^*(K_X+D)).
\]
\end{defn}

\begin{defn}
A \emph{log CY pair} $(X,D)$ is a projective slc pair such that $K_X + D\sim_{\bQ} 0$. A \emph{log Fano pair} $(X,\Delta)$ is a projective slc pair such that $-K_X - \Delta$ is ample. 

A \emph{log CY-Fano triple} $(X,D;\Delta)$ is a log CY pair $(X,D)$ together with a log Fano pair $(X,\Delta)$ such that $\Delta \leq D$. We say that a log CY-Fano triple $(X,D;\Delta)$ is \emph{lc-klt} if $(X,D)$ is lc and $(X,\Delta)$ is klt.
\end{defn}

We note that in some literature, log CY pairs (resp.\ log Fano pairs) are often assumed to be lc (resp.\ klt) which is stronger than our assumptions. Nevertheless,  these stronger assumptions will be imposed at certain places of the paper. In the terminology of \cite{BABWILD, BL24}, a log CY-Fano triple $(X,D;\Delta)$ precisely corresponds to a \emph{boundary polarized CY pair} $(X,\Delta+ B)$ with $B:= D-\Delta$.

\begin{defn}\label{def:affine-CY}
Let $(X,D)$ be an lc log CY pair with $D$ reduced such that $U:= X\setminus D$ is klt and $D$ fully supports an ample divisor. We call $U$ an \emph{affine log CY variety} and $(X,D)$ a \emph{log CY compactification} of $U$. If $0\leq \Delta \leq D$ is a $\bQ$-divisor such that $(X,\Delta)$ is a klt log Fano pair, we call $(X,D;\Delta)$  a \emph{log CY-Fano compactification} of $U$. Such $\Delta$ always exists by the following proposition.
\end{defn}

\begin{prop}\label{prop:log-CY-comp}
Let $U$ be an affine log CY variety with a log CY compactification $(X,D)$. Then there exists a $\bQ$-divisor $0\leq \Delta\leq D$ on $X$ such that $(X,D;\Delta)$ is a log CY-Fano compactification of $U$.
\end{prop}

\begin{proof}
By assumption, there exists an ample effective Cartier divisor $L$ on $X$ fully supported in $D$. Let $\Delta := D-\epsilon L$ for $0<\epsilon\ll 1$. Then clearly $\Delta\leq D$ and $-K_X-\Delta\sim_{\bQ} \epsilon L$ is ample. It remains to show that $(X,\Delta)$ is klt. Since $U = X\setminus D$ is klt, it suffices to show that $A_{X,\Delta}(E)>0$ for any divisor $E$ over $X$ satisfying $c_X(E)\subset D$. This is true as 
\[
A_{X,\Delta}(E) = A_{X,D}(E) + \ord_E(\epsilon L) > A_{X,D}(E) \geq 0. 
\]
Thus $(X,\Delta)$ is klt. 
\end{proof}

The following lemma is probably well-known. A special case can be found in \cite[Lemma 1.4(2)]{GHK15}.

\begin{lem}\label{lem:cbir}
 Suppose $(X,D)$ and $(X',D')$ are two lc log CY pairs. Let $B$ and $B'$ be effective ample $\bQ$-Cartier $\bQ$-divisors on $X$ and $X'$ respectively such that $B_{\red}\leq D$ and $B_{\red}'\leq D'$. 
 Denote by $U:=X\setminus \Supp(B)$ and $U':=X'\setminus \Supp(B')$. Suppose there is an isomorphism of pairs $g: (U, D|_U) \xrightarrow{\sim} (U', D'|_{U'})$. Then the induced birational map $g:(X,D) \dashrightarrow (X',D')$ is crepant birational. In particular, any two log CY compactifications of an affine log CY variety are crepant birational.
\end{lem}

\begin{proof}
Let $\tX$ be the normalization of the graph of $g:X\dashrightarrow X'$. Then we have birational morphisms $p: \tX \to X$ and $q: \tX \to X'$ such that $g = q\circ p^{-1}$. Then we write
\[
0 \sim_{\bQ} q^*(K_{X'} + D') = K_{\tX} + \tD' + E',
\]
where every irreducible component of $\tD'$ is not $p$-exceptional, while $E'$ is $p$-exceptional. Thus we have 
\[
K_X + D\sim_{\bQ} 0\sim_{\bQ} p_*(K_{\tX} + \tD' + E') = K_X + p_* \tD' .
\]
This implies that $D \sim_{\bQ} p_* \tD'$.

Next, we claim that $D \geq p_* \tD'$, which together with $D \sim_{\bQ} p_* \tD'$ implies $D =  p_* \tD'$. Since $g: (U, D|_U)\xrightarrow{\sim} (U', D'|_{U'})$ is an isomorphism, we know that $p$ and $q$ are isomorphic over $U$ and $U'$ respectively. 
Thus we have $D|_U = p_* \tD'|_{U}$ as $E'$ is $p$-exceptional. Thus it suffices to show that for every prime divisor $B_i\subset X\setminus U = \Supp(B)$ we have $\coeff_{B_i}(D) \geq \coeff_{B_i}(p_* \tD')$. Since $D\geq B_{\red}\geq B_i$ and $(X,D)$ is lc, we know that $\coeff_{B_i}(D) = 1$. Since $(X',D')$ is lc, we know that 
\[
\coeff_{B_i}(p_*\tD') = \coeff_{p_*^{-1}B_i} (\tD') =   1 - A_{X', D'}(\ord_{B_i}) \leq 1.
\]
Thus the claim is proved. 

Next, denote by $p^*(K_X+ D) = K_{\tX} + \tD' + E$ where $E$ is $p$-exceptional. Then we have 
\[
0\sim_{\bQ} p^*(K_X + D) - q^*(K_{X'} + D')  = E - E'
\]
is $p$-exceptional. By negativity lemma we conclude that $E= E'$ and hence $p^*(K_X+D) = q^*(K_{X'}+D')$, i.e.\ $g$ is crepant birational.

For the final statement, let \((X,D)\) and \((X',D')\) be two log CY compactifications of an affine log CY variety \(U\). By Proposition \ref{prop:log-CY-comp}, there exist effective ample $\bQ$-Cartier $\bQ$-divisors \(B=D-\Delta\) and \(B'=D'-\Delta'\) fully supported on \(D\) and \(D'\), respectively. Hence $ B_{\rm red}=D$ and $B'_{\rm red}=D'$.
Thus the result follows from the first statement.
\end{proof}


\subsection{Valuations and dual complexes}

Let $X$ be a normal variety over $\bk$ throughout this subsection.

\begin{defn}
A \emph{valuation} $v$ on $X$ is an $\bR$-valuation $v: K(X)^{\times}\to \bR$ satisfying $v(\bk^{\times}) = 0$. We set $v(0) = \infty$ by convention. The \emph{center} of a valuation $v$ on $X$ is a scheme-theoretic point $x\in X$ such that the local ring $\cO_{X,x}$ is dominated by the valuation ring $\cO_v$. Note that the center exists if $X$ is proper but may not exist otherwise, and is unique if it exists by the separatedness of $X$.

The \emph{valuative group} of a valuation $v$ is denoted by $\Phi_v:= v(K(X)^{\times})$. The \emph{rational rank} of $v$ is defined as 
\[
\ratrank(v):= \dim_{\bQ}(\Phi_v\otimes_{\bZ}\bQ).
\]

The space of valuations on $X$ is denoted by $X^{\val}$. We equip $X^{\val}$ with the weakest topology such that $v\mapsto v(f)$ is continuous for every $f\in K(X)^{\times}$. We denote by $v_{\triv}$ the trivial valuation on $X$.
\end{defn}

\begin{defn}
A valuation $v$ on $X$ is called \emph{divisorial} (resp.\ $\bQ$-\emph{divisorial}, $\bZ$-\emph{divisorial}) if there exists a prime divisor $E$ on a normal variety $Y$ birational to $X$ and $c\in \bR_{\geq 0}$ (resp.\ $c\in \bQ_{\geq 0}$, $c\in \bZ_{\geq 0}$) such that $v = c\,\ord_E$. 

We denote the space of divisorial valuations on $X$ by $X^{\rm div}$. Let $X^{\rm div}(\bQ)$ (resp.\ $X^{\rm div}(\bZ)$) denote the space of $\bQ$-divisorial (resp.\ $\bZ$-divisorial) valuations.
\end{defn}

\begin{defn}
Let $Y$ be a normal variety birational to $X$. Let $\eta\in Y$ be a scheme-theoretic point such that $Y$ is regular at $\eta$. Let $y_1,\cdots, y_r\in \cO_{Y, \eta}$ be a system of parameters. Let $\alpha = (\alpha_1, \cdots, \alpha_r)\in \bR_{\geq 0}^r$.
By the Cohen structure theorem, there is an isomorphism
$\widehat{\mathcal O}_{Y,\eta}
\cong
\kappa(\eta)[[y_1,\cdots,y_r]]$,
where $\kappa(\eta)$ is identified with a coefficient field.
We define a valuation $v_\alpha$ on $X$ as follows: for any $f\in \cO_{Y,\eta}$, we may write $f = \sum_{\beta\in \bZ_{\geq 0}^r} c_{\beta} y^{\beta}$ with $c_{\beta}\in \kappa(\eta)$, and we define
\[
v_{\alpha}(f) := \min_{c_{\beta}\neq 0}\langle \alpha, \beta\rangle. 
\]
A valuation $v$ on $X$ is called \emph{quasi-monomial} if $v$ is of the form $v_{\alpha}$ for some $Y$, $\eta$, $y_i$ and $\alpha$ as above.

We denote the space of quasi-monomial valuations on $X$ by $X^{\qm}$.  From \cite[Remark 3.9]{JM12} we know that $X^{\rdiv}$ is the subspace of $X^{\qm}$ consisting of rational rank $\leq 1$ valuations. Moreover, the subspace of $X^{\qm}$ where each $\alpha_i\in \bQ_{\geq 0}$ (resp.\ each $\alpha_i\in \bZ_{\geq 0})$ is precisely $X^{\rdiv}(\bQ)$ (resp.\ $X^{\rdiv}(\bZ)$).
\end{defn}

\begin{defn}
Let $(X,D)$ be an lc pair. Let $(Y, \sum_{i=1}^{r} E_i)$ be a smooth variety with a simple normal crossing divisor together with a birational morphism $\mu: Y\to X$. Let $\eta$ be a generic point of $\cap_{i=1}^r E_i$. Let $y_i\in \cO_{Y,\eta}$ such that $E_i  = (y_i = 0)$. Let $v_{\alpha}$ be the quasi-monomial valuation at $\eta$ in the local coordinates $(y_i)$ of weight $\alpha$.  We define the \emph{log discrepancy} of $v_{\alpha}$ as
\[
A_{X,D}(v_{\alpha}):= \sum_{i=1}^r \alpha_i A_{X,D}(E_i).
\]
\end{defn}

For log discrepancy of general valuations, see \cite{JM12, BdFFU15}.

From the above definition, it is easy to see that a birational map $X\dashrightarrow Y$ induces isomorphisms on the spaces of valuations $X^{\val} \cong Y^{\val}$, $X^{\rm div} \cong Y^{\rm div}$ and $X^{\qm} \cong Y^{\qm}$.

\begin{defn}
Let $(X,D)$ be an lc pair.
A valuation $v\in X^{\val}$ is called an \emph{lc place} of $(X,D)$ if $v$ has a center in $X$ and $A_{X,D}(v) = 0$. Note that this condition implies  $v\in X^{\qm}$. We denote by $\LCP(X,D)$ the set of all lc places of $(X,D)$. 

Assume in addition that $X$ is projective.
Then we know that $\LCP(X,D)\subset X^{\qm}$ is a finite regular $\Delta$-cone complex with an integral PL structure. Denote by $\LCP^{\rdiv}(X,D):=\LCP(X,D)\cap X^{\rdiv}$, $\LCP(X,D)(\bZ):=\LCP(X,D)\cap X^{\rdiv}(\bZ)$, and $\LCP(X,D)(\bQ):=\LCP(X,D)\cap X^{\rdiv}(\bQ)$. 

The \emph{dual complex} $\cD(X,D)$ for a projective lc pair $(X,D)$ is defined to be the dual complex of a dlt modification of $(X,D)$. By \cite{dFKX17}, $\cD(X,D)$ is a finite regular $\Delta$-complex with a rational PL structure that is  independent of the choice of a dlt modification. Moreover, we have a natural identification
\[
(\LCP(X,D)\setminus \{v_{\triv}\})/\bR_{>0} = \cD(X,D),
\]
where the $\bR_{>0}$-action on $\LCP(X,D)$ is given by scaling of valuations. Denote by $\cD(X,D)(\bQ)$ the subset of $\cD(X,D)$ corresponding to the quotient of $\LCP^{\rdiv}(X,D)\setminus \{v_{\triv}\}$. Then $\cD(X,D)(\bQ)$ is precisely the set of rational points in $\cD(X,D)$ with respect to its rational PL structure.

By abuse of notation, we will often regard a nontrivial
valuation $v\in\LCP(X,D)$ as a point of $\cD(X,D)$ and write
$v\in\cD(X,D)$ for its image under the quotient map.
\end{defn}

\begin{defn}
Let $U$ be an affine log CY variety with a log CY compactification $(X,D)$. Then we define the \emph{space of tropical points of $U$} as 
\[
U^{\trop} := \LCP(X,D).
\]
Similarly, we define $U^{\trop}(\bZ):= U^{\trop} \cap U^{\rm div}(\bZ)$ and $U^{\trop}(\bQ):= U^{\trop} \cap U^{\rm div}(\bQ)$.
We also define the \emph{dual complex} of $U$ as
\[
\cD(U):= \cD(X,D) = (U^{\trop}\setminus \{v_{\triv}\})/\bR_{>0}. 
\]
\end{defn}

By Lemma \ref{lem:cbir}, the above definition is independent of the choice of log CY compactifications. Moreover, $U^{\trop}$ is a finite $\Delta$-cone complex with an integral PL structure.

\subsection{Families of pairs}

\begin{defn}
Let $T$ be a normal scheme.
A \emph{family of slc (resp.\ lc, klt) pairs} $f:(X,\Delta) \to T$ over $T$ consists of a flat proper surjective morphism $f: X\to T$ with connected $S_2$ fibers and an effective $\bQ$-divisor $\Delta$ on $X$ such that
\begin{enumerate}
    \item Every irreducible component of $\Delta$ is a relative Mumford divisor;
    \item $K_{X/T}+\Delta$ is $\bQ$-Cartier;
    \item $(X_t, \Delta_t)$ is slc (resp.\ lc, klt) for every $t\in T$.
\end{enumerate}

A family of slc pairs $f:(X,\Delta) \to T$ is called a \emph{family of log Fano pairs} (resp.\ \emph{family of weak log Fano pairs}, \emph{family of log CY pairs}) if in addition $-K_{X/T}-\Delta$ is ample (resp.\ nef and big, $\bQ$-linearly trivial) over $T$.

A \emph{family of log CY-Fano triples} $f:(X,D;\Delta)\to T$ consists of a family of log CY pairs $f:(X,D)\to T$ and a family of log Fano pairs $f:(X,\Delta)\to T$ such that $\Delta\leq D$.
\end{defn}

\begin{prop}\label{prop:slc-family}
Let $f: (X,\Delta) \to T$ be a family of slc pairs over a smooth scheme $T$. Let $V$ be an snc divisor on $T$. Then $(X, \Delta+f^*V)$ is slc. Moreover, if $(X\setminus f^{-1}(V),\Delta|_{X\setminus f^{-1}(V)})$ is klt, then $(X,\Delta)$ is klt.
\end{prop}

\begin{proof}
The first statement follows from \cite[Proof of Theorem 4.54]{Kol24} (see also \cite[Lemma 2.7]{BABWILD}). For the second statement, we first show that $X$ is normal hence $(X,\Delta+f^*V)$ is lc. If $X$ has a codimension $1$ singular locus $Z$, then we have $Z \subset f^{-1} (V)$ as $(X,\Delta)$ is klt away from $f^{-1}(V)$. But then $(X,\Delta + f^*V)$ would not be slc as $f^*V$ contains $Z$, a contradiction. Thus $X$ is demi-normal and regular in codimension $1$ which implies that $X$ is normal. Now, if $(X,\Delta)$ is not klt, then there exists a prime divisor $E$ over $X$ such that $A_{X,\Delta}(E) = 0$. Since $(X,\Delta)$ is klt away from $f^{-1}(V)$, we know that the center of $E$ on $X$ is contained in $f^{-1}(V)$, which implies that $\ord_E(f^*V)>0$. Thus
\[
A_{X,\Delta + f^*V}(E) = A_{X,\Delta}(E) - \ord_E (f^{*}V) <0,
\]
a contradiction. As a result, the pair $(X,\Delta)$ is klt.
\end{proof}

\begin{prop}\label{prop:cbir-CY-family}
Let $f:(X,\Delta) \to T$ be a family of log CY pairs over a smooth scheme $T$ with $X$ normal. Let $f':(X',\Delta')\to T$ be a proper morphism from a normal pair.  Let $\phi: (X,\Delta)\dashrightarrow (X',\Delta')$ be a crepant birational map over $T$. Then $f'$ is also a family of log CY pairs.
\end{prop}

\begin{proof}
    By smoothness of $T$, for any closed point $t\in T$ there exist an open neighborhood $S$ of $t\in T$ and an snc divisor $V=\sum_{i=1}^r V_i$ on $S$ such that $\cap_{i=1}^r V_i = t$. Denote by $f_S: (X_S,\Delta_S) \to S$, $f_S': (X_S',\Delta_S') \to S$ and $\phi_S: (X_S, \Delta_S) \dashrightarrow (X_S', \Delta_S')$ the base change of $f$, $f'$, and $\phi$ to $S$ respectively. Since $f_S$ is a family of slc pairs and $X_S$ is normal, by Proposition \ref{prop:slc-family} we know that $(X_S, \Delta_S + f_S^* V)$ is lc. Thus the crepant birationality of $\phi_S$ together with the normality of $X_S'$ implies that $(X_S', \Delta_S' + f_S'^* V)$ is lc. Thus \cite[Lemma 2.7]{BABWILD} implies that $f'$ is a family of slc pairs in a neighborhood of $t$, and hence a family of slc pairs over $T$ as $t$ can be chosen arbitrarily in $T$. Since $K_{X/T} + \Delta\sim_{\bQ, T} 0$, by the crepant birationality of $\phi$ we know that $K_{X'/T}+\Delta'\sim_{\bQ, T} 0$ and hence $f':(X',\Delta')\to T$ is a family of log CY pairs.  
\end{proof}

\subsection{Special valuations and $\bR$-test configurations}

Unless otherwise stated, throughout this subsection $(X,D;\Delta)$ denotes an lc-klt log CY-Fano triple over $\bk$. Let $L:= - l (K_X+\Delta)$ be an ample Cartier divisor for some $l\in \bN$. Denote the section ring of $(X,L)$ by  $R:= R(X,L) = \oplus_{m\in\bN} R_m$ with $R_m = H^0(X, mL)$.

\subsubsection{Finitely generated valuations and special valuations}
\begin{defn}
For $v\in X^{\val}$, we define the filtration $\cF_v R$ as 
\[
\cF^{\lambda}_v R_m := \{s\in R_m \mid v(s) \geq \lambda\} \quad \textrm{ for any }\lambda\in \bR\textrm{ and }m\in \bN.
\]
We also set \[
\cF^{>\lambda}_v R_m := \cup_{\lambda'>\lambda} \cF^{\lambda'}_v R_m \quad\textrm{ and } \quad \gr_v^{\lambda} R_m := \cF^{\lambda}_v R_m/\cF^{>\lambda}_v R_m.
\]
The \emph{associated graded ring} of $v$ is defined as 
\[
\gr_v R : = \bigoplus_{m\in \bN} \bigoplus_{\lambda\in \Phi_v} \gr_v^{\lambda} R_m.
\]

\end{defn}

\begin{defn}
A quasi-monomial valuation $v\in X^{\qm}$ is called 
\begin{enumerate}
    \item \emph{finitely generated} with respect to $(X,\Delta)$ if $\gr_v R$ is a finitely generated $\bk$-algebra;
    \item \emph{special} with respect to  $(X,\Delta)$ if it is finitely generated with respect to $(X,\Delta)$, and the induced log Fano degeneration $(X_v, \Delta_v)$ of $(X,\Delta)$ is klt where $X_v:=\Proj\, \gr_v R$ and $\Delta_v$ is the $\bQ$-divisor defined as in \cite[Paragraph below Theorem 4.2]{LXZ22}.
\end{enumerate}
\end{defn}

In this paper, we are mainly interested in finitely generated and special valuations with respect to $(X,\Delta)$ that are lc places of $(X,D)$. Note that by \cite[Appendix]{BLX22}, any finitely generated valuation $v$ inducing an slc log Fano degeneration of $(X,\Delta)$ (called a \emph{finitely generated weakly special valuation}) is an lc place of $(X,D)$ for some bounded complement $D$. 

\begin{defn}
Denote by 
\begin{align*}
\LCP^{\rfg}(X,D;\Delta)&:= \{v\in \LCP(X,D)\mid v\textrm{ is finitely generated with respect to }(X,\Delta)\};\\
\LCP^{\SV}(X,D;\Delta)&:= \{v\in \LCP(X,D)\mid v\textrm{ is special with respect to }(X,\Delta)\}.
\end{align*}
Similarly, denote the spaces of finitely generated valuations and special valuations in the dual complex $\cD(X,D)$ by
\begin{align*}
\cD^{\rfg}(X,D;\Delta)&:=(\LCP^{\rfg}(X,D;\Delta)\setminus \{v_{\triv}\})/\bR_{>0};\\ \cD^{\SV}(X,D;\Delta)&:=(\LCP^{\SV}(X,D;\Delta)\setminus \{v_{\triv}\})/\bR_{>0}.
\end{align*}
\end{defn}

\subsubsection{Test configurations}

We first define test configurations for an arbitrary log CY-Fano triple, allowing the underlying pair to be slc.

\begin{defn}
Let $(X,D;\Delta)$ be a log CY-Fano triple. 
Let  $\pi: \cX \to \bA^1$ be a flat proper morphism from a demi-normal scheme $\cX$ equipped with a $\bG_m$-action.
\begin{enumerate}
    \item We say that $\cX$ is a \emph{test configuration} of $(X,D;\Delta)$ if the following hold.
\begin{enumerate}[label=(\roman*)]
    \item $\pi$ is $\bG_m$-equivariant with  the standard $\bG_m$-action on $\bA^1$;
    \item $\cX\setminus \cX_0$ is $\bG_m$-equivariantly isomorphic to $X\times (\bA^1\setminus \{0\})$ with the trivial $\bG_m$-action on $X$;
    \item let $D_{\cX}$ and $\Delta_{\cX}$ be the closure of $D\times (\bA^1\setminus \{0\})$ and $\Delta\times (\bA^1\setminus \{0\})$ on $\cX$ respectively under the above isomorphism, then $(\cX, D_{\cX}+\cX_0)$ is slc and $K_{\cX}+D_{\cX}\sim_{\bQ,\pi} 0$;    
    \item $-K_{\cX} - \Delta_{\cX}$ is $\bQ$-Cartier and $\pi$-ample.
\end{enumerate}

\item We say that $\cX$ is a \emph{semiample test configuration} of $(X, D;\Delta)$ if it satisfies (i), (ii), (iii) and the following:
\begin{enumerate}
    \item[(iv')] $-K_{\cX}-\Delta_{\cX}$ is $\bQ$-Cartier and $\pi$-semiample. 
\end{enumerate}

\item A test configuration $\cX$ of $(X,D;\Delta)$ is called \emph{trivial} if there is a $\bG_m$-equivariant isomorphism $\cX \cong X\times \bA^1=: X_{\bA^1}$ where the $\bG_m$-action on $X$ is trivial.

\item Assume in addition that $(X,D;\Delta)$ is lc-klt. A test configuration $\cX$ of $(X,D;\Delta)$ is called \emph{special} if $(\cX, \Delta_{\cX}+\cX_0)$ is plt.
\end{enumerate}

\end{defn}

We note that in some literature (e.g.\ \cite{BLX22, BABWILD, BL24}), our notion of a test configuration $\cX$ of $(\cX,D;\Delta)$ refers to a \emph{weakly special test configuration} of the boundary polarized CY pair $(X,\Delta +(D-\Delta))$.

For the remainder of this subsection, we again assume that  $(X,D;\Delta)$ is an lc-klt log CY-Fano triple.

 Given a semiample test configuration $\tcX$ of $(X,D;\Delta)$, the ample model $\cX$ of $-K_{\tcX}-\Delta_{\tcX}$ is a test configuration as $\tcX\to \cX$ is a birational contraction which implies that $(\tcX, D_{\tcX})\to (\cX, D_{\cX})$ is crepant birational and hence $(\cX,D_{\cX}+ \cX_0)$ is lc and $K_{\cX}+ D_{\cX} \sim_{\bQ,\pi} 0$.

\begin{defn}
Let $\cX$ be a semiample test configuration of $(X,D;\Delta)$. Let $\cX_0=\cup_{i=1}^{d} \cX_0^{(i)}$ be the irreducible components. Denote by $v_{\cX,i}:=\ord_{\cX_0^{(i)}}|_{K(X)}$ where we identify $K(\cX) = K(X)(t)$ under the isomorphism $\cX\setminus \cX_0 \cong X\times (\bA^1\setminus\{0\})$. If in addition $\cX_0$ is integral, we denote by $v_{\cX}:= v_{\cX,1}$.
\end{defn}

\begin{thm}[cf.\ \cite{BLX22, BABWILD}]\label{thm:sa-tc}
Let $\tcX$ be a semiample test configuration of $(X,D;\Delta)$ where $\tcX_0$ has $d$ irreducible components. Then $\{v_{\tcX,i}\}_{i=1}^{d}\subset \LCP(X,D)(\bZ)$. 

Conversely, for any finite subset $\{v_1, \cdots, v_d\}\subset \LCP(X,D)(\bZ)$, there exists a semiample test configuration $\tcX$ of $(X,D;\Delta)$ where $\tcX_0$ has $d$ irreducible components such that $v_{\tcX,i} = v_{i}$.
Moreover, define  a filtration $\cF$ of $R$ by
\[
\cF^{\lambda} R_m : = \{s\in R_m \mid  v_{i}(s) \geq \lambda + lm A_{X,\Delta}(v_{i})\textrm{ for any }1\leq i \leq d\}.
\]
Denote the Rees algebra of $\cF$ by 
\[
\Rees(\cF):= \bigoplus_{m\in \bZ_{\geq 0}}\bigoplus_{\lambda\in \bZ} t^{-\lambda}\cF^\lambda R_m.
\]
Then $\Rees(\cF)$ is finitely generated as a $\bk[t]$-algebra, and $\cX:=\Proj \,\Rees(\cF)$ is an ample model of $\tcX$. 
\end{thm}

\begin{proof}
For the first statement, we consider a semiample test configuration $\tcX$. Let $\cX$ be an ample model of $\tcX$. Then we know that $(\tcX, D_{\tcX}+\tcX_0)$ is crepant birational to $(\cX, D_{\cX}+\cX_0)$  over $\bA^1$, which is crepant birational to $(X_{\bA^1}, D_{X_{\bA^1}} + X\times \{0\})$ over $\bA^1$ by \cite[Lemma 2.10(1)]{BABWILD}. Thus every irreducible component $\tcX_0^{(i)}$ is an lc place of $(X_{\bA^1}, D_{X_{\bA^1}} + X\times \{0\})$. Since $\ord_{\tcX_0^{(i)}}$ is the Gauss extension of $v_{\tcX,i}$, we know that $v_{\tcX,i}$ is a $\bZ$-divisorial lc place of $(X,D)$. 

For the remaining statement, by \cite[Proposition 4.9]{BABWILD} we know that the Rees algebra $\Rees(\cF)$ is finitely generated, and $\cX:=\Proj\, \Rees(\cF)$ is a test configuration of $(X,D;\Delta)$. Moreover, there exists a subset $\emptyset\neq I\subset \{1,\cdots, d\}$ such that $\cX_0 = \cup_{i\in I} \cX_0^{(i)}$ where $v_{\cX, i} = v_i$. By  \cite[Last paragraph of Proof of Proposition A.5]{BLX22}, for each $j\not\in I$ the Gauss extension of $v_{j}$ is an lc place of $(\cX, \Delta_{\cX}+ \cX_0)$. By \cite{BCHM10}, there exists a proper birational morphism $\tcX\to \cX$ extracting prime divisors $\tcX_0^{(j)}$ for all $j \not\in I$ such that $v_{\tcX, j}= v_j$ and $\tcX$ is a semiample test configuration of $(X,D;\Delta)$. The proof is finished.
\end{proof}

\begin{cor}[cf.\ \cite{CZ22, BABWILD}]
There is a one-to-one correspondence between test configurations $\cX$ of $(X,D;\Delta)$ with integral $\cX_0$ and $\LCP(X,D)(\bZ)$ where $\cX\mapsto v_{\cX}$. 
\end{cor}

\begin{proof}
This is a simple consequence of Theorem \ref{thm:sa-tc} as a test configuration $\cX$ with integral $\cX_0$ is uniquely determined by $\ord_{\cX_0}$. See also \cite[Theorem 4.8]{BABWILD}.
\end{proof}

\begin{cor}
A $\bZ$-divisorial valuation $v\in \LCP(X,D)(\bZ)$ is special if and only if there exists a special test configuration $\cX$ of $(X,D;\Delta)$ such that $v = v_{\cX}$. 
\end{cor}

\begin{proof}
By Theorem \ref{thm:sa-tc}, we know that the test configuration $\cX$ with $v_{\cX} = v$ has central fiber $\cX_0 \cong \Proj \, \gr_v R$. Hence the statement follows from (inversion of) adjunction.
\end{proof}

\begin{prop}\label{prop:div-fg}
We have 
\[
\LCP^{\rdiv}(X,D)\cup \LCP^{\SV}(X,D;\Delta) \subset \LCP^{\rfg}(X,D;\Delta) \subset \LCP(X,D).
\]
In particular,
\[
\cD(X,D)(\bQ)\cup \cD^{\SV}(X,D;\Delta) \subset \cD^{\rfg}(X,D;\Delta) \subset \cD(X,D).
\]
\end{prop}

\begin{proof}
The containment $\LCP^{\SV}(X,D;\Delta) \subset \LCP^{\rfg}(X,D;\Delta) \subset \LCP(X,D)$ is clear from definitions. The containment $\LCP^{\rdiv}(X,D) \subset \LCP^{\rfg}(X,D;\Delta)$ follows from Theorem \ref{thm:sa-tc} and the fact that the central fiber $\cX_0$ of a test configuration $\cX$ with $v_{\cX}= v\in \LCP(X,D)(\bZ)$ satisfies $\cX_0 \cong \Proj\,\gr_v R$.
\end{proof}

\subsubsection{Multi-degenerations and $\bR$-test configurations}

\begin{defn}\label{def:multideg}
Let $(X,D;\Delta)$ be an lc-klt log CY-Fano triple. Let $\pi: \fX \to \bA^r$ be a flat proper morphism from a normal variety $\fX$ equipped with a $\bT=\bG_m^r$-action. 

\begin{enumerate}
    \item We say that $\fX$ is a \emph{multi-degeneration} of $(X,D;\Delta)$ if the following hold.
    \begin{enumerate}[label=(\roman*)]
        \item $\pi$ is $\bT$-equivariant with the standard $\bT$-action on $\bA^r$;
        \item $\fX\times_{\bA^r} \bT$ is $\bT$-equivariantly isomorphic to $X\times \bT$ with the trivial $\bT$-action on $X$;
        \item let $\fD$ and $\Delta_{\fX}$ be the closure of $D\times \bT$ and $\Delta\times \bT$ on $\fX$ respectively under the above isomorphism, then $\pi:(\fX,\fD;\Delta_{\fX})\to \bA^r$ is a family of log CY-Fano triples.
    \end{enumerate}
    \item Let $\xi \in \bR_{>0}^r$ where we identify $\bR^{r} = \Hom(\bG_m, \bT)\times_{\bZ} \bR$. Then we call $(\fX, \xi)$ an \emph{$\bR$-test configuration} of $(X,D;\Delta)$. 
    \item A multi-degeneration $\fX$ is called \emph{special} (resp.\ \emph{integral}) if the log Fano central fiber $(\fX_0, \Delta_{\fX_0})$ is klt (resp.\ integral). An $\bR$-test configuration $(\fX,\xi)$ is called \emph{special} (resp.\ \emph{integral}) if $\fX$ is a special (resp.\ integral) multi-degeneration.
\end{enumerate}

\end{defn}

\begin{thm}[cf.\ {\cite[Theorem 1.3]{Xu21}}]\label{thm:multideg}
Let $E_1,\cdots, E_r$ be prime divisors over $X$ such that $\ord_{E_i}\in \LCP(X,D)$ for every $1\leq i\leq r$. Then there exists a multi-degeneration $\fX$ of $(X,D;\Delta)$ such that if we restrict to the $r$ axes through $(1,\cdots, 1)\in \bA^r$, we obtain    test configurations $(\cX_i)_{i=1}^r$ with $v_{\cX_i} = \ord_{E_i}$. We call $\fX$ the \emph{multi-degeneration induced by} $\{E_i\}_{i=1}^r$.
\end{thm}

\begin{thm}[cf.\ {\cite[Theorem 1.4]{Xu21}}]\label{thm:fg-multideg}
Let $v\in \LCP(X,D)$ be a valuation of rational rank $r$. Then $v$ is finitely generated (resp.\ special) with respect to $(X,\Delta)$ if and only if there exists a birational morphism $Y\to X$ extracting prime divisors $E_1,\cdots, E_r$ such that $(Y, E=E_1+\cdots+E_r)$ is qdlt,  $v$ is toroidal over $(\eta\in Y,  E_1+\cdots+E_r)$ where $\eta$ is a generic point of $\cap_{i=1}^r E_i$, and the multi-degeneration $\fX$ induced by $\{E_i\}_{i=1}^r$ is integral (resp.\ klt). 
\end{thm}

\begin{rem}\label{rem:fg-multideg}
    
    In Theorem \ref{thm:fg-multideg}, if the divisors $E_i$ induce a multi-degeneration $\fX$ that is integral or special, then the same property holds for any $\{E_i'\}_{i=1}^r$ in $\QM_\eta(Y, E)(\bZ)$ such that the simplicial cone generated by $\{E_i'\}_{i=1}^r$ contains $v$. This can be seen by a toric base change of the multi-degeneration.
\end{rem}

\begin{defn}\label{def:geodesic}
Let $\pi:\fX\to \bA^r$ be an integral (resp.\ special) multi-degeneration of $(X,D;\Delta)$. Denote by $V_i:=(t_i=0)\subset \bA^r$ and $\fX_{V_i}:=\pi^*V_i$ where $(t_1,\cdots,t_r)$ is the affine coordinate of $\bA^r$. In particular, the pair $(\fX,\fX_{V_1}+\cdots + \fX_{V_r})$ is snc at the generic point of $\fX_0$. For each $\xi\in \bR_{\geq 0}^r$ we associate a valuation $v_{\fX,\xi}\in X^{\val}$ in the following way. Since we are given a $\bT$-equivariant isomorphism between $\fX\times_{\bA^r} \bT$ and $X\times\bT$, there is a canonical field extension $K(X) \hookrightarrow K(\fX)$. Then we define $v_{\fX,\xi}:= \wt_{\xi}|_{K(X)}$ where  $\wt_{\xi}$ is the monomial valuation in the local coordinates determined by  $\fX_{V_1},\cdots, \fX_{V_r}$  of weight $\xi$. By \cite{BLX22, Xu21} we know that every $v_{\fX,\xi}$ is in $\LCP(X,D)$. We call the subset
\[
\{v_{\fX,\xi}\mid \xi\in \bR_{\geq 0}^r \}\subset \LCP(X,D)
\]
a \emph{finitely generated (resp.\ special) geodesic $r$-simplicial cone}, and its projectivization in $\cD(X,D)$ is called
a \emph{finitely generated (resp.\ special) geodesic $(r-1)$-simplex}. A geodesic $1$-simplex is called a \emph{geodesic line}, while a geodesic $2$-simplex is called a \emph{geodesic triangle}.
\end{defn}

We note that our notion of geodesic simplices comes from a more general notion of geodesics of filtrations; see \cite[Section 3.1.2]{BLXZ23}.

\begin{prop}[{cf.~\cite[Lemma 4.8]{XZ22}}]\label{prop:geodesic-simplex}
Let $(Y, D_Y) \to (X,D)$ be a dlt modification. If a finitely generated geodesic simplex $P$ satisfies that all its vertices belong to the same cell $P'$ in $\cD(Y, D_Y)$, then $P\subset P'$ is a rational subsimplex.
\end{prop}

\subsection{Valuations of maximal rational rank}\label{sec:max-rank}

In this subsection, we show the equivalence between specialness and finite generation for valuations of maximal rational rank.

\begin{thm}\label{thm:toric}
Let $(X,D;\Delta)$ be an lc-klt log CY-Fano triple of dimension $n$ such that $D\geq (1+\epsilon)\Delta$ for some $\epsilon>0$. Let $v\in \LCP(X,D)$ be a valuation of rational rank $n$. Then $v$ is special with respect to $(X,\Delta)$ if and only if it is  finitely generated with respect to $(X,\Delta)$.
\end{thm}

\begin{proof}
The ``only if'' part is trivial, so we may focus on the ``if'' part. Assume that $v\in \cD^{\rfg}(X,D;\Delta)$. By Theorem \ref{thm:fg-multideg}, there exists a multi-degeneration $\fX\to \bA^n$ of $(X,D;\Delta)$ such that $\fX_0$ is integral where $v$ corresponds to $\wt_{\xi}$ for some $\xi\in \bR_{>0}^n$. Since $v$ has maximal rational rank, we know that the log CY-Fano triple $(\fX_0, \fD_0;\Delta_{\fX_0})$ carries an effective $\bG_m^n$-action. Then by Lemma \ref{lem:toric} we know that $(\fX_0, \fD_0)$
 is a normal toric pair. Moreover, since $D\geq (1+\epsilon)\Delta$ we know that $\fD\geq (1+\epsilon)\Delta_{\fX}$ which implies $\fD_0\geq (1+\epsilon)\Delta_{\fX_0}$. In particular, we have that $(\fX_0, \Delta_{\fX_0})$ is klt. This implies $v\in \cD^{\SV}(X,D;\Delta)$ by Theorem \ref{thm:fg-multideg}.
\end{proof}

\begin{lem}\label{lem:toric}
Let $(X,D;\Delta)$ be an irreducible log CY-Fano triple of dimension $n$ with an effective $\bG_m^n$-action. Then $X$ is a normal projective toric variety and $D$ is the reduced toric boundary.
\end{lem}

\begin{proof}
The statement is clear if $X$ is normal. 
Assume to the contrary that $X$ is not normal.
Let $\nu: (X', G'+ D')\to (X,D)$ be the normalization map where $G'$ is the conductor divisor on $X'$ and $D':= \nu_*^{-1} D$. Let $\Delta':=\nu_*^{-1}\Delta$. Then we have 
\[
K_{X'}+ G' + D' = \nu^*(K_X+D)\sim_{\bQ} 0, \quad -K_{X'}- G' - \Delta' = \nu^*(-K_X-\Delta)\sim_{\bQ} D'-\Delta'
\]
Thus $(X', G'+D')$ is a normal toric pair, i.e.\ $X'$ is an irreducible normal projective toric variety and $G'+D'$ is the reduced toric boundary divisor. Moreover, the $\bQ$-Cartier $\bQ$-divisor $D'-\Delta'$ is ample.

Next, let $\Sigma$ denote the toric fan of $X'$ supported on $N_{\bR} = \bR^n$. Denote by $\rho_1,\cdots, \rho_m$  the rays of $\Sigma$. Let $v_{i}\in N = \bZ^n$ be the primitive vector on the ray $\rho_i$. Let $D_{i}$ be the torus invariant prime divisor on $X'$ corresponding to the ray $\rho_i$. Since $(X,D)$ is slc, after permuting the indices of $\{\rho_i\}$ we may write $G'= \sum_{i=1}^{2k} D_i$ and $D'=\sum_{i=2k+1}^{m} D_i$ with $1\leq k\leq \lfloor\frac{m}{2}\rfloor$ such that $D_{2j-1}$ and $D_{2j}$ are identified under the normalization map $\nu$ for every $1\leq j\leq k$. Then we can write 
\[
D'-\Delta' = \sum_{i=2k+1}^m c_i D_i, \quad \textrm{ where } c_i \in [0,1]. 
\]
By \cite[P. 70, Proposition]{Ful93}, the ampleness of $D'-\Delta'$ implies that there exists a strictly convex piecewise-linear function $\psi_{D'-\Delta'}:|\Sigma|=\bR^n\to \bR$ such that $\psi_{D'-\Delta'}(v_i) = 0$ for $1\leq i \leq 2k$ and $\psi_{D'-\Delta'}(v_i) = c_i$ for $2k+1\leq i\leq m$. Since the torus action on $X'$ descends to $X$, we know that the stabilizer groups of $D_{2j-1}$ and $D_{2j}$ are the same for $1\leq j\leq k$, which implies that $\rho_{2j-1}$ and $\rho_{2j}$ are antipodal, i.e.\ $v_{2j-1} = -v_{2j}$. In particular, $\psi_{D'-\Delta'}(v_1) = \psi_{D'-\Delta'}(v_2) = 0$ while $v_1+v_2=0$. This contradicts the strict convexity of $\psi_{D'-\Delta'}$. The proof is finished.
\end{proof}

\subsection{Degeneration to normal cone and geodesic simplices}\label{sec:geodesic-normal-cone}

In this subsection, we provide a criterion for the existence of finitely generated geodesic simplices where one of the vertices comes from a divisor on $X$ proportional to $-K_X-\Delta$.

\begin{prop}\label{prop:deg-normal-cone}
Let $(X,D;\Delta)$ be an lc-klt log CY-Fano triple.  Let $v_1,\cdots, v_r\in \LCP(X,D)(\bZ)$ be divisorial lc places that form a finitely generated geodesic $(r-1)$-simplex in $\cD(X,D)$. Denote by $(\cX, D_{\cX};\Delta_{\cX})\to \bA^r$ the multi-degeneration of $(X,D;\Delta)$ induced by $v_1,\cdots, v_r$. Let $S\subset \lfloor D \rfloor$ be a prime divisor on $X$ such that $S\sim_{\bQ} \rho (-K_X-\Delta)$ for some $\rho\in \bQ_{>0}$.
Let $\cS$ be the closure of $S\times \bG_m^r$ in $\cX$. 
 Then $v_1,\cdots, v_r, \ord_S$ form a finitely generated geodesic $r$-simplex in $\cD(X,D)$ if and only if $\cS_0$ is irreducible in $\cX_0$.  
\end{prop}

\begin{proof}
First of all, we show that $\cS \sim_{\bQ} \rho(-K_{\cX} - \Delta_{\cX})$, in particular $\cS$ is $\bQ$-Cartier. It is clear that this holds over $\bT:= \bG_m^r$. Since $\pi:\cX\to \bA^r$ comes from a finitely generated geodesic simplex, we know that $\pi$ has geometrically integral fibers. As a result, we know that $\cS - \rho(-K_{\cX}-\Delta_{\cX})$ is $\bQ$-linearly equivalent to a $\bQ$-divisor supported in $\pi^* V$ where $V= V_1+\cdots + V_r$ and $V_i:= (t_i=0)\subset \bA^r$. Since each fiber $\cX_t$ is integral, we know that $\pi^* V_i$ is a prime principal divisor. Thus $\cS - \rho(-K_{\cX}-\Delta_{\cX})$ is $\bQ$-linearly equivalent to a $\bQ$-linear combination of $\pi^* V_i$, which is $\bQ$-linearly trivial. 

Next, choose $l\in \bZ_{>0}$ sufficiently divisible such that $l(-K_{\cX}-\Delta_{\cX}) \sim \rho^{-1} l \cS$ is   Cartier and $\pi$-ample, in particular $\varrho:=\rho^{-1} l\in \bZ_{>0}$. Let $L:= l(-K_X-\Delta)$ and $R=\oplus_{m\in \bN} R_m := R(X,L)$. Let $\cF_i$ be filtrations on $R$ induced by $v_i$, i.e. 
\[
\cF_i^\lambda R_m := \{s\in R_m\mid v_i(s)\geq \lambda + lm A_{X,\Delta}(v_i)\}.
\]
Then by \cite[Proposition 3.17]{BL26} we have that 
\[
\cX \cong \Proj\, \Rees(\cF_1, \cdots, \cF_r),
\]
where
\[
\Rees(\cF_1, \cdots, \cF_r):= \bigoplus_{m\in\bN}\bigoplus_{\lambda\in \bZ^r} \left(\bigcap_{i=1}^r \cF_i^{\lambda_i} R_m \right) t_1^{-\lambda_1}\cdots t_r^{-\lambda_r}.
\]
Let $\cF_S$ be the filtration on $R$ induced by $\ord_S$, i.e.
\[
\cF_S^\lambda R_m := \{s\in R_m\mid \ord_S(s)\geq \lambda + lm A_{X,\Delta}(\ord_S)\}.
\]
Then by Theorem \ref{thm:multideg} and \cite[Proposition 3.17]{BL26} we have a multi-degeneration $\varpi: (\fX,\fD;\Delta_{\fX}) \to \bA^{r+1}$ of $(X,D;\Delta)$ induced by $v_1,\cdots, v_r, \ord_S$ such that 
\[
\fX \cong \Proj_{\bA^{r+1}}\Rees(\cF_1, \cdots, \cF_r, \cF_S).
\]

Next, let $\cE_m:= \pi_* \cO_{\cX}(m\cL)$ where $\cL:= l(-K_{\cX}-\Delta_{\cX})$. Let $a:=lA_{X,\Delta}(\ord_S) \in \bZ_{>0}$. Define a filtration $\cG$ on $\cE_m$ by
\[
\cG^b \cE_m:= \begin{cases} \pi_* \cO_{\cX}(m\cL - (b+ma)\cS) & \textrm{ if }b\in \bZ_{\geq -ma},\\
\cE_m & \textrm{ if }b\in \bZ_{< -ma}.
\end{cases}
\]
We claim that $\fX \cong \Proj_{\bA^r} \Rees(\cG)$ where 
\[
\Rees(\cG) := \bigoplus_{m\in \bN}\bigoplus_{b\in \bZ} \cG^b \cE_m~ t_{r+1}^{-b}.
\]
Indeed, by setting $b=\lambda_{r+1}$ and $\fL:=l(-K_{\fX}-\Delta_{\fX})$, we have 
\[
H^0(\fX, m\fL) \cong \bigoplus_{b\in \bZ} H^0(\bA^r,\cE_m) \cap (\cF_S^b R_m \otimes \bk[t_1^{\pm}, \cdots, t_r^{\pm}])~ t_{r+1}^{-b}.
\]
For any $s\in R_m$ with $s t_1^{-\lambda_1}\cdots t_r^{-\lambda_r} \in H^0(\bA^r, \cE_m)$, by definition we have that $s\in \cF_S^b R_m$ if and only if $\ord_S(s) \geq b+ma$, which is equivalent to $s t_1^{-\lambda_1}\cdots t_r^{-\lambda_r} \in H^0(\bA^r, \cG^b \cE_m)$. Thus we have 
\[
H^0(\fX, m\fL)\cong \bigoplus_{b\in\bZ} H^0(\bA^r, \cG^b \cE_m) t_{r+1}^{-b},
\]
which implies the claim.

Next, we show that 
\begin{equation}\label{eq:filt-normal-cone}
    \cG^b \cE_m \otimes \kappa (0) \cong H^0(\cX_0, m\cL_0 - (b+ma) \cS_0).
\end{equation}
We may assume that $ml -(b+ma)\rho \geq 0$ as otherwise both sides vanish.  Since $(\cX,D_{\cX}) \to \bA^r$ is a family of slc pairs, we know that $\cS$ is a $\bQ$-Cartier relative Mumford divisor. Hence by \cite[Corollary 4.33]{Kol-modbook} we know that $\cO_{\cX}(m\cL - (b+ma) \cS)$ is flat over $\bA^r$ and 
\[
\cO_{\cX}(m\cL - (b+ma) \cS) \otimes \kappa(0)\cong \cO_{\cX_0}(m\cL_0 - (b+ma) \cS_0).
\]
Next by Fujino's generalization of Kawamata--Viehweg vanishing theorem \cite[Theorem 1.7]{Fuj14} we have $R^i \pi_* \cO_{\cX}(m\cL - (b+ma) \cS) = 0$ which implies \eqref{eq:filt-normal-cone} by Grauert's theorem. 

Finally, we show the main statement. By Theorem \ref{thm:multideg} and Definition \ref{def:geodesic}, we know that $v_1,\cdots, v_r, \ord_S$ form a finitely generated geodesic $r$-simplex in $\cD(X,D)$ if and only if the central fiber $\fX_{(0,0)}$ is irreducible. Since $\fX \cong \Proj_{\bA^r} \Rees(\cG)$, we know that $\fX_{(0,0)}$ is the central fiber of the test configuration $\fX\times_{\bA^{r+1}} \bA^1_{t_{r+1}}$. Hence by \eqref{eq:filt-normal-cone} we have
\begin{align*}
\fX_{(0,0)} &  \cong \Proj \bigoplus_{m\in \bN} \bigoplus_{b\in \bZ} \cG^b\cE_m\otimes \kappa(0)/\cG^{b+1}\cE_m\otimes \kappa(0)\\
& \cong \Proj \bigoplus_{m\in \bN} \bigoplus_{i\in \bN} H^0(\cX_0, m\cL_0 - i \cS_0)/H^0(\cX_0, m\cL_0 - (i+1) \cS_0).
\end{align*}
Thus we know that $\fX_{(0,0)}$ carries a $\bG_m$-action where each term indexed by $i$ has weight $i$. Since $\varrho \cS_0$ is Cartier, 
by Fujino's vanishing theorem as above for $i\in \varrho\bN$ we have 
\[
H^0(\cX_0, m\cL_0 - i \cS_0)/H^0(\cX_0, m\cL_0 - (i+1) \cS_0) \cong H^0(\cS_0, (m\cL_0-i\cS_0)|_{\cS_0}).
\]
Thus $\fX_{(0,0)}/\bmu_{\varrho} $ is isomorphic to the projective cone over $\cS_0$ with polarization $\varrho \cS_0|_{\cS_0}$. Therefore, $\fX_{(0,0)}$ is irreducible if and only if    $\fX_{(0,0)}/\bmu_{\varrho} $ is irreducible as $\bmu_{\varrho}\leq \bG_m$ and $\bG_m$ is connected, which is equivalent to $\cS_0$ being irreducible. This finishes the proof.
\end{proof}

\subsection{Analytifications, tropicalizations, and essential skeleta}\label{sec:an-trop}

In this subsection, we recall some non-Archimedean language which will only be used in Section \ref{sec:Markov} to relate the dual complex to tropical Markov dynamics.

Let $K = \bk(\!(t)\!)$ be the formal Laurent power series field over $\bk$. Then $K$ is a complete non-Archimedean field equipped with the standard norm $|\cdot |_K$ where $|f|_K= e^{-v_K(f)}$ where $v_K$ is the discrete valuation on $K$ defined as $v_K(f) := \ord_t(f)$. 
Let $U = \Spec\, A$ be an integral affine scheme of finite type over $K$. Then we can associate the Berkovich analytification $U^{\an}$ of $U$ in the sense of \cite{Ber90}. As a topological space, $U^{\an}$ consists of multiplicative seminorms $\|\cdot\|$ of $A$ extending $|\cdot|_K$ equipped with the weakest topology such that $\|\cdot\|\mapsto \|f\|$ is continuous for every $f\in A$. Let $U^{\val}\subset U^{\an}$ be the subset of multiplicative seminorms $\|\cdot \|_{v}:= e^{-v(\cdot)}$ where $v$ is a valuation on $\mathrm{Frac}(A)$ extending $v_K$. We often use the valuation notation for $U^{\val}$. Note that both definitions for the Berkovich analytification $X^{\an}$ and its subset $X^{\val}$ are extended to quasi-projective varieties $X$ over $K$.

Next, we fix a closed immersion $\varphi: U\hookrightarrow \bA_K^N$ into an affine space with coordinates $(x_1, \cdots, x_N)$. Following \cite{Gub13}, we define the \emph{tropicalization map} as 
\[
\Trop_{\varphi}: U^{\an} \to (\bR\cup\{\infty\})^N,\qquad \|\cdot\| \mapsto (-\log \|x_1\circ \varphi\|, \cdots, -\log \|x_N\circ \varphi\|). 
\]
We know that $\Trop_{\varphi}$ is continuous.
The image $\Trop_\varphi(U^{\an})$, denoted by $\Trop_\varphi(U)$ for simplicity, is called the \emph{tropicalization} of $U$ with respect to $\varphi$. It is the support of a polyhedral complex in $(\bR\cup\{\infty\})^N$.
We often omit the subscript $\varphi$ in $\Trop_\varphi$ when it is clear from context. 

Let $\oK: = \cup_{m\in \bZ_{>0}} \bk(\!(t^{1/m})\!)$ be the field of Puiseux series over $\bk$ equipped with the standard norm $|\cdot|_{\oK}$ as an  extension of $|\cdot|_K$. Then we know that $\oK$ is an algebraic closure of $K$. For a $\oK$-point $p\in U(\oK)$, 
we can associate a multiplicative seminorm $\|\cdot\|_p\in U^{\an}$ such that 
$\|f\|_p :=|f(p)|_{\oK}$.
In particular, there is a map $U(\oK) \to U^{\an}$ by $p\mapsto \|\cdot\|_p$ whose image is dense by \cite[2.6]{Gub13}. If we write
$\varphi(p) = (p_1,\cdots, p_N)$ where each $p_i\in \oK$, then it is clear that 
\[
\Trop_{\varphi}(\|\cdot\|_p) = (-\log |p_1|_{\oK},\cdots, -\log|p_N|_{\oK}).
\]
Moreover, by \cite[Proposition 3.8]{Gub13} we have that 
\[
\Trop_{\varphi}(U) = \overline{\{\Trop_{\varphi}(\|\cdot\|_p)\mid p\in U(\oK)\}}.
\]

Next, assume that $U$ is an affine log CY $K$-variety with a log CY compactification $(X,D)$. According to \cite{BM19} generalizing \cite{MN15, NX16} to the pair case which originated from \cite{KS06}, one has the \emph{essential skeleton} $\Sk(X,D)$ of $(X,D)$ as a (possibly open) polyhedral subcomplex of $X^{\val}$. Since the essential skeleton is invariant under crepant birational maps,  Lemma \ref{lem:cbir} implies that $\Sk(X,D)$ is independent of the choice of compactifications and thus we can write $\Sk(U):=\Sk(X,D) \subset X^{\val} = U^{\val}$ and call it the \emph{essential skeleton} of $U$.

Suppose $0\in T$ is a smooth pointed curve over $\bk$. Let $\pi:\sX\to T$ be a  flat projective morphism from a normal variety $\sX$ over $\bk$. Let $\sD$ be a reduced divisor on $\sX$ not containing any irreducible component of fibers of $\pi$. We fix an isomorphism $\widehat{\cO_{T,0}} \cong \bk\llbracket t\rrbracket$ over $\bk$ which induces a morphism $\Spec\,K \to T$. Then we say that $(\sX,\sD)/(0\in T)$ is a \emph{log CY model} of $(X,D)$ if $(\sX, \sD + \sX_{0,\red})$ is lc, $K_{\sX}+\sD+\sX_{0, \red} \sim_{\bQ, T} 0$, and $(\sX,\sD)\times_T \Spec\, K \cong (X,D)$. 

\begin{prop}[{cf.\ \cite[Proposition 5.1.7]{BM19}}]\label{prop:Sk-LCP}
Under the above assumptions, we have a PL homeomorphism from the space
$ \{v\in \LCP(\sX, \sD+ \sX_{0,\red})\mid v(t) = 1\}$ to $\Sk(U)$ induced by the restriction of valuations from $X^{\val}$ to $\{v\in \sX^{\val}\mid v(t)= 1\}$. 
\end{prop}


\section{Intrinsic properties of special and finitely generated valuations}\label{sec:intrinsic}

In this section, we establish two intrinsic properties of special and finitely
generated valuations. First, we prove their compactification-independence,
and hence Theorem \ref{thm:main} and Corollary \ref{cor:aut}. We then extend these results to affine log CY pairs. Finally, we introduce a natural partial order induced by
regular functions and characterize special valuations among finitely generated
valuations as its maximal elements, hence proving Theorem \ref{thm:maximality-intro}.

Throughout this section, unless otherwise stated, let $(X,D;\Delta)$  be an lc-klt log CY-Fano triple. Denote by $B:=D-\Delta$ the polarizing boundary and set $U:= X\setminus \Supp(B)$. Note that $B\sim_{\bQ}-(K_X+\Delta)$ is ample, and hence $U$ is affine.
For a second lc-klt log CY-Fano triple $(X',D';\Delta')$, we use the analogous notation for $B':=D'-\Delta'$ and $U':= X'\setminus \Supp(B')$.

\subsection{Semiample test configurations}

Let $L= -l(K_X+\Delta)$ be an ample Cartier divisor for $l\in \bZ_{>0}$ and $R:= R(X,L)$. Let $\{v_1,\cdots, v_d\}\subset \LCP(X,D)(\bZ)$ be a finite set of $\bZ$-divisorial lc places of $(X,D)$. Then Theorem \ref{thm:sa-tc} implies that there exists a semiample test configuration $\tcX$ of $(X,D;\Delta)$ such that $\tcX_0$ has $d$ irreducible components $\{\tcX_0^{(i)}\}_{i=1}^d$ and $v_i = \ord_{\tcX_0^{(i)}}|_{K(X)}$. Let $\cX$ be the test configuration obtained as an ample model of $\tcX$. 

\begin{prop}\label{prop:sa-tc-integral}
Notation as above. Then $\cX_0$ is irreducible and is the birational transform of $\tcX_0^{(1)}$ if and only if we have $v_1(f) \leq v_i(f)$ for any $2\leq i\leq d$ and any $f\in \cO_U(U)$.
\end{prop}

\begin{proof}
By Theorem \ref{thm:sa-tc}, we know that $\cX\cong \Proj \,\Rees(\cF)$ where 
\[
\cF^{\lambda} R_m = \{s\in R_m \mid v_i(s) \geq \lambda + lm A_{X,\Delta}(v_i)\textrm{ for any }1\leq i \leq d\}.
\]
Thus $\cX_0$ is the birational transform of $\cX_0^{(1)}$ if and only if 
\[
\cF^{\lambda} R_m =  \{s\in R_m \mid v_1(s) \geq \lambda + lm A_{X,\Delta}(v_1)\}.
\]
This happens if and only if 
\begin{equation}\label{eq:sa-tc}
    v_i(s)-lm A_{X,\Delta}(v_i) \geq v_1(s) - lm A_{X,\Delta}(v_1)\quad \textrm{ for any }s\in R_m\textrm{ and any }2\leq i\leq d.
\end{equation}
Since \eqref{eq:sa-tc} is preserved after replacing $s$ by a power of $s$, we may assume that $m$ is sufficiently divisible. In particular, we may assume that $m$ is divisible by $m_0\in \bN$ such that $lm_0(B) = \rdiv(s_0)$ for some $s_0\in R_{m_0}$. Now every $s\in R_m$ can be written as $s = f\cdot s_0^{m/m_0}$ for some $f\in \cO_U(U)$, and every such $f$ comes from some $s\in R_m$ for $m$ sufficiently large. Since each $v_i$ is an lc place of $(X,D)$, we know that $A_{X,\Delta}(v_i) = v_i (B) = (lm_0)^{-1} v_i(s_0)$. Thus \eqref{eq:sa-tc} is equivalent to 
\[
v_i(f\cdot s_0^{m/m_0}) - lm A_{X,\Delta}(v_i) \geq v_1 (f\cdot s_0^{m/m_0}) - lm A_{X,\Delta} (v_1),
\]
which is then equivalent to 
\[
v_i(f) \geq v_1 (f) \quad\textrm{ for any }f\in \cO_U(U)\textrm{ and any }2\leq i \leq d.
\]
The proof is finished.
\end{proof}

\subsection{Compactification-independence of finitely generated valuations}

\begin{thm}\label{thm:fg}
Let $g: (X,D) \dashrightarrow (X',D')$ be a crepant birational map such that $g|_U: U\to U'$ is an isomorphism. Then $g_*$ induces a bijection
\[
g_*:\LCP^{\mathrm{fg}}(X,D;\Delta)
\xrightarrow{\sim}
\LCP^{\mathrm{fg}}(X',D';\Delta').
\]
\end{thm}

To prove Theorem \ref{thm:fg}, we introduce a class of crepant birational maps for which lc places of the associated degenerations can be compared.

\begin{defn}
 A birational map $g:(X,D;\Delta)\dashrightarrow (X',D';\Delta')$ is called a \emph{crepant birational log-contraction} if $g:(X,D)\dashrightarrow (X',D')$ is crepant birational and 
 its inverse $g^{-1}$ extends to a birational morphism $g^{-1}|_{U'}: U'\to U$. 
\end{defn}

The following lemma gives some criteria for crepant birational log-contractions.

\begin{lem}\label{lem:log-cont}
Let $g: X\dashrightarrow X'$ be a birational map between normal projective varieties. Let $B$ and $B'$ be effective ample $\bQ$-Cartier $\bQ$-divisors on $X$ and $X'$ respectively. Denote by $U:= X\setminus \Supp(B)$ and $U':= X'\setminus \Supp(B')$. Then the following are equivalent.
\begin{enumerate}
    \item $g^{-1}$ extends to a birational morphism $g^{-1}|_{U'}: U'\to U$.
    \item For some (or equivalently, any) common resolution $p: W\to X$ and $q: W\to X'$, we have $\Supp(p^*B)\subset \Supp(q^* B')$.
    \item For any divisor $E$ over $X$, if $\ord_E(B)>0$ then $\ord_E(B')>0$.
\end{enumerate}
\end{lem}

\begin{proof}
First of all, we show that (2) is independent of the choice of the common resolution $W$. Indeed, if $\varphi: W'\to W$ is a higher resolution with $p': W'\to X$ and $q': W'\to X'$, then $\Supp(p'^*B) = \varphi^{-1}(\Supp(p^*B)) $ and $\Supp(q'^*B') = \varphi^{-1}(\Supp(q^*B'))$. Thus $\Supp(p^*B)\subset \Supp(q^* B')$ if and only if $\Supp(p'^*B)\subset \Supp(q'^* B')$. Since any two resolutions admit a common higher resolution, (2) is independent of the choice of $W$. 

Next, we show that (1) implies (3). Since $g^{-1}|_{U'}:U'\to U$ is a morphism, we know that if a divisor $E$ is centered in $U'$ then it is centered in $U$. Taking the contrapositive yields that if $E$ is centered in $\Supp(B)$ (equivalently, $\ord_E(B)>0$) then it is centered in $\Supp(B')$ (equivalently, $\ord_E(B')>0$). Thus (3) holds.

Next, we show that (3) implies (2). It is clear that $\ord_E(B) = \ord_E(p^*B)$ and $\ord_E(B') = \ord_E(q^*B')$. If a prime divisor $E$ on $W$ is contained in $\Supp(p^*B)$, then $\ord_E(p^*B)>0$ which implies $\ord_E(q^*B')>0$ by (3), hence $E$ is contained in $\Supp(q^*B')$. Thus (2) holds.

Finally, we show that (2) implies (1). Clearly, (2) implies 
\[
q^{-1}(U') = W\setminus \Supp(q^*B') \subset  W\setminus \Supp(p^*B) = p^{-1}(U).  
\]
There exists an open subset $U''\subset U'$ such that $\codim_{U'} U'\setminus U''\geq 2$ and that $q^{-1}:U''\to q^{-1}(U')$ is an open immersion. Hence composing with $p:p^{-1}(U) \to U$ we obtain a birational morphism $g^{-1}= p\circ q^{-1}: U''\to U$. Since $U$ is affine, we know that $g^{-1}:U''\to U$ corresponds to a ring homomorphism $\cO_U(U)\to \cO_{X'}(U'')$. By normality we know that  $\cO_{X'}(U'')= \cO_{X'}(U')$, and hence $g^{-1}$ extends to a morphism from $U'$ to $U$. The proof is finished.
\end{proof}

\begin{prop}\label{prop:fg-lcp}
Let $g: (X,D;\Delta)\dashrightarrow (X', D';\Delta')$ be a crepant birational log-contraction. Let $E_1,\cdots, E_r$ be a collection of prime divisors over $X$ that are lc places of $(X,D)$. 
Let $f:(\fX,\fD;\Delta_{\fX})\to \bA^r$ (resp.\ $f':(\fX',\fD';\Delta_{\fX'})\to \bA^r$) be a multi-degeneration of $(X,D;\Delta)$ (resp.\ of $(X',D'; \Delta')$) induced by $E_1, \cdots, E_r$. 
Let $o\in C$ be a  smooth pointed curve in $\bA^r$ such that $C\setminus o\subset \bG_m^r$. Let
$(\fX_C,\fD_{C};\Delta_{\fX_C}):= (\fX, \fD;\Delta_{\fX})\times_{\bA^r} C$ and $(\fX_C',\fD_C';\Delta_{\fX_C'}):= (\fX', \fD';\Delta_{\fX'})\times_{\bA^r} C$.
Then $g$ induces a crepant birational map $(\fX_C,\fD_C)\dashrightarrow (\fX_C',\fD_C')$ over $C$ such that 
\[
\LCP(\fX_C, \Delta_{\fX_C}+ \fX_o) \supset \LCP(\fX_C', \Delta_{\fX_C'}+ \fX_o').
\]
\end{prop}

\begin{proof}
Denote by $\bT:=\bG_m^r$. By Definition \ref{def:multideg} we know that $(\fX, \fD;\Delta_{\fX}) \to \bA^r$  is a family of log CY-Fano triples. Let $V:=V(t_1\cdots t_r)$ be the snc toric boundary divisor of $\bA^r$ with affine coordinates $(t_1,\cdots, t_r)$.
Since $(\fX\setminus f^{-1}(V), \Delta_{\fX}|_{\fX \setminus f^{-1}(V)})\cong (X,\Delta)\times \bT$ is klt, by Proposition \ref{prop:slc-family} we know that $(\fX, \Delta_{\fX})$ (resp.\ $(\fX, \fD)$) is klt (resp.\ lc).

Let $F_1,\cdots, F_k$ be prime divisors on $X'$ that are $g^{-1}$-exceptional. Let $\fF_i$ be the birational transform of $F_i \times \bT$ as a prime divisor over $\fX$. Denote by $d_i:=\ord_{F_i}(B)\geq 0$ and 
\[
a_i:=A_{X,D}(F_i) = A_{X',D'}(F_i) = 1-\coeff_{F_i}(D')\in [0, 1].
\]
Thus for $0<\epsilon_1< 1$ we have that 
\[
A_{\fX,\Delta_{\fX,\epsilon_1}}(\fF_i)= A_{X, \epsilon_1\Delta + (1-\epsilon_1)D} (F_i) = a_i+\epsilon_1 d_i,
\]
where $\Delta_{\fX,a}:=a \Delta_{\fX} + (1-a)\fD$.
From the above discussion we know that the pair $(\fX, \Delta_{\fX,\epsilon_1})$ is klt. 
If $\ord_{F_i}(B') >0$, then we have $\coeff_{F_i}(D')>0$ which implies $a_i <1$. Hence we have $A_{\fX,\Delta_{\fX,\epsilon_1}}(\fF_i)<1$ for $0<\epsilon_1\ll 1$. If $\ord_{F_i}(B')=0$, then by Lemma \ref{lem:log-cont} we have $d_i = \ord_{F_i}(B) = 0$ which implies that $A_{\fX,\Delta_{\fX,\epsilon_1}}(\fF_i) = a_i \leq 1$. 
Thus by \cite{BCHM10} and \cite[Corollary 1.38]{Kol13} there exists a $\bT$-equivariant projective birational morphism $\phi: \hfX\to \fX$ where $\hfX$ is $\bQ$-factorial such that the exceptional divisors of $\phi$ are precisely $\fF_1,\cdots, \fF_k$.

Let $\hX$ be a general fiber of $\hf= f\circ \phi:\hfX\to \bA^r$. Then we have a projective birational morphism $p: \hX \to X$ such that $\hX$ is $\bQ$-factorial and the exceptional divisors of $p$ are precisely $F_1, \cdots, F_k$.
In addition, we have a birational contraction $q: \hX \dashrightarrow X'$. 
By Lemma~\ref{lem:log-cont}(3), we have
\[
\Supp(q_*p^*B)\subseteq \Supp(B').
\]
Indeed, if a prime divisor $F$ on $X'$ appears in $q_*p^*B$, then its
birational transform $q_*^{-1}F$ on $\hX$ satisfies
$\ord_F(B) = \ord_{q_*^{-1}F}(p^*B)>0$, hence $\ord_F(B')>0$. 
Therefore, for $0<\epsilon \ll \epsilon_1$ we have
\begin{equation}\label{eq:fg-epsilon}
    B'\geq \epsilon q_* p^* B. 
\end{equation}

Denote by $\hDelta_{\fX,\epsilon}:= \phi_*^{-1}\Delta_{\fX,\epsilon} + \sum_{i=1}^k (1-a_i-\epsilon d_i) \fF_i$. Thus $\phi: (\hfX, \hDelta_{\fX,\epsilon})\to (\fX, \Delta_{\fX,\epsilon})$ is crepant birational over $\bA^r$.
Since $f:(\fX, \Delta_{\fX,\epsilon})\to \bA^r$ is a family of log  Fano pairs, we know that $-K_{\hfX}- \hDelta_{\fX,\epsilon} = \phi^*(-K_{\fX}- \Delta_{\fX,\epsilon})$ is nef and big over $\bA^r$. Moreover, since $(\fX, \Delta_{\fX,\epsilon})$ is klt, so is $(\hfX, \hDelta_{\fX,\epsilon})$. Thus $\hf: \hfX\to \bA^r$ is of Fano type.

Let $\fH\in |-K_{\fX}-\Delta_{\fX}|_{\bQ}$ be a general  $\bQ$-divisor. Denote by $\Gamma:= \Delta_{\fX,\epsilon}+\epsilon \fH$. Then by Bertini's theorem we have that $(\fX, \Gamma)\to \bA^r$ is a family of log CY pairs, and 
\[
\LCP(\fX_C, \Gamma_C + \fX_0)=\LCP(\fX_C,\Delta_{\fX_C,\epsilon} + \fX_0),
\]
where $\Gamma_C:=\Gamma|_{\fX_C}$ and $\Delta_{\fX_C,\epsilon}:=\Delta_{\fX,\epsilon}|_{\fX_C}$.
Denote by $\hGamma:= \phi_*^{-1}\Gamma + \sum_{i=1}^k (1-a_i - \epsilon d_i) \fF_i$. 
Since $\fH$ is general, we may assume that $\ord_{\fF_i}(\fH)=0$ which implies that
 $\phi: (\hfX,  \hGamma)\to (\fX, \Gamma)$ is crepant birational.


Next, we show that the birational map $\psi:\hfX \dashrightarrow \fX'$ induced by $q$ is a birational contraction. Over $\bT$ this is clear as $q: \hX \dashrightarrow X'$ is a birational contraction. Let $\cX_i$ (resp.\ $\hcX_i, ~\cX_i'$) be the base change of $\fX$ (resp.\ $\hfX,~ \fX'$) to the line $C_i\subset \bA^r$  where $t_j=1$ for any $j\neq i$. Clearly, $\cX_i$ and $\cX_i'$ are test configurations induced by $E_i$ which implies that $\cX_i \dashrightarrow \cX_i'$ is birational on their central fibers. Since $\hcX_i \to \cX_i$ only extracts divisors dominating $C_i$, we know that $\hcX_i\dashrightarrow \cX_i'$ is also birational on their central fibers. Thus $\psi$ is birational on  fibers over every generic point of $V$. Since both $\hfX$ and $\fX'$ are flat over $\bA^r$, we conclude that $\psi$ is a birational contraction.


Let $\Gamma':=\psi_* \hGamma$. Since $\psi$ is a birational contraction, by the negativity lemma we know that $ (\fX', \Gamma')$ is crepant birational to $(\hfX, \hGamma)$ over $\bA^r$, and hence crepant birational to $(\fX, \Gamma)$ over $\bA^r$. Thus Proposition \ref{prop:cbir-CY-family} implies that both $(\fX',\Gamma')\to \bA^r$ and $(\hfX,\hGamma)\to \bA^r$ are families of log CY pairs.
Since $C$ is smooth, there exists an snc divisor $W:=W_1+\cdots+ W_{r-1}$ in a neighborhood of $o\in \bA^r$ such that $\cap_{i=1}^{r-1} W_i = C$. Since the birational map $\fX\dashrightarrow \fX'$ restricts to $g$ on every fiber over $\bT$, it induces a birational map $\fX_C \dashrightarrow \fX_C'$ over $C$. By applying adjunction inductively to   the successive intersections $\cap_{i=1}^j f^* W_i$ ($1\leq j\leq r-1$) for the crepant birational map $(\fX, \Gamma + f^*W)\dashrightarrow (\fX', \Gamma'+f'^* W)$, we obtain a crepant birational map  between $(\fX_C, \Gamma_C + \fX_o)$ and $(\fX_C', \Gamma_C' + \fX_o')$ where $\Gamma_C':=\Gamma'|_{\fX_C}$. Since $\Delta_{\fX_C,\epsilon} =\epsilon \Delta_{\fX_C} + (1-\epsilon) \fD_{C} $ where $\fD_{C}:= \fD|_{\fX_C}$, we have that 
\[
\LCP(\fX_C, \Delta_{\fX_C} + \fX_o) = \LCP(\fX_C, \Gamma_C + \fX_o)= \LCP(\fX_C', \Gamma_C' + \fX_o').
\]
Next, we claim that $\Gamma'\geq \Delta_{\fX'}$. 
Assuming this claim, we have the desired containment
\[
\LCP(\fX_C, \Delta_{\fX_C} + \fX_o) = \LCP(\fX_C', \Gamma_C' + \fX_o')\supset \LCP(\fX_C', \Delta_{\fX_C'} + \fX_o').
\]

Finally, we prove the claim. Indeed, since $\hGamma\geq \hDelta_{\fX,\epsilon}$ where both $ \hDelta_{\fX,\epsilon}$ and  $\Delta_{\fX'}$ are $\bT$-equivariant Mumford $\bQ$-divisors on $\hfX$ and $\fX'$ over $\bA^r$ respectively, it suffices to show that on a general fiber $ X'$ of $\fX'\to \bA^r$, we have 
\begin{equation}\label{eq:fg-Delta-ineq}
q_* (p_*^{-1}(\epsilon \Delta+ (1-\epsilon) D) + \sum_{i=1}^k (1-a_i -\epsilon d_i) F_i) \geq \Delta'.
\end{equation}
Let $\hD$ be the $\bQ$-divisor on $\hX$ such that $(\hX, \hD) \dashrightarrow (X,D)$ is crepant birational. Then we have 
\[
\hD = p_*^{-1} D + \sum_{i=1}^k (1-a_i) F_i. 
\]
Since $(X,D)\dashrightarrow (X', D')$ is crepant birational, we know that $q_* \hD = D'$. Hence by \eqref{eq:fg-epsilon}, we have 
\begin{align*}
\Delta' & \leq D' - \epsilon q_* p^*B = q_* (\hD - \epsilon p^*B) \\
& = q_* (p_*^{-1}D - \epsilon p_*^{-1}B + \sum_{i=1}^k (1-a_i -\epsilon d_i) F_i). 
\end{align*}
Simplifying the above inequality by $D-\epsilon B = \epsilon\Delta+ (1-\epsilon) D$ precisely gives \eqref{eq:fg-Delta-ineq}. 
\end{proof}






\begin{proof}[Proof of Theorem \ref{thm:fg}]
For simplicity, we identify the spaces of valuations of $X$ and $X'$ via $g_*$.

Let $v\in \LCP^{\rfg}(X,D;\Delta)$ with rational rank $r$. Then by Theorem \ref{thm:fg-multideg} there exists a birational model $Y \to X$ and prime divisors $E_1,\cdots, E_r\in \LCP(X,D)$  on $Y$ such that $(Y, E= E_1+\cdots+E_r)$ is qdlt, $v$ is toroidal over $(\eta\in Y, E)$, and the multi-degeneration of $\fX$ of $(X,D;\Delta)$ induced by $E_1,\cdots, E_r$ has integral central fiber $\fX_0$. 
After replacing \(Y\) by a higher qdlt model dominating both \(X\) and \(X'\), and choosing nearby divisorial generators as in Remark \ref{rem:fg-multideg}, which preserves integrality of the induced multi-degeneration, we may assume the \(E_i\) live on such a common model.
Let $\fX'$ be the multi-degeneration of $(X', D';\Delta')$ induced by $E_1,\cdots, E_r$. 

Since $g|_U: U\to U'$ is an isomorphism, both $g$ and $g^{-1}$ are crepant birational log-contractions. Now applying Proposition \ref{prop:fg-lcp} to both $g$ and $g^{-1}$ with $C$ being the diagonal line in $\bA^r$ and $o=0$, we get 
\[
\LCP(\fX_C, \Delta_{\fX_C}+ \fX_0) = \LCP(\fX_C', \Delta_{\fX_C'}+ \fX_0').
\]

Clearly,  $\fX_C$ and $\fX_C'$ are test configurations of $(X,D;\Delta)$ and $(X',D';\Delta')$ respectively where the $\bG_m$-action is induced by the $1$-parameter subgroup in $\bT$ of weight $(1,1,\cdots, 1)$.
Let $\fX_0':=\cup_{i=1}^d \fX_0'^{(i)}$ be the irreducible components. Then we know that each $\fX_0'^{(i)}$ (resp.\ $\fX_0$) is an lc place of $(\fX_C, \Delta_{\fX_C} + \fX_0)$ (resp.\ of $(\fX_C', \Delta_{\fX_C'} + \fX_0')$). In particular, each $v_{\fX_C', i}$ (resp.\ $v_{\fX_C}$) is an lc place of $(X,D)$ (resp.\ of $(X',D')$). By Theorem \ref{thm:sa-tc}, there exists a semiample test configuration $\tfX_C$ (resp.\ $\tfX_C'$) of $(X,D;\Delta)$ (resp.\ of $(X',D';\Delta')$) induced by $\{v_{\fX_C}\}\cup \{v_{\fX_C', i}\}_{i=1}^d$. Moreover, we know that both $(\tfX_C, \Delta_{\tfX_C} + \tfX_0)\dashrightarrow (\fX_C, \Delta_{\fX_C} + \fX_0)$ and $(\tfX_C', \Delta_{\tfX_C'} + \tfX_0')\dashrightarrow (\fX_C', \Delta_{\fX_C'} + \fX_0')$ are crepant birational contractions as they only extract lc places. Thus we conclude that both $\tfX_C\to\fX_C$ and $\tfX_C'\to \fX_C'$ are birational morphisms giving ample models of the semiample test configurations.

Applying Proposition \ref{prop:sa-tc-integral} to $\tfX_C\to \fX_C$, we have that $v_{\fX_C}(f) \leq v_{\fX_C', i} (f)$ for any $1\leq i\leq d$ and any $f\in \cO_U(U)$.
Since $g|_U:U\xrightarrow{\sim}U'$, pullback induces an isomorphism
$g^*:\cO_{X'}(U')\xrightarrow{\sim}\cO_U(U)$.
Thus the same inequalities hold for every
$f'\in\cO_{X'}(U')$, after identifying valuations via $g_*$.
Applying Proposition \ref{prop:sa-tc-integral} again to $\tfX_C'\to \fX_C'$ implies that $\fX_0'$ is a birational transform of $\fX_0$ and hence irreducible. Thus $v\in\LCP^{\rfg}(X',D';\Delta')$ by Theorem \ref{thm:fg-multideg}, i.e.\ $\LCP^{\rfg}(X,D;\Delta)\subset \LCP^{\rfg}(X',D';\Delta')$. By symmetry we have $\LCP^{\rfg}(X,D;\Delta)\supset \LCP^{\rfg}(X',D';\Delta')$ and the proof is finished.
\end{proof}

\subsection{Compactification-independence of special valuations}

\begin{thm}\label{thm:sv}
Let $g: (X,D) \dashrightarrow (X',D')$ be a crepant birational map such that $g|_U: U\to U'$ is an isomorphism. Then $g_*$ induces a bijection
\[
g_*:\LCP^{\SV}(X,D;\Delta)\xrightarrow{\sim}\LCP^{\SV}(X',D';\Delta').
\]
\end{thm}

\begin{prop}\label{prop:SV-contraction}

Let
$g:(X,D;\Delta)\dashrightarrow (X',D';\Delta')$
be a crepant birational log-contraction. Then
\[
g_*\bigl(\LCP^{\SV}(X,D;\Delta)\bigr)
\subseteq
\LCP^{\SV}(X',D';\Delta').
\]
\end{prop}

\begin{proof}
Let $v\in\LCP^{\SV}(X,D;\Delta)$ with rational rank $r$. Similar to the proof of Theorem \ref{thm:fg}, we choose a qdlt model $(Y, E_1+\cdots + E_r)$ over which $v$ is toroidal such that $Y\to X$ and $Y\to X'$ are both birational morphisms.
By Theorem \ref{thm:fg-multideg} we have a multi-degeneration $\fX$ (resp.\ $\fX'$) of $(X,D;\Delta)$ (resp.\ of $(X',D';\Delta')$) induced by $E_1,\cdots, E_r$. Applying Proposition \ref{prop:fg-lcp} to $g$ with $C$ being the diagonal line in $\bA^r$ and $o = 0$, we get
\[
\LCP(\fX_C, \Delta_{\fX_C}+ \fX_0) \supset \LCP(\fX_C', \Delta_{\fX_C'}+ \fX_0').
\]
Since  $v$ is special with respect to $(X,\Delta)$, we know that $(\fX_0, \Delta_0)$ is klt which implies that $(\fX_C, \Delta_{\fX_C} + \fX_0)$ is plt. Thus $(\fX_C', \Delta_{\fX_C'}+ \fX_0') $ is plt as well, which implies that $(\fX_0', \Delta_{\fX_0'})$ is klt. Thus $v$ is special with respect to $(X',\Delta')$ by Theorem \ref{thm:fg-multideg}.
\end{proof}

\begin{proof}[Proof of Theorem \ref{thm:sv}]
Since $g|_U: U\to U'$ is an isomorphism, both $g$ and $g^{-1}$ are crepant birational log-contractions.
Thus the statement follows from applying Proposition \ref{prop:SV-contraction} to both $g$ and $g^{-1}$.
\end{proof}

\subsection{Affine log CY pairs}\label{sec:affine-CY}

\begin{defn}
An \emph{affine log CY pair} is a klt pair $(U,\Delta_U)$ such that $U$
is affine and $K_U+\Delta_U\sim_{\bQ}0$, which admits an lc-klt
log CY-Fano triple $(X,D;\Delta)$ such that, writing $B:=D-\Delta$, we have $U=X\setminus\Supp(B)$, $B_{\red}\leq D$, and $\Delta_U=\Delta|_U$.
We call $(X,D;\Delta)$ a \emph{log CY-Fano compactification} of
$(U,\Delta_U)$.

By Lemma \ref{lem:cbir}, if \((X,D;\Delta)\) and \((X',D';\Delta')\) are two log CY-Fano compactifications of \((U,\Delta_U)\), the induced birational map \((X,D)\dashrightarrow(X',D')\) is crepant. Hence the dual complex \(\mathcal D(X,D)\) is independent of the compactification. We denote it by
$\cD(U,\Delta_U)$ and call it the \emph{dual complex} of $(U,\Delta_U)$.
\end{defn}

The above definition generalizes Definition \ref{def:affine-CY} to pairs as an affine log CY variety $U$ precisely corresponds to an affine log CY pair $(U,0)$.

\begin{thm}\label{thm:independence}
Let $(U,\Delta_U)$ be an affine log CY pair. Then the cones 
\[
\LCP^{\SV}(X,D;\Delta)\subset\LCP^{\rfg}(X,D;\Delta)
\subset \LCP(X,D)
\]
are independent of the choice of log CY-Fano compactification
$(X,D;\Delta)$. Consequently, the corresponding skeleta
\[
\cD^{\rm SV}(X,D;\Delta) \subset \cD^{\rm fg}(X,D;\Delta)\subset
\cD(U,\Delta_U)
\]
are also independent of the compactification.
\end{thm}

\begin{proof}
Let $(X,D;\Delta)$ and $(X',D';\Delta')$ be two such
compactifications. By Lemma~\ref{lem:cbir},
the birational map
$g:(X,D)\dashrightarrow(X',D')$
induced by the identity on $U$ is crepant birational. 
Thus Theorems \ref{thm:fg} and \ref{thm:sv} give the desired independence.
\end{proof}

\begin{proof}[Proof of Theorem \ref{thm:main}]
It follows directly from Theorem \ref{thm:independence} by setting $\Delta_U = 0$.
\end{proof}

\begin{defn}
Let $(U,\Delta_U)$ be an affine log CY pair. Let $(X,D;\Delta)$ be a log CY-Fano compactification. 

We define the \emph{space of tropical points} $(U,\Delta_U)^{\trop}$  as 
\[
(U,\Delta_U)^{\trop}:= \LCP(X,D).
\]

We define  the \emph{cone of finitely generated valuations} $(U,\Delta_U)^{\rfg}$ and the \emph{cone of special valuations} $(U,\Delta_U)^{\SV}$ in the space of tropical points $(U,\Delta_U)^{\trop}$ as
\[
(U,\Delta_U)^{\rfg} := \LCP^{\rfg}(X,D;\Delta)\quad \textrm{ and } \quad (U,\Delta_U)^{\SV}:=\LCP^{\SV}(X,D;\Delta).
\]

Similarly, we define the \emph{finitely generated skeleton} $\cD^{\rfg}(U,\Delta_U)$ and the \emph{special skeleton} $\cD^{\SV}(U,\Delta_U)$ in the dual complex $\cD(U,\Delta_U)$ as
\[
\cD^{\rfg}(U,\Delta_U) := \cD^{\rfg}(X,D;\Delta)\quad \textrm{ and } \quad \cD^{\SV}(U,\Delta_U):=\cD^{\SV}(X,D;\Delta).
\]

By Theorem \ref{thm:independence}, these cones and skeleta are independent of the choice of the log CY-Fano compactification $(X,D;\Delta)$.
It is clear that 
\[
\cD^{\rfg}(U,\Delta_U) = ((U,\Delta_U)^{\rfg}\setminus \{v_{\triv}\})/\bR_{>0}\quad \textrm{ and }\quad \cD^{\SV}(U,\Delta_U) = ((U,\Delta_U)^{\SV}\setminus \{v_{\triv}\})/\bR_{>0}.
\]

For an affine log CY variety $U$, we simplify the above notation by $U^{\rfg}, U^{\SV}, \cD^{\rfg}(U), \cD^{\SV}(U)$.  
\end{defn}

\begin{cor}\label{cor:aut-pair}
    Let $(U,\Delta_U)$ be an affine log CY pair. Then the cones
\[
(U,\Delta_U)^{\SV}
\subset
(U,\Delta_U)^{\rfg}
\subset
(U,\Delta_U)^{\trop}
\]
are invariant under the $\Aut(U,\Delta_U)$-action. Consequently, the skeleta
\[
\cD^{\SV}(U,\Delta_U)
\subset
\cD^{\rfg}(U,\Delta_U)
\subset
\cD(U,\Delta_U)
\]
are also invariant.
\end{cor}
\begin{proof}
Let $g\in\Aut(U,\Delta_U)$ and choose a log CY-Fano compactification
$(X,D;\Delta)$ of $(U,\Delta_U)$. Transporting the compactification by $g$
gives another log CY-Fano compactification $(X',D';\Delta')$, and $g$ extends to an isomorphism
$(X,D;\Delta)\xrightarrow{\sim}(X',D';\Delta')$ which we also denote by $g$.
Since finite generation and specialness are preserved under isomorphisms, 
we have $g_*\bigl(\LCP^{\rfg}(X,D;\Delta)\bigr)
 =\LCP^{\rfg}(X',D';\Delta')$ and
$g_*\bigl(\LCP^{\SV}(X,D;\Delta)\bigr)
 =\LCP^{\SV}(X',D';\Delta')$.
By Theorem \ref{thm:independence}, the right-hand sides are respectively
$(U,\Delta_U)^{\rfg}$ and $(U,\Delta_U)^{\SV}$.
Hence both cones are $\Aut(U,\Delta_U)$-invariant.
The corresponding statement for the skeleta follows by projectivization.
\end{proof}

\begin{proof}[Proof of Corollary \ref{cor:aut}]
It follows directly from Corollary \ref{cor:aut-pair} by setting $\Delta_U = 0$.
\end{proof}

\begin{prop}\label{prop:geod-indep}
Let $(U,\Delta_U)$ be an affine log CY pair. Then finitely generated (resp.\ special) geodesic simplices in $\cD(U,\Delta_U)$ are independent of the choice of log CY-Fano compactifications. In particular, the image of a finitely generated (resp.\ special) geodesic simplex under an automorphism of $(U,\Delta_U)$ is still a finitely generated (resp.\ special) geodesic simplex.
\end{prop}

\begin{proof}
This follows directly from the proofs of Theorems \ref{thm:fg} and \ref{thm:sv}.
\end{proof}

\subsection{A maximality characterization of special valuations}

In this subsection, we return to the general setting of lc-klt log CY-Fano triples from the beginning of this section. The main application is to prove Corollary \ref{cor:maximal-affineCY}, which yields Theorem \ref{thm:maximality-intro} for affine log CY varieties, while the key discrepancy computation holds in the general setting.

\begin{defn}
We define a partial order $\preceq$ on $\LCP(X,D)$ as follows. For  two valuations
 $v,w\in \LCP(X,D)$, we define  
\[
w\preceq v \quad\textrm{if and only if} \quad w(f) \leq v(f)  \textrm{ for every }f\in \cO_U(U).
\]
This is a partial order as $\mathrm{Frac}(\cO_U(U)) = K(X)$ by the affineness of $U$. 
\end{defn}

For $w\in X^{\val}$ and
$\xi=(\xi_1,\ldots,\xi_r)\in\bR_{\geq 0}^r$, we denote by
$w_\xi$ the Gauss extension of $w$ to
$K(X)(t_1,\ldots,t_r)$ satisfying
\[
w_\xi\left(\sum_\beta f_\beta t^\beta\right)
:=
\min_\beta\{w(f_\beta)+\langle\beta,\xi\rangle\}.
\]

\begin{prop}\label{prop:max-log-disc}
Let $\pi:(\fX,\fD;\Delta_{\fX})\to \bA^r$ be a multi-degeneration of $(X,D;\Delta)$ with integral central fiber. Let $v:= v_{\fX,\xi}$ for $\xi\in \bR_{>0}^r$. Let $V:=(t_1\cdots t_r=0)\subset \bA^r$. 
Then for $w\in \LCP(X,D)$ we have 
\[
A_{\fX, \Delta_{\fX}+\fX_V}(w_{\xi}) = 0 \quad \textrm{if and only if}\quad v\preceq w.
\]
\end{prop}

\begin{proof}
Choose $l\in \bN$ sufficiently divisible so that both $\fL:=l(-K_{\fX}-\Delta_{\fX})\sim l \fB$ and $L_{\bA^r}:=l (-K_{X_{\bA^r}}-\Delta_{X_{\bA^r}})\sim l B_{\bA^r}$ are Cartier. Hence $L:= l(-K_X-\Delta)\sim lB$ is also Cartier. 
For $m\in \bN$ and $0\neq s\in R_m = H^0(X, mL)$, let $\os \in H^0(X_{\bA^r}, m L_{\bA^r})$ denote the pullback of $s$ under the projection $X_{\bA^r}\to X$. Let $\os_{\fX}$ be the rational section of the line bundle $m\fL$ such that $\os_{\fX}|_{\fX\times_{\bA^r} \bT} = \os|_{X_{\bT}}$ under the identification $\fX\times_{\bA^r} \bT\cong X_{\bT}$. 
Then by \cite[Proposition 3.17]{BL26} we have $t^{-a(s)}\os_{\fX} \in H^0(\fX, m\fL)$ where 
\[
a(s) = (a_i(s))_{1\leq i\leq r} := (v_i(s)-ml A_{X,\Delta}(v_i))_{1\leq i\leq r}.
\]

Let $\fY$ be a resolution of the graph of the birational map $\fX\dashrightarrow X_{\bA^r}$. Hence we have proper birational morphisms $p: \fY \to \fX$ and $q:\fY \to X_{\bA^r}$. Denote by $\fG:=p^*\fL - q^* L_{\bA^r}$. Then we have 
\begin{equation}\label{eq:max-1}
    0 \leq w_{\xi}(t^{-a(s)}\os_{\fX}) = w_{\xi}(t^{-a(s)} \os) + m w_{\xi} (\fG) = - \langle a(s), \xi \rangle + w(s) + m w_{\xi} (\fG). 
\end{equation}

Next, by \cite{BL26} there exists a valuatively independent basis $\Theta_m:=\{s_1, \cdots, s_{N_m}\}$ for $H^0(X, mL)$ with respect to $\LCP(X,D)$. If $s\in \Theta_m$, then we have 
\begin{equation}\label{eq:max-2}
    \langle a(s), \xi\rangle = \sum_{i=1}^r \xi_i v_i(s) - ml \sum_{i=1}^r \xi_i A_{X,\Delta}(v_i)  = v(s) - ml A_{X,\Delta}(v). 
\end{equation}
Here the last equality follows from the linearity of $\zeta\mapsto v_{\fX,\zeta}(s)$ for $s$ in a valuatively independent basis, together with the linearity of $\zeta\mapsto A_{X,\Delta}(v_{\fX,\zeta})$.
Moreover, by \cite[Proposition 3.17]{BL26} we know that $H^0(\fX, m\fL)$ is generated by $\{t^{-a(s)} \os_{\fX}\mid s\in \Theta_m\}$ as a $\bk[t_1,\cdots, t_r]$-module. After replacing $l$ by a sufficiently divisible multiple, we may assume that $\fL$ (and hence $m\fL$ for every $m\in \bN$) is globally generated. 
Thus there exists $s\in \Theta_m$ such that $w_{\xi}(t^{-a(s)} \os_{\fX}) = 0$. Combining this with \eqref{eq:max-1} and \eqref{eq:max-2} yields
\begin{equation}\label{eq:max-3}
    w_{\xi}(\fG) = \sup_{m,\,s\in \Theta_m} \frac{v(s) - w(s)}{m} - l A_{X,\Delta}(v)
    = \sup_{m,\,s\in R_m\setminus\{0\}} \frac{v(s) - w(s)}{m} - l A_{X,\Delta}(v).
\end{equation}
Here the last equality follows from the fact that for every $s\in R_m$, if we write $s = \sum_{j} c_j s_j$, then valuative independence implies that 
\[
v(s) - w(s) = \min_{c_j\neq 0} v(s_j) - \min_{c_j\neq 0} w(s_j) \leq  \max_{c_j\neq 0} (v(s_j) - w(s_j)).
\]

Next, we compute $A_{\fX, \Delta_{\fX}+\fX_{V}}(w_{\xi})$. It is clear that 
\[
\fG = l (p^*(-K_{\fX}-\Delta_{\fX})- q^*(-K_{\bA^r} - \Delta_{X_{\bA^r}})) = l(q^*(K_{\bA^r} + \Delta_{X_{\bA^r}} +X_{V})-p^*(K_{\fX}+\Delta_{\fX}+\fX_V)).
\]
Let $s_0\in H^0(X,L)$ be the section satisfying $(s_0=0) = lB$. Then we have $A_{X,\Delta}(u) = u(B) = u(s_0)/l$ for any $u\in \LCP(X,D)$. Combining this with \eqref{eq:max-3} we obtain
\begin{align*}
A_{\fX, \Delta_{\fX}+\fX_{V}}(w_{\xi})&  = A_{X_{\bA^r}, \Delta_{X_{\bA^r}} +X_{V}}(w_{\xi}) + w_{\xi}(l^{-1} \fG) = A_{X,\Delta}(w) + l^{-1} w_{\xi}(\fG) \\
& = \sup_{m,\,s\in R_m\setminus\{0\}} \left(A_{X,\Delta}(w) - A_{X,\Delta}(v) + \frac{v(s)-w(s)}{ml}\right)\\
& = \sup_{m,\,s\in R_m\setminus\{0\}} \frac{v(s/s_0^m) - w(s/s_0^m)}{ml}\\
& = \sup_{m,\, f\in A_m\setminus\{0\}} \frac{v(f) - w(f)}{ml}.
\end{align*}
Here we write $A_m := \{ s/s_0^m \mid s\in R_m\} \subset \cO_U(U)$, and we have $\cup_{m\in \bN} A_m = \cO_U(U)$ as $U = X\setminus \Supp(B)$. 
Since $v(1) = w(1) =0$, we always have $\sup_{m,\, f\in A_m\setminus\{0\}} \frac{v(f) - w(f)}{ml}\geq 0$. Thus 
$A_{\fX, \Delta_{\fX}+\fX_{V}}(w_{\xi}) = 0$ if and only if $\sup_{m,\, f\in A_m\setminus\{0\}} \frac{v(f) - w(f)}{ml} \leq 0$, which is equivalent to $v(f) \leq w(f)$ for every $f\in \cO_U(U)$. The proof is finished.
\end{proof}

\begin{thm}\label{thm:maximal-special}
Let $v\in \LCP^{\mathrm{fg}}(X,D;\Delta)$. Then $v$ is special
with respect to $(X,\Delta)$ if and only if $v$ is maximal in
$\LCP(X,D)$ with respect to $\preceq$.
\end{thm}

\begin{proof}
    First of all, let $(\fX, \fD; \Delta_{\fX}) \to \bA^r$ be a multi-degeneration induced by a finitely generated geodesic $(r-1)$-simplex containing $v$ in the interior where $r$ is the rational rank of $v$. Then we know that $\fX_0$ is integral, and we can write $v= v_{\fX, \xi}$ for some $\xi\in \bR_{>0}^r$. By Proposition \ref{prop:max-log-disc}, we know that $v$ is maximal with respect to $\preceq$ if and only if $A_{\fX, \Delta_{\fX}+\fX_V}(w_{\xi}) = 0$ for $w\in \LCP(X,D)$ implies $w = v$, or equivalently, $w_{\xi} = \wt_{\xi}$. 

    For the ``only if'' direction, suppose $v$ is special, i.e.\ $(\fX_0, \Delta_{\fX_0})$ is klt. Then by \cite[Proposition 34]{dFKX17} we know that $(\fX, \Delta_{\fX}+\fX_V)$ is dlt with minimal lc center $\fX_0$. Thus $A_{\fX, \Delta_{\fX}+\fX_V}(w_{\xi}) = 0$ implies that $w_{\xi}$ is a monomial valuation of $(\fX,\fX_{V_1}+\cdots + \fX_{V_r})$ centered at $\fX_0$. Since $w_{\xi}(\fX_{V_i}) = \xi_i$, we conclude that $w_{\xi} = \wt_{\xi}$.

    For the ``if'' direction, we show the contrapositive. Suppose $v$ is not special. Then we know that $(\fX_0, \Delta_{\fX_0})$ is integral but not klt. By \cite[Proposition 34]{dFKX17} we know that $(\fX, \Delta_{\fX}+\fX_V)$ is snc in a neighborhood of the generic point of $\fX_0$.
    Let $\mu:\fY \to (\fX, \Delta_{\fX}+\fX_V)$ be a $\bT$-equivariant dlt modification such that $\mu$ is isomorphic over the generic point of $\fX_0$; such a dlt modification exists by \cite{BCHM10} and \cite[Corollary 1.36]{Kol13}. Let $E_i:=\mu_*^{-1} \fX_{V_i}$ and  $Z:=\mu_*^{-1}\fX_0$. Then  $Z$ is an irreducible component of $\cap_{i=1}^r E_i$. Moreover, since $(\fX_0, \Delta_{\fX_0})$ is not klt, by inversion of adjunction, there exists a minimal lc center $W$ of $(\fX, \Delta_{\fX}+\fX_V)$ properly contained in $\fX_0$. Thus after replacing $\fY$ with a higher $\bT$-equivariant dlt modification we may assume that there is a prime divisor $F\subset \fY$ centered over $W$ that is an lc place of  $(\fX, \Delta_{\fX}+\fX_V)$ and $E_1\cap \cdots \cap E_r\cap F = Z\cap F\neq \emptyset$. Let $\tW$ be an irreducible component of $Z\cap F$.

    Next, let $\tw$ be the monomial valuation of $(\fY, E_1+\cdots +E_r+F)$ centered at $\tW$ of weight $(\zeta_1,\cdots,\zeta_r, c)$. Then for each $1\leq i\leq r$ we have 
    \[
    \tw(t_i) = \zeta_i + c\, \ord_F(t_i). 
    \]
    Since each $\xi_i>0$, we may choose $0<c\ll 1$ such that $\zeta_i = \xi_i - c\, \ord_F(t_i)>0$. Hence we have $\tw(t_i) = \xi_i$ for each $1\leq i\leq r$. Since $\tw$ is $\bT$-equivariant, we know that there exists a quasi-monomial valuation $w\in X^{\qm}$ such that $\tw = w_{\xi}$. Moreover, we have 
    \[
    0\leq A_{X,D}(w) = A_{\fX, \fD + \fX_V}(\tw) \leq A_{\fX, \Delta_{\fX} + \fX_V}(\tw) = 0,
    \]
    where the last equality follows from the fact that $\tw$ is an lc place of $(\fY, E_1+\cdots + E_r + F)$. Therefore, we have $w\in \LCP(X,D)$ such that $A_{\fX, \Delta_{\fX}+ \fX_V}(w_{\xi})=0$. On the other hand, we know that the center of $w_{\xi}= \tw$ is $W$ which is properly contained in $\fX_0$ as the center of $\wt_{\xi}$. Thus $w_{\xi}\neq \wt_{\xi}$. The proof is finished. 
\end{proof}

\begin{cor}\label{cor:maximal-affineCY}
Let $(U,\Delta_U)$ be an affine log CY pair. Every special valuation in $(U,\Delta_U)^{\trop}$ is maximal with respect to $\preceq$. Conversely, every finitely generated maximal valuation in $(U,\Delta_U)^{\trop}$ is special.
In particular, a divisorial valuation $v\in (U,\Delta_U)^{\trop}$ is special if and only if it is maximal with respect to $\preceq$.
\end{cor}

\begin{proof}
    This follows directly from Theorem \ref{thm:maximal-special} by choosing a log CY-Fano compactification $(X,D;\Delta)$ of $(U,\Delta_U)$ and the fact that divisorial valuations are finitely generated by Proposition \ref{prop:div-fg}.
\end{proof}

By \cite[Lemma~3.20 and Theorem~3.22]{BL26}, there is a
canonical collection of tropical theta functions
$\{\theta_i^{\trop}:(U,\Delta_U)^{\trop}\to\bR\}_{i\in I}$,
obtained from any valuatively independent basis
$\{\theta_i\}_{i\in I}$ of $\cO_U(U)$ by
$\theta_i^{\trop}(v):=v(\theta_i)$.
The collection $\{\theta_i^{\trop}\}_{i\in I}$ is independent of the choice of the
valuatively independent basis, up to reordering.

Using the tropical theta functions of \cite{BL26}, the partial
order $\preceq$ admits the following equivalent description.

\begin{prop}\label{prop:theta-order}
For $v,w\in (U,\Delta_U)^{\trop}$, we have
\[
v\preceq w
\quad\textrm{ if and only if }\quad
\theta_i^{\trop}(v)
\leq
\theta_i^{\trop}(w)
\quad\text{for every }i\in I.
\]
\end{prop}

\begin{proof}
The forward direction is clear as each $\theta_i\in \cO_U(U)$. For the backward direction, any regular function $f = \sum_{i\in I} a_i \theta_i$. Thus the valuative independence of $\{\theta_i\}_{i\in I}$ implies
\[
v(f) = \min_{a_i\neq 0} v(\theta_i)= \min_{a_i\neq 0} \theta_i^{\trop}(v)\leq  \min_{a_i\neq 0} \theta_i^{\trop}(w)=\min_{a_i\neq 0} w(\theta_i) = w(f).
\]
Thus we have $v\preceq w$. The proof is finished.
\end{proof}

\begin{proof}[Proof of Theorem \ref{thm:maximality-intro}]
This follows directly from Corollary \ref{cor:maximal-affineCY} by setting $\Delta_U=0$.
\end{proof}

\begin{expl}[The algebraic torus]\label{ex:torus}
Let $\bT=\bG_m^n$. Let $N:=\Hom(\bG_m,\bT)\cong \bZ^n$ and $M:=N^\vee$. Then 
$\bT^{\trop}\cong N_{\bR}$
and every valuation in $\bT^{\trop}$ is toric. Hence
\[
\bT^{\SV}=\bT^{\rfg}=\bT^{\trop}.
\]

The full automorphism group of $\bT$ is
\[
\Aut(\bT)=\bT\rtimes \GL(N)
       \cong \bG_m^n\rtimes \GL_n(\bZ).
\]
Under the identification $\bT^{\trop}\cong N_{\bR}$, the translation subgroup $\bT$ acts trivially, while $\GL(N)$ acts by the standard integral linear action. Thus the action of $\Aut(\bT)$ on $\bT^{\trop}$ factors through the action of $\GL(N)$ on $N_{\bR}$.

Moreover, the partial order $\preceq$ on $\bT^{\trop}$ is trivial. Indeed, if
$v_\xi\preceq v_\eta$, then
$\langle m,\xi\rangle\le \langle m,\eta\rangle$
for every $m\in M$. Applying the same inequality to $-m$ gives
$\xi=\eta$. Thus every valuation in $\bT^{\trop}$ is maximal, in accordance with Theorem \ref{thm:maximality-intro}.
\end{expl}

\section{Surfaces with infinite discrete automorphism groups}\label{sec:surfaces}

In this section, we collect three examples of affine log CY surfaces admitting
infinite discrete groups of automorphisms. The induced dynamics on the dual
complexes, as well as the structures of the special and finitely generated
skeleta, exhibit different behaviors: a nonempty open special skeleton
strictly smaller than the dual complex, a special skeleton obtained by deleting
one point, and an empty special skeleton. For further results on  special skeleta of affine log CY surfaces, see \cite{Pen25}.

\begin{expl}[The nodal cubic complement]\label{ex:nodal-cubic}
Let $(X,D):=\bigl(\bP^2,(xyz+x^3+y^3=0)\bigr)$ and $U:=X\setminus D$.
Then $D$ is an irreducible nodal cubic curve, so $(X,D)$ is an lc log CY pair
and $U$ is an affine log CY surface. Let $p:=[0,0,1]\in D$ be the node. Choose
analytic coordinates $(u_1,u_2)$ at $p$ such that $D=(u_1 u_2=0)$, and denote by
$v_{a,b}$ the corresponding monomial valuation of weight
$(a,b)\in\bR_{\geq0}^2$. Since the two local branches at $p$ belong to
the same irreducible component $D$, the two boundary rays are identified as $
v_{0,b}=v_{b,0}=b\,\ord_D$.
Thus
\[
U^{\trop}\cong
\bR_{\geq0}^2/\sim, \qquad (0,b)\sim (b,0),
\]
and hence $\cD(U)\cong_{\PL}\bS^1$.

By \cite[Section~6]{LXZ22}, the special and finitely generated cones are
\[
U^{\SV}
=
\{v_{\triv}\}\cup
\left\{
v_{a,b}\ \middle|\ a,b>0,\ 
\frac{7-3\sqrt5}{2}<\frac{b}{a}<
\frac{7+3\sqrt5}{2}
\right\},
\]
and
\[
U^{\rfg}
=
U^{\SV}\cup \bR_{>0} \cdot U^{\trop}(\bQ).
\]
After projectivizing, every point can be represented by $[v_{1,t}]$ with
$t\in [0,+\infty)$, while $t=+\infty$ denotes the class $[v_{0,1}]$.
Thus
\[
\cD(U)\cong [0,+\infty]/(0\sim+\infty),
\]
and
\[
\cD^{\SV}(U)
=
\left(
\frac{7-3\sqrt5}{2},
\frac{7+3\sqrt5}{2}
\right),
\qquad
\cD^{\rfg}(U)
=
\cD^{\SV}(U)\cup \cD(U)(\bQ).
\]

Next, we construct two involutions $\tau_1,\tau_2\in\Aut(U)$ and later study
their induced action on $U^{\SV}$. We follow Koll\'ar's reinterpretation of
this construction in terms of the Geiser involution on a degree $2$ del
Pezzo surface \cite[Paragraphs~21--23]{Kol24}; see also
\cite{Yos85,Ore02} for earlier work.

Choose an analytic branch $D_1=(u_1=0)$ of $D$.
Let $\pi_1:Y_1\to X$
be the composition of seven successive blow-ups, starting at $p$ and then
repeatedly blowing up the intersection of the strict transform of $D_1$
with the newest exceptional divisor. Denote the exceptional divisors by
$E_1,\ldots,E_7$. Then $Y_1$ is a weak del Pezzo surface of degree $2$.
Its anticanonical model $Y_1\to Y_1'$
contracts the chain $E_1,\ldots,E_6$, and $Y_1'$ is a del Pezzo surface
of degree $2$ with an $A_6$-singularity
\cite[Example~24]{Kol24}. The anticanonical linear system on $Y_1'$ gives  a double cover
$Y_1'\to\bP^2$, whose covering involution lifts to an involution $\widetilde\tau_1$ of $Y_1$. Moreover,
$\widetilde\tau_1$ exchanges the strict transform of $D$ and $E_7$;
see \cite[(23.4)--(23.5)]{Kol24}. Hence
$\tau_1:=\pi_1\circ\widetilde\tau_1\circ\pi_1^{-1}$
is a birational involution of $X$ whose restriction to $U$ is an
automorphism. This is the Geiser involution associated to the branch
$D_1$.

Equivalently, the birational morphism $Y_1'\to X$
is the $(7,1)$-weighted blow-up at $p$. Thus the ray
$\bR_{\geq0}v_{7,1}$ subdivides $U^{\trop}$ into the two cones
\[
\sigma_1^-:=\Cone((1,0),(7,1)),
\qquad
\sigma_1^+:=\Cone((7,1),(0,1)).
\]
The induced action of $\tau_1$ is integral linear on each of these cones.
Translating the transformation formula of
\cite[Proposition~30]{Kol24} (with $r=7$) into our monomial coordinates,
we obtain
\[
\tau_{1,*}(v_{a,b})
=
\begin{cases}
v_{\,7a-48b,\;a-7b} & \textrm {if } a\geq 7b,\\
v_{\,7b-a,\;b} & \textrm {if } a\leq 7b.
\end{cases}
\]
Equivalently,
\[
\left.\tau_{1,*}\right|_{\sigma_1^-}
=
\begin{pmatrix}
7&-48\\
1&-7
\end{pmatrix},
\qquad
\left.\tau_{1,*}\right|_{\sigma_1^+}
=
\begin{pmatrix}
-1&7\\
0&1
\end{pmatrix}.
\]
Indeed, the first matrix exchanges the vectors $(1,0)$ and $(7,1)$,
while the second exchanges $(7,1)$ and $(0,1)$. Along the common ray
$(7,1)$, the two formulas give $(1,0)$ and $(0,1)$,
which are identified in $U^{\trop}$.

The second involution $\tau_2$ is obtained by applying the same construction
to the other analytic branch $D_2=(u_2=0)$. Thus the corresponding wall is
the ray $\bR_{\geq 0}v_{1,7}$. Equivalently, $\tau_2$ is conjugate to
$\tau_1$ by interchanging the two branches, so
\[
\tau_{2,*}(v_{a,b})
=
\begin{cases}
v_{\,a,\;7a-b} & \textrm {if }b\leq 7a,\\
v_{\,-7a+b,\;-48a+7b} & \textrm {if } b\geq 7a.
\end{cases}
\]
In particular, $\tau_{2,*}$ is integral linear on the two cones separated
by the ray $(1,7)$.

We next show that the entire special skeleton $\cD^{\SV}(U)$ is generated, up to
automorphisms of $U$, by two simple special geodesic lines. Let
\[
s([x,y,z])=([y,x,z]),
\qquad
G:=\langle\tau_1,\tau_2\rangle,
\qquad
\Gamma:=\langle G,s\rangle<\Aut(U),
\]
so that $s_*(t)=1/t$ in the coordinate $t=b/a$. Consider
\[
\sigma:=\Cone((1,2),(2,1)),
\qquad
\sigma':=\Cone((1,2),(1,5)).
\]
The cone $\sigma$ consists of toric valuations on $\bP^2$, while
$\sigma'$ induces the degeneration
$\bP^2\rightsquigarrow\bP(1,1,4)$.

We compare these two cones with the wall decomposition of
\cite[Section~6]{LXZ22}. Let
\[
1=t_0<t_1=2<t_2=5<t_3=\frac{13}{2}<\cdots
\longrightarrow \frac{7+3\sqrt5}{2}
\]
be the walls appearing there; in particular, we have the recurrence formula $t_{k+2}=7-\frac1{t_k}$ for $k\geq 1$.
For $k\in\bZ$, define
\[
\Pi_k:=
\begin{cases}
[t_k,t_{k+1}] & \textrm{ if }k\geq 1,\\
[1/t_1, t_1] & \textrm{ if }k = 0,\\
[1/t_{1-k},1/t_{-k}] & \textrm{ if }k\leq -1.
\end{cases}
\]
Then we have $
\Pi_0=\bP(\sigma)$ and $\Pi_1=\bP(\sigma')$.
By \cite[Section~6]{LXZ22}, each $\Pi_k$ is a special geodesic line and
\[
\cD^{\SV}(U)=\bigcup_{k\in\bZ}\Pi_k.
\]
Indeed, let $F_\bullet$ denote the Fibonacci sequence, with $F_0=0$ and $F_1=1$. Then for $k\geq 1$, the interior of $\Pi_{\pm k}$ corresponds to the special degeneration
\[
\bP^2 \rightsquigarrow
\bP\bigl(1,F_{2k-1}^2,F_{2k+1}^2\bigr),
\]
where $(1,F_{2k-1},F_{2k+1})$ is a Markov triple. These weighted projective planes form the top infinite chain in the Markov tree and appear among the
Hacking--Prokhorov degenerations of $\bP^2$ \cite{HP10}.

The action of $
\Gamma=\langle\tau_1,\tau_2,s\rangle$
on these geodesic lines is particularly simple. On $\cD^{\SV}(U)$, the above piecewise-linear formulas reduce to 
\[
\tau_{1,*}(t)=\frac{t}{7t-1},
\qquad
\tau_{2,*}(t)=7-t,
\qquad
s_*(t)=\frac1t.
\]
Together with the recurrence for $t_n$, these formulas
give
\[
\tau_{1,*}(\Pi_k)=\Pi_{-k-2},
\qquad
\tau_{2,*}(\Pi_k)=\Pi_{2-k},
\qquad
s_*(\Pi_k)=\Pi_{-k}.
\]
Consequently,
\[
\Gamma\cdot\Pi_0=\bigcup_{k\in2\bZ}\Pi_k,
\qquad
\Gamma\cdot\Pi_1=\bigcup_{k\in2\bZ+1}\Pi_k.
\]
Combining this with the above description of $\cD^{\SV}(U)$ gives
\[
\cD^{\SV}(U)
=
\Gamma\cdot(\Pi_0\cup\Pi_1).
\]
Thus, up to the $\Gamma$-action, the entire special skeleton is generated
by the toric geodesic line $\Pi_0$ and the geodesic line $\Pi_1$
corresponding to the degeneration
$\bP^2\rightsquigarrow\bP(1,1,4)$.

\end{expl}

\begin{expl}[A line-parabola complement on $\bP^1\times\bP^1$]\label{ex:line-parabola}
Let $X := \bP^1_{[x,y]}\times \bP^1_{[z,t]}$, $D :=  C+\ell$ where $C:=(xyz + x^2 t + y^2 t =0)$ and $\ell:=(t=0)$. It is clear that $(X,D)$ is snc, $D$ is ample, and $K_X+D\sim 0$. Hence $U :=X\setminus D$ is an affine log CY surface. Then $C$ and $\ell$ intersect transversely at two points $p_1:=([1,0],[1,0])$ and $p_2:=([0,1],[1,0])$. 

We will show that
\[
U^{\SV}=U^{\trop}\setminus \bR_{>0}\ord_C,
\qquad
U^{\rfg}=U^{\trop}.
\]
Equivalently,
\[
\cD^{\SV}(U)=\cD(U)\setminus\{[\ord_C]\},
\qquad
\cD^{\rfg}(U)=\cD(U).
\]
Moreover, $\cD^{\SV}(U)$ contains infinitely many geodesic chambers, obtained from the action of the two Vieta involutions described below.


In the affine chart $t=1$, we have 
\[
U = (\bP^1_{[x,y]}\times \bA^1_z)\setminus (xyz+ x^2+y^2=0).
\]
Thus there are two Vieta involutions on $\bP^1\times \bA^1$:
\[
\tau_1([x,y],z) := ([-yz-x, y],z), \qquad \tau_2([x,y],z) := ([x, -xz-y], z).
\]
It is clear that $\tau_1, \tau_2\in \Aut(\bP^1\times \bA^1)$ and they preserve $C^\circ:= C\setminus \ell$. Since $U= (\bP^1\times\bA^1)\setminus C^\circ$, the Vieta involutions generate a subgroup $G:=\langle \tau_1, \tau_2\rangle < \Aut(U)$.

The birational extensions of $\tau_1$ and $\tau_2$ to $X$ have unique
indeterminacy points $p_1$ and $p_2$, respectively, and each contracts
$\ell$ to the corresponding point. We regularize $\tau_1$ as follows.
Let $\pi_1:\widehat X_1:=\Bl_{p_1}X\to  X$
be the blow-up at $p_1$, with exceptional divisor $E_1$, and let
$\widehat\ell_1$ be the strict transform of $\ell$. In the affine chart $x=z=1$, with local coordinates $(y,t)$
centered at $p_1$, we have
$\tau_1(y, t) =(-\frac{yt}{y+t},
t)$.
On the blow-up chart $(y_1,t)$ with $y=y_1 t$, the induced map sends  $(y_1,t)\mapsto (-\frac{y_1}{1+y_1},t)$ where $E_1=(t=0)$. Thus 
we have $\tau_{1,*}\ord_{E_1}=\ord_{E_1}$.

Since the strict transform
$\widehat\ell_1$ has self-intersection $-1$, contracting it gives
$\widehat X_1\to X_1\cong\bF_1$.
The induced birational self-map $\overline{\tau}_1$ of $X_1$ is an isomorphism in
codimension one, hence an automorphism. Writing
$D_1=C_1+E_1$ where $C_1$ is the birational transform of $C$,
we have that $(X_1,D_1)$ is another snc log CY compactification of $U$. The same construction regularizes
$\tau_2$ on a second copy $X_2\cong\bF_1$.

We identify
\[
U^{\trop}\cong
\{(a,b)\in\bR^2\mid a\geq0\}/\sim,
\qquad
(0,b)\sim(0,-b),
\]
where the first and fourth quadrants correspond to the strata $p_1$
and $p_2$, respectively, and
$\ord_\ell=(1,0)$, $\ord_C=(0,1)\sim(0,-1)$.
Since $E_1$ is obtained by blowing up the transverse intersection
$p_1 \in C\cap\ell$, we have
$\ord_{E_1}=(1,1)$.
With respect to the subdivision determined by the compactification
$(X_1,C_1+E_1)$, the two maximal cones are
\[
\sigma_1^+:=\Cone((0,1),(1,1)),
\qquad
\sigma_1^-:=\Cone((1,1),(0,-1)).
\]

Since $\overline\tau_1$ preserves $C_1$ and $E_1$ and exchanges their
two intersection points, $\tau_{1,*}$ fixes $(1,1)$ and exchanges
$(0,1)$ and $(0,-1)$. Hence by the integral linearity of $\tau_{1,*}$ on $\sigma^{\pm}_1$,
\[
\tau_{1,*}
=
\begin{pmatrix}
1&0\\
2&-1
\end{pmatrix}.
\]
Interchanging $x$ and $y$ acts on $U^{\trop}$ by
$(a,b)\mapsto(a,-b)$ and conjugates $\tau_1$ to $\tau_2$, so
\[
\tau_{2,*}
=
\begin{pmatrix}
1&0\\
-2&-1
\end{pmatrix}.
\]

Finally, let
\[
\sigma:=\Cone((1,1),(1,-1)).
\]
Every valuation in $\sigma$ is toric with respect to
$X=\bP^1\times\bP^1$, and hence
$\sigma\subset U^{\SV}$.

We now describe the $G$-orbit of $\sigma$. Any ray in
$U^{\trop}$ other than $\bR_{\geq 0}\ord_C$ has a unique representative
of the form $(1,s)$ for some $s\in\bR$. In the slope coordinate $s$,
the above formulas become
\[
\tau_{1,*}(s)=2-s,
\qquad
\tau_{2,*}(s)=-2-s.
\]
Thus $\tau_1$ and $\tau_2$ act as the reflections of $\bR$ about
$1$ and $-1$, respectively. In particular, $[-1,1]$ is a fundamental
domain for the action of $
G=\langle\tau_1,\tau_2\rangle\cong D_\infty$ 
on $\bR$. Since $\sigma$ is precisely the cone over $[-1,1]$, we obtain
\[
\bigcup_{g\in G}g_*\sigma
=
U^{\trop}\setminus\bR_{>0}\ord_C \subset U^{\SV},
\]
where the inclusion follows from the $G$-invariance of $U^{\SV}$ by Corollary \ref{cor:aut}.
Hence the only possible non-special ray is
$\bR_{>0}\ord_C$.

It remains to show that $\ord_C$ is not special. We claim that
\[
\ord_C(f)\leq 0=v_{\triv}(f)
\]
for every $f\in\cO_U(U)\setminus\{0\}$. Suppose otherwise that
$\ord_C(f)>0$ for some such $f$.
Since $f$ is regular on $U=(\bP^1\times\bA^1)\setminus C^\circ$ and does not have a pole along $C^\circ$, it is regular at every codimension-one
point of the normal variety $\bP^1\times\bA^1$, and hence $f\in \Gamma(\bP^1\times\bA^1,\cO)
=\bk[z]$.
But $C^\circ$ dominates $\bA^1_z$, so for every nonzero
$f\in\bk[z]$ we have
$\ord_C(f)=0$,
a contradiction. Thus
$\ord_C\preceq v_{\triv}$.
Since $\ord_C\neq v_{\triv}$, the valuation $\ord_C$ is not maximal.
By Theorem~\ref{thm:maximality-intro}, it is therefore not special.

Combining this with the preceding fundamental-domain argument gives
\[
U^{\SV}
=
U^{\trop}\setminus\bR_{>0}\ord_C.
\]
Since $\ord_C$ is divisorial and hence finitely generated,  we obtain
$U^{\rfg}=U^{\trop}$.
After projectivizing, we have $
\cD^{\SV}(U)=\cD(U)\setminus\{[\ord_C]\}$ and $\cD^{\rfg}(U)=\cD(U)$.

Finally, the cones $g_*\sigma$ for $g\in G$ give infinitely many distinct maximal special geodesic cones. Indeed, $\sigma$ is precisely the maximal
cone in $U^{\trop}$ consisting of toric valuations with respect to
$X=\bP^1\times\bP^1$. If $\sigma$ were properly contained in a special
geodesic cone $\sigma'$, then any  valuation in the interior of $\sigma'$ would induce isomorphic degenerations of $X$ by Theorem \ref{thm:fg-multideg}. Since every valuation in $\sigma$ induces a product $\bR$-test configuration of $X$, so does every valuation 
in the interior of $\sigma'$. This implies that $\Int(\sigma')$ only consists of toric valuations with respect to $X$, contradicting the
maximality of $\sigma$. Thus $\sigma$ is a maximal special geodesic cone.
By Proposition \ref{prop:geod-indep}, so is every $g_*\sigma$. Since their projectivizations
correspond to the distinct intervals
$[2k-1,2k+1]$ for $ k\in\bZ$,
they give infinitely many geodesic chambers in
$\cD^{\SV}(U)$.
\end{expl}

\begin{expl}[The Cayley cubic]\label{ex:cayley-cubic}
Consider the affine Cayley cubic surface
\[
U:=\{x^2+y^2+z^2+xyz-4=0\}\subset \bA^3.
\]
The Cayley cubic admits the quotient presentation $U\cong \bT/\langle\iota\rangle$ where $\bT= \bG_m^2$ and $\iota$ is the involution on $\bT$ given by $\iota(t_1,t_2):=(t_1^{-1},t_2^{-1})$,
via
\[
(x,y,z)=(t_1+t_1^{-1}, t_2+t_2^{-1}, -t_1 t_2-t_1^{-1}t_2^{-1}).
\]
We take the projective compactification
\[
X=
(xyz+w(x^2+y^2+z^2)-4w^3=0)
\subset\bP^3
\]
with boundary $D = X\cap (w=0)$. The surface $U$ has four $A_1$-singularities, corresponding to the
fixed points $\bT[2]$ of $\iota$. Moreover, $X$ is smooth along $D$
and $D$ is an snc triangle. Thus $(X,D)$ is lc, and by adjunction $K_X+D\sim 0$.
Since $D$ is an ample hyperplane section and $U$ is klt, $(X,D)$ is
a log CY compactification of the affine log CY surface $U$.

We will show that 
\[
\cD(U) \cong \bP^1(\bR), \qquad \cD^{\SV}(U) = \emptyset, \qquad \cD^{\rfg}(U) = \cD(U)(\bQ).
\]

We start by describing $U^{\trop}$. By \cite[Remark~1.1]{GHKS22}, the involution $\iota$ on $\bT$
extends to the toric del Pezzo surface $\widetilde X$ of degree $6$, and
the quotient $\widetilde X/\langle\iota\rangle$ is the projective Cayley cubic surface $X$. 
Moreover, the quotient of the toric boundary $\tD$ is the triangle
$D=X\setminus U$. 
Since $\iota$ is fixed-point-free away from $\bT[2]$, we know that the quotient map $(\tX,\tD) \to (X,D)$ is \'etale in a neighborhood of the boundary.  Hence the corresponding
lc places are identified modulo $\iota_*$, and therefore
\[
U^{\trop}\cong \bT^{\trop}/\langle \iota_*\rangle \cong \bR^2/\{\pm1\}.
\]
After projectivizing, we
obtain
\[
\cD(U)\cong
(\bR^2\setminus\{0\})/
\bigl(\bR_{>0}\times\{\pm1\}\bigr)
\cong \bP^1(\bR).
\]

We next exhibit an infinite discrete group acting on $U$.
The natural
$\GL_2(\bZ)$-action on $\bT$ commutes with the involution
$\iota=-I$. Hence it descends to  an action of $\PGL_2(\bZ)=\GL_2(\bZ)/\langle -I\rangle$ on  $U$. 
Under the identification $U^{\trop}\cong \bR^2/\{\pm1\}$, 
the induced $\PGL_2(\bZ)$-action on $U^{\trop}$ is the standard linear action. After projectivizing, this gives the standard $\PGL_2(\bZ)$-action on $\bP^1(\bR)\cong\cD(U)$.
In particular, $\PGL_2(\bZ)$ acts transitively on $\cD(U)(\bQ)\cong\bP^1(\bQ)$.

Similar to the Markov cubic in Section \ref{sec:Markov}, the Cayley cubic $U$ carries three natural Vieta involutions
\[
\tau_1(x,y,z)=(-yz-x,y,z),\quad
\tau_2(x,y,z)=(x,-xz-y,z),\quad
\tau_3(x,y,z)=(x,y,-xy-z).
\]
The subgroup
$G:=\langle \tau_1,\tau_2,\tau_3\rangle
< \Aut(U)$
generated by the three Vieta involutions
is identified with $\ker\bigl(\PGL_2(\bZ)\to\PGL_2(\bF_2)\bigr)$.
Thus $G\cong \bmu_2 *\bmu_2 * \bmu_2$ is the $(\infty,\infty,\infty)$-triangle reflection group. The full $\PGL_2(\bZ)$-action is obtained from $G$ by adjoining the
coordinate permutations of the Cayley cubic.

We now determine the special and finitely generated skeleta of $U$.
For the special skeleton, let $v_{\triv}\neq v = \wt_{\xi}|_{K(U)} \in U^{\trop}$ for $\xi\in N_{\bR}\setminus \{0\}$. For
\[
f=\sum_{m\in M}c_m\chi^m\in\cO_U(U)\setminus \{0\},
\]
we have $c_m=c_{-m}$ as $\cO_U(U) = \cO_{\bT}(\bT)^{\iota} = \bk[M]^{\iota}$. Hence
\[
v(f) = \wt_{\xi}(f) = \min_{c_m\neq0}\langle m,\xi\rangle
\leq 0 = v_{\triv}(f).
\]
Thus $v\preceq v_{\triv}$ which implies that no nontrivial valuation in $U^{\trop}$
is maximal. It follows from Theorem \ref{thm:maximality-intro} that $\cD^{\SV}(U)=\emptyset$.

Finally, every rational point of $\cD(U)$ is divisorial and hence finitely generated, so
$\cD(U)(\bQ)\subset \cD^{\rfg}(U)$.
On the other hand, every irrational point of $\cD(U)$ has rational rank
$2=\dim U$. If such a valuation were
finitely generated, then Theorem \ref{thm:toric} would imply that it is special, a contradiction to $\cD^{\SV}(U)=\emptyset$. Thus
$\cD^{\rfg}(U)=\cD(U)(\bQ)$.
\end{expl}




\section{Skeleta and dynamics of the Markov cubic complement}\label{sec:Markov}

In this section, we shall focus on a specific affine log CY threefold $U$ as the affine complement of the Markov cubic surface. Let $X = \bP^3$ with coordinates $[x,y,z,w]$, and $D: = \oS+H$ where $H := (w=0)$ and $\oS := (xyz+(x^2  + y^2 +z^2)w = 0)$. Let $U := X\setminus D = \bA^3 \setminus S$, where $\bA^3=\bP^3\setminus H$, and   $S := \oS\setminus H$ is the affine Markov cubic surface.
We shall write $f(x,y,z):= xyz + x^2 + y^2 + z^2$ and $S=(f=0)$ in the affine chart $w = 1$.

Since $H$ is smooth, $\oS$ has a single $A_1$-singularity, and $(H, \oS|_H)\cong (\bP^2, (xyz=0))$ is snc, by inversion of adjunction we know that $(X, D; 0)$ is an lc-klt log CY-Fano triple, and $U$ is an affine log CY variety.

The goal of this section is to describe the special and finitely generated skeleta of $U$ by analyzing the $\Aut(U)$-action on $\cD(U)$ and relating it to results on tropical dynamics from \cite{Jan23}\footnote{We are grateful to Simion Filip who brought this paper to our attention.} (see Theorems \ref{thm:markov-special}, \ref{thm:markov-fg}, \ref{thm:satake}), and finally prove Theorems \ref{thm:intro-markov} and \ref{thm:counterex-LX24}.

\begin{rem}[Relation with \cite{Jan23}]
The tropical-dynamical input in this section comes from Jang's study of
Vieta involutions on tropical Markov surfaces. In particular,
\cite{Jan23} introduces the invariant tropical skeleton, computes the
tropical Vieta involutions, describes its decomposition into the orbit
of the central ``ping-pong'' triangle and exceptional rays, and identifies
the resulting dense open dynamics with the $(\infty,\infty,\infty)$
triangle reflection group acting on $\bH^2$.

Our contribution is to relate this tropical picture to the dual complex
and to special and finitely generated skeleta. We identify the two parts $\cD^\pm$ of
$\cD(U)$ equivariantly with the corresponding tropical skeleta, determine $\cD^{\SV}(U)$ and $\cD^{\rfg}(U)$, and interpret the exceptional
rays in terms of finitely generated and special valuations. We also extend
the hyperbolic description to the boundary of the special skeleton and
use this structure to prove the failure of simplexwise local closedness.
\end{rem}

\subsection{Dual complex}
\begin{defn}\label{def:monomial-val}
Consider the closed simplicial cone
\[
\sigma_0:= \Cone\bigl((1,0,1),(0,1,1),(0,0,1)\bigr) \subset \bR^3.
\]
In other words, $\sigma_0$ consists of $\alpha=(\alpha_1, \alpha_2, \alpha_3)\in \bR^3_{\geq 0}$ satisfying $\alpha_1 + \alpha_2 \leq \alpha_3$. Let $p_0:=[0,0,0,1]$, $p_1:=[1,0,0,0]$, $p_2:=[0,1,0,0]$, and $p_3:=[0,0,1,0]$.
Let $v_{\alpha,i}$ and $v_{\gamma,i}'$ for $i\in \{1,2,3\}$ be monomial valuations at $p_i$ on $\bP^3$ of weight $\alpha,\gamma\in \sigma_0$ respectively in the following local coordinates:
\begin{itemize}
    \item $v_{\alpha,1}: (y,z,w)$ in the affine chart $x=1$;
    \item $v_{\alpha,2}: (x,z,w)$ in the affine chart $y=1$;
    \item $v_{\alpha,3}: (x,y,w)$ in the affine chart $z=1$;
    \item $v_{\gamma,1}': (y,z,w')$ in the affine chart $x=1$ where $w':=w(y^2+z^2 + 1)+yz$;
    \item $v_{\gamma,2}': (x,z,w')$ in the affine chart $y=1$ where $w':=w(x^2+z^2 + 1)+xz$;
    \item $v_{\gamma,3}': (x,y,w')$ in the affine chart $z=1$ where $w':=w(x^2+y^2 + 1)+xy$.
\end{itemize}
Note that $w'$ is simply the dehomogenized polynomial defining the Markov cubic surface in the corresponding affine chart.
\end{defn}

\begin{lem}\label{lem:dltmodel}
Let $\mu_i:\hX_i \to X$ be the $(1,1,2)$-weighted blow-up at $p_i$ in the corresponding affine coordinates of $v_{\alpha, i}$ from Definition \ref{def:monomial-val}. Denote by $E_i$ the exceptional divisor of $\mu_i$. Let $\mu: \hX \to X$ be the composition of these three weighted blow-ups. Then $(\hX, \mu_*^{-1} D + E_1+E_2+E_3)$ is a dlt modification of $(X,D)$. 
\end{lem}

\begin{proof}
From the equation of $D$ we know that $(X,D)$ is snc away from $\{p_0,p_1,p_2,p_3\}$ and plt at $p_0$. Moreover, we have $A_X(E_i) = \ord_{E_i}(D) = 4$ which implies that each $E_i$ is an lc place of $(X,D)$. 
Thus it suffices to show that $(\hX, \mu_*^{-1}D+E_1+E_2+E_3)$ is dlt in a neighborhood of each $E_i$, or equivalently,  $(\hX_i, \mu_{i,*}^{-1} D+E_i)$ is dlt in a neighborhood of $E_i$ for each $i$. By symmetry we assume $i = 1$. Then we can work in the affine chart $x = 1$ and perform the $(1,1,2)$-weighted blow-up at $p_1=(0,0,0)\in \bA^3_{(y,z,w)}$ where $D$ has equation $(yz+ (1+y^2+z^2) w) w  = 0$. Since the problem is local near $p_1$, after replacing $w$ by $(1+y^2+z^2) w$ we may assume that $D$ has equation $(yz+w) w = 0$. 

Next, we choose three affine charts $(y_j, z_j, w_j)$ for $j\in \{1,2,3\}$ to cover the weighted blow-up, where 
\[
(y, z, w)  = (y_1, y_1 z_1, y_1^2 w_1) = (y_2 z_2, z_2, z_2^2 w_2) = (y_3 w_3, z_3 w_3,w_3^2). 
\]
Here $(y_j,z_j,w_j)$ represents a smooth $\bA^3$ for $j\in \{1,2\}$ and the $\bmu_2$-quotient $\bA^3/\{\pm 1\}$ for $j = 3$. In each coordinate, we can write
\[
\mu_{1,*}^{-1}D + E_1 = \begin{cases}
V((z_1 + w_1)w_1) + V(y_1) & \textrm{ if }j=1;\\
V((y_2 + w_2)w_2) + V(z_2) & \textrm{ if }j=2;\\
V(y_3 z_3 + 1) + V(w_3) & \textrm{ if }j=3.
\end{cases}
\]
Thus we see that $(\hX_1, \mu_{1,*}^{-1}D+E_1)$ is snc along the smooth locus of $E_1$ and plt at the unique singular point of $E_1$, which implies that it is dlt along $E_1$. 

Finally, we present a geometric interpretation of the above computation. It is clear that $E_1\cong \bP(1,1,2)$ with projective coordinates $[y,z,w]$ by abuse of notation. Then we know that 
\[
\mu_{1,*}\oS|_{E_1} = V(yz + w)\quad \textrm{ and }\quad \mu_{1,*}H|_{E_1}= V(w).
\]
In particular, both  $\mu_{1,*}\oS|_{E_1}$ and $\mu_{1,*}H|_{E_1}$ are smooth curves in the smooth locus of $E_1$ that intersect transversely at two points $q_1:=[1,0,0]$ and  $q_1':=[0,1,0]$. Since both $\hX_1$ and $E_1$ are smooth away from the unique singular point $[0,0,1]$ of $E_1$, we know that $(\hX_1, \mu_{1,*} D + E_1)$ is snc along the smooth locus of $E_1$ and plt at $[0,0,1]$, leading to the same conclusion as above.
\end{proof}

\begin{prop}\label{prop:LCP-Markov}
A valuation $v$ is an lc place of $(X,D)$ if and only if $v$ is of the form $v_{\alpha,i}$ or $v_{\gamma,i}'$ for some $i \in \{1,2,3\}$ and $\alpha,\gamma\in \sigma_0$. 
\end{prop}

\begin{proof}
We first show that each $v_{\alpha,i}$ or $v_{\gamma,i'}$ is an lc place of $(X,D)$. By symmetry, we may assume that $i = 1$ and work in the affine chart $x = 1$. Then we have
\[
A_{X}(v_{\alpha,1}) = \alpha_1 + \alpha_2 + \alpha_3, \qquad A_{X}(v_{\gamma,1}') = \gamma_1+\gamma_2+\gamma_3.
\]
Recall that $D= (w w'=0)$ where $w'=w(y^2 + z^2 + 1) + yz$.
Moreover, we have 
\[
v_{\alpha,1}(D) = v_{\alpha,1}(w) + v_{\alpha,1}(w') = \alpha_3 + \min \{\alpha_3, \alpha_1+\alpha_2\}= \alpha_3 + \alpha_1 + \alpha_2.
\]
Thus $A_{(X,D)}(v_{\alpha, 1}) = 0$. Since $ w= (y^2+z^2+1)^{-1}( w' - yz)$ where $y^2 + z^2 + 1$ does not vanish at the origin, we have 
\[
v_{\gamma,1}' (D) = v_{\gamma,1}'(w) + v_{\gamma,1}'(w') = v_{\gamma,1}'(w' - yz) + v_{\gamma,1}'(w')  = \min\{\gamma_3, \gamma_1+\gamma_2 \} + \gamma_3 =  \gamma_1 + \gamma_2 + \gamma_3.
\]
Thus $A_{(X,D)}(v_{\gamma, 1}') = 0$.

Next, we show that every lc place $v$ is of the form $v_{\alpha,i}$ or $v_{\gamma, i}'$. If the center of $v$ is a surface, then $v=c\ord_{\oS}$ or $v = c\ord_H$, where we can write $v = v_{\alpha,1}$ or $v=v_{\gamma,1}'$ with $\alpha = (0,0,c)$. If the center of $v$ is a curve $C$, then it has to be a line in the intersection $\oS\cap H = V(xyz, w)$. By symmetry, assume $C= V(y,w)$. Then $v$ is monomial in the coordinate $(w,w')$ of weight $(a,b)\in \bR_{>0}^2$ along $C$. If $a\geq b$, then we have $v = v_{\alpha, 1}$ with $\alpha = (b,0,a)$. If $a \leq b$, then we have $v = v_{\gamma, 1}'$ with $\gamma = (a,0,b)$. 

Next, assume that the center of $v$ is a point.
Since $(X,D)$ is plt at $p_0$, no lc place has center
$p_0$. Thus, by symmetry, we may assume that the center
of $v$ is $p_1$.
From the proof of Lemma \ref{lem:dltmodel}, we know that $v$ is monomial in $(E_1,\mu_{1,*}^{-1} H , \mu_{1,*}^{-1}\oS)$ of weight $(a,b,c)\in \bR_{>0}\times\bR_{\geq 0}^{2}$ at $q_1$ or $q_1'$. If $0<c\leq b$, then there are precisely two monomial valuations $v$, one centered at $q_1$ and the other centered at $q_1'$. It is clear that 
\begin{align*}
v_{\alpha,1}(E_1) &= \min \{v_{\alpha,1}(y), v_{\alpha,1}(z), \tfrac{1}{2}v_{\alpha,1}(w)\} = \min \{\alpha_1, \alpha_2\}, \\
v_{\alpha,1}(\mu_{1,*}^{-1} H )& = v_{\alpha,1}(H)- 2 v_{\alpha,1}(E_1) = \alpha_3 - 2\min \{\alpha_1, \alpha_2\},\\
v_{\alpha,1}(\mu_{1,*}^{-1} \oS ) &= v_{\alpha,1}(\oS)- 2 v_{\alpha,1}(E_1) = \alpha_1 + \alpha_2 - 2\min \{\alpha_1, \alpha_2\}.
\end{align*}
Thus we can take $\alpha = (a, a+c, 2a + b)$ or $(a+c,a, 2a+b)$ which gives us two valuations of the form $v_{\alpha,1}$ with the same weight as $v$. Since $v_{\alpha,1}$ is an lc place, we conclude that we must have $ v=v_{\alpha,1}$ for some $\alpha$ from the two choices. If $0<b\leq c$, then similarly we have 
\begin{align*}
v_{\gamma,1}'(E_1) &= \min \{v_{\gamma,1}'(y), v_{\gamma,1}'(z), \tfrac{1}{2}v_{\gamma,1}'(w')\} = \min \{\gamma_1, \gamma_2\}, \\
v_{\gamma,1}'(\mu_{1,*}^{-1} H )& = v_{\gamma,1}'(H)- 2 v_{\gamma,1}'(E_1) = \gamma_1+\gamma_2 - 2\min \{\gamma_1, \gamma_2\},\\
v_{\gamma,1}'(\mu_{1,*}^{-1} \oS ) &= v_{\gamma,1}'(\oS)- 2 v_{\gamma,1}'(E_1) = \gamma_3 - 2\min \{\gamma_1, \gamma_2\}.
\end{align*}
Thus we can take $\gamma = (a, a+b, 2a + c)$ or $(a+b,a, 2a+c)$ such that $v = v_{\gamma,1}'$ for some $\gamma$ from the two choices. If either $b=0$ or $c=0$, then there is only one valuation $v$ of given weight $(a,b,c)$ whose center on $\hX_1$ is a curve. Then the above two choices of $\alpha$ or $\gamma$ coincide and the same argument works. 
\end{proof}

\begin{cor}
The dual complex $\cD(U)$ is PL homeomorphic to $\bS^2$.  The space of tropical points $U^{\trop}$ is PL homeomorphic to $\bR^3$.
\end{cor}

\begin{proof}
This follows from the gluing relations of $v_{\alpha,i}$ and $v_{\gamma, i}'$ which we summarize below.
\begin{align*}
v_{(0, \alpha_2, \alpha_3),1} = v_{(0, \alpha_2, \alpha_3), 2}, \quad v_{(\alpha_1, 0, \alpha_3),1} = v_{(0, \alpha_1, \alpha_3), 3}, \quad v_{(\alpha_1,0, \alpha_3),2} = v_{(\alpha_1, 0, \alpha_3), 3};\\
v_{(0, \gamma_2, \gamma_3),1}' = v_{(0, \gamma_2, \gamma_3), 2}', \quad v_{(\gamma_1, 0, \gamma_3),1}' = v_{(0, \gamma_1, \gamma_3), 3}', \quad v_{(\gamma_1,0, \gamma_3),2}' = v_{(\gamma_1, 0, \gamma_3), 3}';
\end{align*}
\[
v_{(\alpha_1, \alpha_2, \alpha_1+\alpha_2), i} = v_{(\alpha_1, \alpha_2, \alpha_1+\alpha_2), i}'\quad \textrm{ for each } i\in \{1,2,3\}.
\]
See Figure \ref{fig:dualcx} for an illustration of the gluing relations. 
\end{proof}

\begin{figure}[htbp] 
    \centering
        \begin{tikzpicture}[scale=6,
    single/.style={postaction={decorate},decoration={markings, mark=at position 0.5 with {\arrow[scale=1.5]{>}}}},
    double/.style={postaction={decorate},decoration={markings, 
        mark=at position 0.48 with {\arrow[scale=1.5]{>}},
        mark=at position 0.52 with {\arrow[scale=1.5]{>}}}},
    triple/.style={postaction={decorate},decoration={markings, 
        mark=at position 0.47 with {\arrow[scale=1.5]{>}},
        mark=at position 0.5  with {\arrow[scale=1.5]{>}},
        mark=at position 0.53 with {\arrow[scale=1.5]{>}}}}
]

\coordinate (A) at (0,0);
\coordinate (B) at (1,0);
\coordinate (C) at (0.5,{sqrt(3)/2});

\coordinate (D) at ($(A)!0.5!(B)$);
\coordinate (E) at ($(B)!0.5!(C)$);
\coordinate (F) at ($(C)!0.5!(A)$);

\coordinate (G) at (barycentric cs:D=1,E=1,F=1);

\draw (A) -- (B) -- (C) -- cycle;
\draw (D) -- (E) -- (F) -- cycle;
\draw (G) -- (D);
\draw (G) -- (E);
\draw (G) -- (F);

\node at (barycentric cs:A=1,D=1,F=1) {$v_{\gamma,1}'$};
\node at (barycentric cs:B=1,D=1,E=1) {$v_{\gamma,2}'$};
\node at (barycentric cs:C=1,E=1,F=1) {$v_{\gamma,3}'$};

\node[xshift=0.05cm] at (barycentric cs:G=1,D=1,F=1) {$v_{\alpha,1}$}; 
\node[xshift=-0.05cm] at (barycentric cs:G=1,D=1,E=1) {$v_{\alpha,2}$}; 
\node[yshift=-0.04cm] at (barycentric cs:G=1,E=1,F=1) {$v_{\alpha,3}$}; 

\draw[single] (F) -- (A);
\draw[single] (F) -- (C);

\draw[double] (D) -- (A);
\draw[double] (D) -- (B);

\draw[triple] (E) -- (B);
\draw[triple] (E) -- (C);

\end{tikzpicture}        
    \caption{The dual complex $\cD(U)$}
    \label{fig:dualcx}
\end{figure}
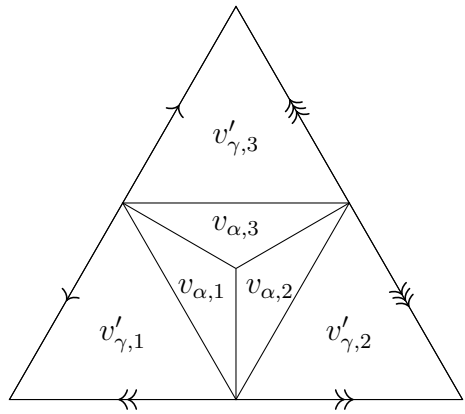


\subsection{Vieta involutions}\label{sec:Vieta}
We have the following involutions on $\bA^3$ preserving $S$:
\begin{align*}
\tau_1: (x,y,z)\mapsto (-yz-x, y,z)\\
\tau_2: (x,y,z) \mapsto (x, -xz-y, z)\\
\tau_3: (x,y,z) \mapsto (x, y, -xy-z)
\end{align*}
In particular, we have $\tau_1, \tau_2, \tau_3\in \Aut(U)$. 
Let $G:=\bmu_2 * \bmu_2 * \bmu_2$, then we have an injective group homomorphism $G\hookrightarrow\Aut(U)$ whose image is the subgroup $\langle\tau_1, \tau_2, \tau_3\rangle$ by \cite[Theorem 1]{EH74} and \cite[Theorem 3.1]{CL09}. We shall describe the $G$-action on $\cD(U)\cong \bS^2$. 

\begin{lem}
Let $\phi_1:\bP^3 \dashrightarrow \bP(1,1,1,2)$ be the birational map given by $[x,y,z, w]\mapsto [w, y, z, xw]$. Then $\phi_1$ can be resolved by $\mu_1:\hX_1\to\bP^3$ and a birational morphism $\hat{\phi}_1: \hX_1\to \bP(1,1,1,2)$ that contracts the strict transform of $H$ to a rational curve. Moreover, $\tau_1$ is regularized on $\bP(1,1,1,2)$ by $\phi_1$. 
\end{lem}

\begin{proof}
It is clear that $\phi_1$ is induced by the  linear subsystem $\langle w^2, y^2, z^2, wy, wz, yz, xw \rangle\subset |\cO_{\bP^3}(2)|$. The base ideal $\cI$ of this linear subsystem is cosupported at $[1,0,0,0]$. In the affine coordinate $x=1$, we have $\cI = (w, y^2, z^2, yz)$. Thus by blowing up $\cI$ we resolve $\phi_1$ and get a morphism $\Bl_{\cI} \bP^3 \to \bP(1,1,1,2)$. Clearly, we have $\cI^m = \fa_2m(E_1)$ for any positive integer $m$. Thus we have $\hX_1\cong \Bl_{\cI}\bP^3$ which implies the first statement. Let $[u_0, u_1, u_2, u_3]$ be the projective coordinates of $\bP(1,1,1,2)$. Then we have $[x, y, z, w]= [\frac{u_3}{u_0},  u_1, u_2,u_0]$. Thus the conjugation of $\tau_1$ under $\phi_1$ is given by $[u_0, u_1, u_2, u_3]\mapsto [u_0 ,u_1, u_2, -u_1u_2-u_3]$ which is clearly an automorphism of $\bP(1,1,1,2)$.
\end{proof}

\begin{prop}\label{prop:tau_1-expression}
We have the following description of $\tau_{1,*}$-action on $\cD(U)$. 
\begin{enumerate}
    \item We have $\tau_{1,*}(v_{\alpha, i}) = v_{\gamma,1}'$ for $i \in \{1,2,3\}$ where 
    \[
    (\gamma_1, \gamma_2, \gamma_3)= \begin{cases}
    (\alpha_3-\alpha_2, \alpha_3-\alpha_1, 3\alpha_3 - 2\alpha_1-2\alpha_2) & \textrm{ if } i = 1,\\
    (\alpha_3-\alpha_2, \alpha_3, 3\alpha_3 - 2\alpha_2+ \alpha_1) & \textrm{ if } i = 2,\\
    (\alpha_3, \alpha_3-\alpha_2, 3\alpha_3 - 2\alpha_2+ \alpha_1) & \textrm{ if } i =3.
    \end{cases}
    \]
    \item We have $\tau_{1,*}(v_{\alpha,i}') = v_{\gamma,1}'$ for $i\in \{2,3\}$ where 
    \[
    (\gamma_1, \gamma_2, \gamma_3) =\begin{cases}
    (\alpha_1, \alpha_1+\alpha_2, \alpha_3 + 3\alpha_1) & \textrm{ if } i = 2,\\
    (\alpha_1+\alpha_2, \alpha_1, \alpha_3 + 3\alpha_1) & \textrm{ if } i = 3.
    \end{cases}
    \]
\end{enumerate}
Since \(\tau_1\) is an involution, part (1) also determines $\tau_{1,*}(v'_{\alpha,1})$.
\end{prop}

A proof of Proposition \ref{prop:tau_1-expression} can be given by explicitly writing down the map $\phi_{1,*}:U^{\trop}\to U_1^{\trop}$ where $U_1:= \bP(1,1,1,2) \setminus \phi_{1,*}S$, as $\phi_1: (X,D) \dashrightarrow (\bP(1,1,1,2), \phi_{1,*}S)$ is a crepant birational map inducing an isomorphism between $U$ and $U_1$. Since the computation is quite lengthy, in the next subsection we will give an alternative proof using essential skeleta and computations from \cite{Jan23}.

\subsection{Essential skeleton and tropicalization}

In this subsection, we review the work of Jang \cite{Jan23} on the action of Vieta involutions on the tropicalized Markov cubic surfaces. Our main result (Theorem \ref{thm:trop-iso-action}) shows the dual complex $\cD(U)$ can be divided into two parts $\cD^{-}$ and $\cD^+$ where each part resembles a $G$-action on the tropical skeleton of certain Markov cubic surfaces from \cite{Jan23}. We  follow the setup from Section \ref{sec:an-trop}.

Recall that $U = \bA^3_{(x,y,z)}\setminus (f=0)$ where $f(x,y,z)=xyz+x^2 + y^2 + z^2$. Let $\varphi^{\pm}: U \hookrightarrow \bA^3 \times (\bA^1_t\setminus \{0\})$ be the graph of maps $U\to \bA^1_t \setminus \{0\}$ given by $(x,y,z) \mapsto t = f(x,y,z)^{\pm 1}$. Then $\varphi^{\pm}$ are closed immersions whose images $U^{\pm}:=\varphi^{\pm}(U)$ are  hypersurfaces $(f(x,y,z) - t^{\pm 1} = 0) \subset \bA^3 \times (\bA^1_{t}\setminus \{0\})$. Let $K:= \bk(\!(t)\!)$. Then the base change $U^{\pm}_K:= U^{\pm}\times_{\bA^1_t \setminus \{0\}} \Spec\, K$ are hypersurfaces in $\bA^3_K$ defined by $(f(x,y,z) - t^{\pm 1} =0)$. According to the notation $S_{ABCD}$ from \cite{Jan23}, we have $U_K^{+} = S_{000t}$ and $U_K^{-} = S_{000t^{-1}}$. 

Next, notice that $U_K^{\pm}$ is an affine log CY $K$-variety where a log CY compactification $(X_K^{\pm}, D_K^{\pm})$ is given by 
\[
X_K^{\pm}  := (xyz+(x^2+y^2+z^2) w - t^{\pm 1} w^3 = 0) \subset \bP^3_K, \quad D_K^{\pm} := (w=0)|_{X_K^{\pm}}.
\]
Here we use $[x,y,z,w]$ as the projective coordinates for $\bP^3$ as abuse of notation.
Let us describe the essential skeleta $\Sk(U_K^{\pm})$. We start from constructing log CY models of $(X_K^{\pm}, D_K^{\pm})$. Consider the hypersurface 
\[
\tX:= (t_0(xyz + (x^2+y^2+z^2)w) - t_1 w^3=0) \subset \bP^3\times \bP^1_{[t_0, t_1]}. 
\]
It follows from direct Jacobian computations that $\tX$ is a normal projective threefold and the first projection $\pr_1: \tX\to X=\bP^3$ is birational.
Let $\tD:= (xyz=w=0)\subset \tX$. Denote by $0:=[1,0]$ and $\infty:=[0,1]$ in $\bP^1$. 

\begin{prop}\label{prop:CY-model}
The pair $(\tX,\tD+\tX_0 +\tX_{\infty, \red})$ is an lc log CY pair that is crepant birational to $(X,D)$ under $\pr_1$. Moreover, $(\tX,\tD)/(0\in \bP^1)$ (resp.\ $(\tX,\tD)/(\infty\in \bP^1)$) is a log CY model of $(X_K^{+}, D_K^{+})$ (resp.\ of  $(X_K^{-}, D_K^{-})$) as in Section \ref{sec:an-trop}.
\end{prop}

\begin{proof}
By analyzing the projection $\pr_1$, we know that $\pr_1: \tX \to X$ is birational with $\tD$ the only exceptional divisor. Since $\tX$ is a $(3,1)$-divisor in $\bP^3\times \bP^1$, we have $-K_{\tX} \sim \cO(1,1)$ by adjunction. Thus we have $K_{\tX} + \tX_0 + (w=0)|_{\tX} \sim 0$. Then it is easy to see that $(w=0)|_{\tX}  = \tD + \tX_{\infty, \red}$ which  implies $K_{\tX} + \tD+ \tX_0 + \tX_{\infty, \red}  \sim 0$. Moreover, since $(X,D)$ is an lc log CY pair with $\pr_{1,*}^{-1}D = \tX_0 + \tX_{\infty, \red}$, we know that its crepant pull-back is an lc log CY pair $(\tX, a \tD + \tX_0 + \tX_{\infty, \red})$ for some coefficient $a$ which implies $K_{\tX} + a\tD+ \tX_0 + \tX_{\infty, \red}  \sim 0$. Thus we have $(a-1)\tD\sim 0$ and hence $a=1$. The last statement follows from the definition and adjunction.
\end{proof}

\begin{prop}\label{prop:Sk-Markov}
The essential skeleton
$\Sk(U_K^{\pm})$ is PL homeomorphic to the subset $\{v\in U^{\trop}\mid v(f) =  \pm 1\}$ of $U^{\trop}$. Under this identification, we have
\begin{align*}
\Sk(U_K^-) &= 
\cup_{i=1}^3(\{v_{\alpha,i}\mid \alpha\in \sigma_0,~3\alpha_3 -\alpha_1 - \alpha_2  =  1\} \cup \{v_{\gamma, i}'\mid \gamma\in \sigma_0,~3\gamma_1+3\gamma_2 - \gamma_3 = 1\}),\\
\Sk(U_K^+) &= \cup_{i=1}^3\{v_{\gamma, i}'\mid \gamma\in\sigma_0,~\gamma_3 - 3\gamma_1-3\gamma_2 = 1\}.
\end{align*}
\end{prop}

\begin{proof}
By Propositions \ref{prop:Sk-LCP} and \ref{prop:CY-model}, we know that the restriction map $U_K^{\pm,\val} \to \{v\in U^{\val}\mid v(f) = \pm 1\}$ induces a PL homeomorphism 
\[
\Sk(U_K^{\pm}) \to \{v\in \LCP(X,D)\mid v(f) = \pm 1\}.
\]
Here we use the fact that crepant birational log CY pairs $(X,D)$ and  $(\tX,\tD+\tX_0 +\tX_{\infty, \red})$ have the same $\LCP$. 
Then the first statement follows from the identification that $\LCP(X,D) = U^{\trop}$.
Since $\mathrm{div}(f) = \oS - 3H$ on $X$,  the proof of Proposition \ref{prop:LCP-Markov} implies 
\begin{align*}
v_{\alpha,i}(f) &  = v_{\alpha,i}(\oS) -  3v_{\alpha,i}(H) = (\alpha_1+\alpha_2) - 3 \alpha_3, \\
v_{\gamma,i}'(f) &  = v_{\gamma,i}'(\oS) -  3v_{\gamma,i}'(H) = \gamma_3 - 3(\gamma_1+\gamma_2).
\end{align*}
Thus the second statement follows from Proposition \ref{prop:LCP-Markov} and the above computations.
\end{proof}

Following Section \ref{sec:an-trop}, we have the tropicalization maps 
\begin{equation}\label{eq:tropicalization}
    \Trop: U_K^{\pm,\an} \to (\bR\cup \{\infty\})^3,
\end{equation}
whose restrictions on $U_K^{\pm,\val}$ are precisely given by $v\mapsto (v(x), v(y), v(z))$. By \cite[Section 2]{Jan23}, we know that $\Trop(U_K^{\pm})$ are the closure of the loci in $\bR^3_{\ux}$ with $\ux = (x_1,x_2,x_3)$ where the tropicalized polynomials
\[
f^{\pm}(x_1,x_2,x_3):=\min(x_1+x_2+x_3, 2x_1, 2x_2, 2x_3, \pm 1)
\]
are not differentiable.

\begin{defn}[{\cite[Definition 3.2]{Jan23} and \cite{Fil19}}]\label{def:Sk-trop}
The \emph{tropical skeleton} $\Sk^{\trop}(U_K^{\pm})$ of $U_K^{\pm}$ is defined as 
\[
\Sk^{\trop}(U_K^{\pm}):= \{\ux\in \Trop(U_K^{\pm})\cap \bR^3\mid x_1+x_2+x_3=f^{\pm}(x_1,x_2,x_3)\}.
\]
\end{defn}

We note that $\Sk^{\trop}(U_K^{\pm})$ were denoted by $\Sk(\infty,\infty,\infty, \pm 1)$ in \cite{Jan23}. 

\begin{prop}\label{prop:trop-Sk}
The tropicalization maps \eqref{eq:tropicalization} induce PL homeomorphisms between the essential skeleton $\Sk(U_K^{\pm})$ and the tropical skeleton $\Sk^{\trop}(U_K^{\pm})$.
\end{prop}

\begin{proof} 
From the definition we know that $\Trop:\Sk(U_K^{\pm}) \to \bR^3$ is a continuous PL map. Thus it suffices to show that it is bijective onto $\Sk^{\trop}(U_K^{\pm})$. From Definition \ref{def:Sk-trop} we have 
\[
\Sk^{\trop}(U_K^{\pm}) = \{\ux\in  \bR^3\mid x_1+x_2+x_3=\min(2x_1, 2x_2, 2x_3, \pm 1)\}.
\]

We start from $U_K^+$. Clearly, if $\min(2x_1, 2x_2, 2x_3) \geq 1$ then each $x_i\geq \frac{1}{2}$ which implies that $x_1+x_2+x_3 >1$ so $(x_1,x_2,x_3)\not\in \Sk^{\trop}(U_K^{\pm})$. Thus we have 
\begin{align*}
\Sk^{\trop}(U_K^+) & = \{\ux\in  \bR^3\mid x_1+x_2+x_3 =  \min(2x_1,2x_2,2x_3)\}\\
& = \cup_{i=1}^3 \{\ux\in  \bR_{\leq 0}^3\mid x_i = x_{i+1}+x_{i+2}\}.
\end{align*}
Here we adopt the mod $3$ convention for the subscript. Geometrically, $\Sk^{\trop}(U_K^+)$ is the boundary of the simplicial cone generated by $(-1,-1,0)$, $(-1,0,-1)$, and $(0,-1,-1)$.
By Proposition \ref{prop:Sk-Markov}, we have
\[
\Sk(U_K^{+}) = \cup_{i=1}^3 \{v_{\gamma,i}'\mid \gamma_1,\gamma_2\in \bR_{\geq 0}\textrm{ and }\gamma_3=3\gamma_1+3\gamma_2+1\}.
\]
Then computations give
\[
\Trop(v_{\gamma,i}') =\begin{cases}
(-\gamma_1 - \gamma_2, -\gamma_2, -\gamma_1)& i = 1\\
(-\gamma_2, -\gamma_1-\gamma_2, -\gamma_1) & i = 2\\
(-\gamma_2, -\gamma_1, -\gamma_1-\gamma_2) & i = 3
\end{cases}
\]
Thus by checking each $1\leq i\leq 3$ we conclude that $\Trop$ maps $\Sk(U_K^+)$ bijectively onto $\Sk^{\trop}(U_K^+)$.

Next, we consider $U_K^-$. Similar analysis shows that 
\[
\Sk^{\trop}(U_K^-) =\{\ux\in \bR_{\geq -\tfrac{1}{2}}^3\mid x_1+x_2+x_3 = -1\}\cup \cup_{i=1}^3 \{\ux\in  \bR_{\leq 0}^3\mid x_i = x_{i+1}+x_{i+2}\leq -\tfrac{1}{2}\}.
\]
Geometrically, $\Sk^{\trop}(U_K^-)$ is the boundary of the truncated simplicial cone generated by $(-1,-1,0)$, $(-1,0,-1)$, and $(0,-1,-1)$ satisfying $x_1+x_2+x_3\leq -1$. Hence it consists of a bounded central triangle and three unbounded triangular faces.
For $v_{\alpha,i}\in \Sk(U_K^{-})$, computations give
\[
\Trop(v_{\alpha,i}) =\begin{cases}
(-\alpha_3, \alpha_1 - \alpha_3, \alpha_2 - \alpha_3) & i = 1\\
(\alpha_1-\alpha_3, -\alpha_3, \alpha_2 - \alpha_3) & i = 2\\
(\alpha_1-\alpha_3, \alpha_2 - \alpha_3, -\alpha_3) & i = 3
\end{cases}
\]
The formula for $\Trop(v_{\gamma,i}')$ is the same as before. Thus from the constraints on $\alpha$ and $\gamma$ from Proposition \ref{prop:Sk-Markov} we conclude that $\Trop:\Sk(U_K^-) \to \Sk^{\trop}(U_K^-)$ is bijective and maps the set of $v_{\alpha,i}$'s to the central triangle and the set of $v_{\gamma,i}'$'s to the three unbounded triangular faces.
\end{proof}

Next, we relate the action of three Vieta involutions on the dual complex $\cD(U)$ to the tropical Vieta involutions on $\Sk^{\trop}(U_K^{\pm})$. 

\begin{defn}\label{def:D^pm}
We define two open subsets $\cD^{\pm}$ and a closed subset $\cD^0$ of $\cD(U)$ as 
\begin{align*}
\cD^- &:= \{v_{\alpha,i}\mid \alpha\in \sigma_0, ~1\leq i\leq 3\} \cup \{v_{\gamma, i}'\mid \gamma\in \sigma_0, ~\gamma_3<3\gamma_1+3\gamma_2 ,~1\leq i\leq 3\},\\
\cD^+ &:= \{v_{\gamma, i}'\mid \gamma\in \sigma_0,~\gamma_3 > 3\gamma_1+3\gamma_2,~1\leq i\leq 3\},\\
\cD^0 & := \{v_{\gamma, i}'\mid \gamma\in \sigma_0,~\gamma_3 = 3\gamma_1+3\gamma_2,~1\leq i\leq 3\}.
\end{align*}
\end{defn}

It is straightforward to see 
\[
\cD(U) = \cD^+\sqcup \cD^0 \sqcup \cD^-, \qquad \cD^0 = \partial \cD^+ = \partial \cD^-.
\]



Next, denote by $\tau_i^{\pm}$ the induced involutions on $U_K^{\pm}$ by $\tau_i$. Recall from \cite[(3.1)--(3.3)]{Jan23} that the tropical Vieta involutions $\Trop(\tau_i^{\pm})$ on $\Sk^{\trop}(U_K^{\pm})$ are defined as 
\begin{align*}
\Trop(\tau_1^{\pm})(\ux) & := (\min(2x_2, 2x_3, \pm 1)-x_1, x_2, x_3), \\
\Trop(\tau_2^{\pm})(\ux) & := (x_1, \min(2x_1, 2x_3, \pm 1)-x_2, x_3), \\
\Trop(\tau_3^{\pm})(\ux) & := (x_1, x_2, \min(2x_1, 2x_2, \pm 1)-x_3).
\end{align*}
Thus the group $\langle \Trop(\tau_1^{\pm}), \Trop(\tau_2^{\pm}), \Trop(\tau_3^{\pm})\rangle$ provides a $G$-action on $\Sk^{\trop}(U_K^{\pm})$. The following result matches up the $G$-actions on the dual complex with the tropical skeleta.

\begin{thm}\label{thm:trop-iso-action}
Both open subsets $\cD^{\pm}\subset \cD(U)$ are $G$-invariant. Moreover, the tropicalization maps \eqref{eq:tropicalization} induce $G$-equivariant PL homeomorphisms  $\Phi^{\pm}:\cD^{\pm}\to \Sk^{\trop}(U_K^{\pm})$. 
\end{thm}


\begin{proof}

First of all, by composing the identification of $\Sk(U_K^{\pm})$ with the subset $\{v\in U^{\trop}\mid v(f) = \pm 1\}$ from Proposition \ref{prop:Sk-Markov} and the projection $U^{\trop}\setminus \{v_{\triv}\} \to \cD(U)$, it is straightforward to see that $\Sk(U_K^{\pm})$ is $G$-equivariantly PL homeomorphic to $\cD^{\pm}$. In particular, the open subsets $\cD^{\pm}$ are both $G$-invariant. This defines the PL homeomorphism $\Phi^{\pm}$ as the composition $\cD^{\pm} \to \Sk(U_K^{\pm}) \xrightarrow{\Trop} \Sk^{\trop}(U_K^{\pm})$. Thus it suffices to show that $\Phi^{\pm}$ is $G$-equivariant. Since the $G$-action on $U$ is generated by $\tau_1, \tau_2, \tau_3$, this reduces  to showing $\Phi^{\pm}\circ \tau_{i,*} = \Trop(\tau_i^{\pm})\circ \Phi^{\pm}$ for $1\leq i\leq 3$ as maps from $\cD^{\pm}$ to $\Sk^{\trop}(U_K^{\pm})$. By symmetry, we may assume $i=1$. 

Next, for any $[v]\in \cD^{\pm}$ we may choose a representative $v$ as the restriction of a valuation in $\Sk(U_K^{\pm})$. In particular, we may assume $v(f) = \pm 1$. Then we have 
\[
\Phi^{\pm}(\tau_{1,*}(v)) = (v(x\circ \tau_1), v(y\circ \tau_1), v(z\circ \tau_1)) = (v(-yz-x), v(y), v(z)).
\]
Meanwhile, we have
\[
\Trop(\tau_1^{\pm})(\Phi^{\pm}(v)) = (\min(2v(y), 2v(z), \pm 1)-v(x), v(y), v(z)).
\]
Thus it suffices to show that $v(-yz-x) = \min(2v(y), 2v(z), \pm 1)-v(x)$, or equivalently, $v(xyz+x^2) =  \min(v(y^2), v(z^2), v(f))$. This follows directly from Lemma \ref{lem:xyz+x^2}. Thus the proof is finished.
\end{proof}

\begin{lem}\label{lem:xyz+x^2}
For every valuation $v\in U^{\trop}$, we have 
\begin{equation}\label{eq:xyz+x^2}
v(xyz+x^2) = \min(v(y^2), v(z^2), v(f)).
\end{equation}
\end{lem}

\begin{proof}
We work in the projective coordinates $[x,y,z,w]$ of $\bP^3$ where \eqref{eq:xyz+x^2} becomes 
\begin{equation}\label{eq:proj-xyz+x^2}
    v(\tfrac{xyz+x^2 w}{w^3})  = \min(v(\tfrac{y^2}{w^2}), v(\tfrac{z^2}{w^2}), v(\tfrac{xyz+(x^2+y^2+z^2)w}{w^3})).
\end{equation}

By Proposition \ref{prop:LCP-Markov}, the valuation $v$ has the form $v_{\alpha,i}$ or $v_{\gamma,i}'$ for $\alpha,\gamma\in \sigma_0$ and $1\leq i\leq 3$. By symmetry in $(y,z)$ we may assume that $i = 1$ or $2$. Below we shall split into four cases.

If $v$ has the form $v_{\alpha,1}$, then in the affine chart $x= 1$ the equation \eqref{eq:proj-xyz+x^2} becomes 
\begin{equation}\label{eq:xyz+x^2-chart-x}
    v(yz+ w) = \min(v(y^2 w) , v(z^2 w), v(yz + (1+y^2+z^2)w).
\end{equation}
Since $v$ is monomial in $(y,z,w)$, we have 
\begin{align*}
v(yz+w) & = \min(\alpha_1+\alpha_2, \alpha_3) = \alpha_1+\alpha_2,\\
v(y^2w) & = 2\alpha_1+\alpha_3\geq \alpha_1+\alpha_2,\\
v(z^2w) & = 2\alpha_2+\alpha_3\geq \alpha_1+\alpha_2,\\
v(yz + (1+y^2+z^2)w) & =\min (\alpha_1+\alpha_2, \alpha_3) = \alpha_1+\alpha_2.
\end{align*}
Thus \eqref{eq:xyz+x^2-chart-x} holds.

If $v$ has the form $v_{\gamma,1}'$, then similarly in the affine chart $x=1$ it suffices to show \eqref{eq:xyz+x^2-chart-x}. Since $v$ is monomial in $(y,z,w')$ with $w' = yz + (1+y^2+z^2)w$, we have $v(w) = v(w'-yz) = v(yz)$ as $v(w') = \gamma_3 \geq \gamma_1+\gamma_2 = v(yz)$. Thus we have 
\begin{align*}
v(yz+w)&  = v(yz + \tfrac{w'-yz}{1+y^2+z^2}) = v(yz(1+y^2+z^2) + w' - yz) = v(yz(y^2+z^2) + w') \\
& = \min(v(y^2\cdot yz), v(z^2\cdot yz), v(w')) = \min (v(y^2 w), v(z^2 w), v(w')).
\end{align*}
Thus \eqref{eq:xyz+x^2-chart-x} holds.

If $v$ has the form $v_{\alpha,2}$, then in the affine chart $y=1$ the equation \eqref{eq:proj-xyz+x^2} becomes
\begin{equation}\label{eq:xyz+x^2-chart-y}
    v(xz+x^2 w) = \min(v(w), v(z^2 w), v(xz + (1+x^2+z^2)w)). 
\end{equation}
Since $v$ is monomial in $(x,z,w)$, we have
\begin{align*}
v(xz+x^2 w) & = \min(\alpha_1+\alpha_2, 2\alpha_1 + \alpha_3) = \alpha_1+\alpha_2,\\
v(w) & = \alpha_3\geq \alpha_1+\alpha_2,\\
v(z^2w) & = 2\alpha_2+\alpha_3\geq \alpha_1+\alpha_2,\\
v(xz + (1+x^2+z^2)w) & = \min(\alpha_1+\alpha_2, \alpha_3) = \alpha_1+\alpha_2.
\end{align*}
Thus \eqref{eq:xyz+x^2-chart-y} holds.

If $v$ has the form $v_{\gamma,2}'$, then similarly in the affine chart $y=1$ it suffices to show \eqref{eq:xyz+x^2-chart-y}.
Since $v$ is monomial in $(x,z,w')$ with $w' = xz+(1+x^2+z^2)w$, we have $v(w) = v(w'-xz) = v(xz)$ as $v(w') = \gamma_3\geq \gamma_1+\gamma_2 = v(xz)$. Thus we have 
\begin{align*}
v(xz+x^2 w)&  = v(xz + \tfrac{x^2(w'-xz)}{1+x^2+z^2}) = v(xz(1+x^2+z^2) + x^2(w' - xz) )\\ 
& = v(xz(1+z^2) + x^2 w')  = \min( v(xz), v(z^2\cdot xz), v(x^2 w'))\\& = v(xz) = \gamma_1+\gamma_2.
\end{align*}
On the other hand, we have 
\[
\min(v(w), v(z^2 w), v(w')) = \min(v(w), v(w')) = \min (\gamma_1+\gamma_2, \gamma_3) = \gamma_1+\gamma_2.
\]
Thus \eqref{eq:xyz+x^2-chart-y} holds. The proof is finished.
\end{proof}

\begin{proof}[Proof of Proposition \ref{prop:tau_1-expression}]
This follows directly from the $G$-equivariant homeomorphisms $\Phi^{\pm}:\fD^{\pm}\to \Sk^{\trop}(U_K^{\pm})$ from Theorem \ref{thm:trop-iso-action}, the formula for tropical Vieta involutions from \cite[(3.1)--(3.3)]{Jan23}, and the description of $\Phi^{\pm}$ from the proof of Proposition \ref{prop:trop-Sk}.
\end{proof}

\subsection{Special and finitely generated skeleta}

In this subsection, we completely describe the special skeleton $\cD^{\SV}(U)$ and the finitely generated skeleton $\cD^{\rfg}(U)$.

\begin{defn}
    The \emph{central triangle} $\Pi\subset \cD(U)$  is defined as 
    \[
    \Pi:= \{[v_{\alpha,i}]\mid \alpha\in \sigma_0, ~1\leq i \leq 3\}.
    \]
\end{defn}

From the definition, one can see that $\Pi$ precisely consists of all valuations in $\cD(U)$ that are toric (or equivalently, monomial) with respect to the affine coordinates $(x,y,z)$ of $\bA^3$. Since toric valuations on $\bA^3$ are naturally identified with $\bR^3$ under the tropicalization map $\Phi^-$, we know that  $\Pi$ is identified with the triangle in $\bR^3$ with vertices at $(-\frac{1}{2},-\frac{1}{2},0), (-\frac{1}{2},0,-\frac{1}{2}), (0, -\frac{1}{2}, -\frac{1}{2})$. This justifies calling $\Pi$  a triangle.

\begin{defn}
    Let $p,q\in \bZ_{\geq 0}$ be coprime non-negative integers. Define an \emph{open exceptional line} $\ell_{p,q,i}\subset \cD(U)$ for $1\leq i\leq 3$ as
    \[
    \ell_{p,q,i} = \{v_{\gamma,i}'\mid \gamma_1= p, ~\gamma_2 = q,~ \gamma_3> 3p+3q -2\}. 
    \]
    A \emph{closed exceptional line} $\oell_{p,q,i}$ is the closure of $\ell_{p,q,i}$ in $ \cD(U)$. We also define an \emph{open (resp.\ a closed) $\pm$-exceptional ray} as 
    \[
    \ell_{p,q,i}^{\pm} := \ell_{p,q,i}\cap \cD^{\pm}\quad (\textrm{resp.}~\oell_{p,q,i}^{\pm} := \oell_{p,q,i}\cap \overline{\cD^{\pm}}).
    \]
\end{defn}

See Figure~\ref{fig:exc-lines} for an illustration of the closed exceptional
lines in $\cD(U)$.

Next, we state the two main results of this subsection.

\begin{thm}\label{thm:markov-special}
The special skeleton $\cD^{\SV}(U)$ is the disjoint union of $\{\ord_S\}$ and $\cD^{\SV, -} := \cD^{\SV}(U)\cap \cD^-$ where $\cD^{\SV,-}$ is a path-connected component of $\cD^{\SV}(U)$. Moreover, we have the following description of $\cD^{\SV,-}$.
\begin{enumerate}
    \item We have $\cD^{\SV,-} = G\cdot \Pi$. Moreover, the central triangle $\Pi$ is a fundamental domain for the $G$-action on $\cD^{\SV,-}$.
    \item $\cD^{\SV,-}$ is dense in $\cD^{-}$. The complement  $\cD^-\setminus \cD^{\SV,-}$ is the union of all open $-$-exceptional rays $\ell_{p,q,i}^{-}$.
\end{enumerate}
\end{thm}

\begin{thm}\label{thm:markov-fg}
The finitely generated skeleton $\cD^{\rfg}(U)$ is the union of $\cD^-$ and all closed exceptional lines $\oell_{p,q,i}$. In other words, every valuation in $\cD^-$ is finitely generated, and a valuation $v_{\gamma,i}'\in \overline{\cD^+}$ is finitely generated if and only if $\gamma_1$ and $\gamma_2$ are linearly dependent over $\bQ$. 
\end{thm}

See Figures~\ref{fig:SV-Markov} and~\ref{fig:hyp-reflection} for the geometry of the component
$G\cdot\Pi$ and its relation with the
$(\infty,\infty,\infty)$-triangle reflection group.

\begin{figure}[htbp]
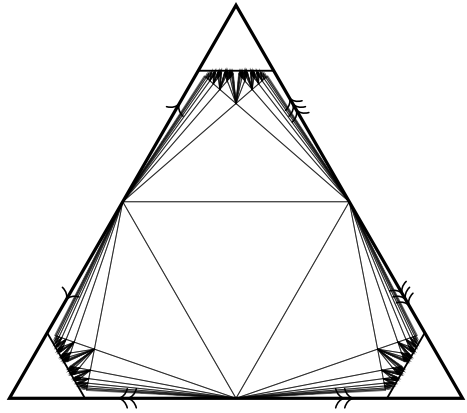
 
    \centering
        \include{dualcx-sv.tex}
    \caption{The connected component $G\cdot \Pi$ of the special skeleton}
    \label{fig:SV-Markov}
\end{figure}

\begin{figure}[htbp]
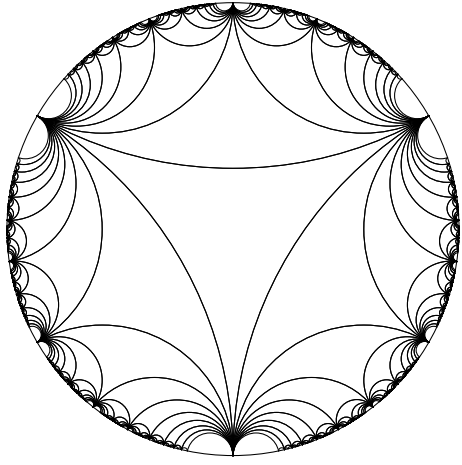
 
    \centering
        \include{hyp-reflection.tex}        
    \caption{The $(\infty,\infty,\infty)$-triangle reflection group action on the Poincar\'e disk}
    \label{fig:hyp-reflection}
\end{figure}

\begin{figure}[htbp]
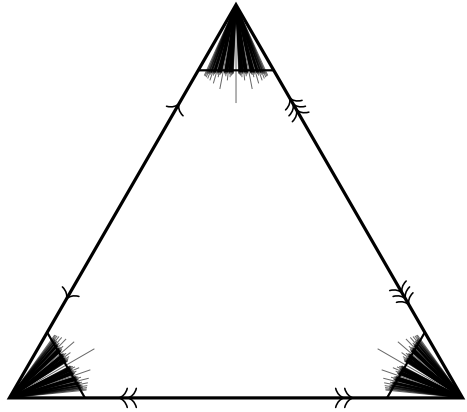
 
    \centering
        \include{dualcx_exceptional.tex}
    \caption{Closed exceptional lines in the dual complex}
    \label{fig:exc-lines}
\end{figure}

We need some preparation before proving the above two theorems.

\begin{prop}\label{prop:exc-rays}
The set $\cD^- \setminus G\cdot \Pi$ is the union of all open $-$-exceptional rays $\ell^-_{p,q,i}$. 
\end{prop}

\begin{proof}
Under the $G$-equivariant PL homeomorphism $\Phi^{-}:\cD^-\to \Sk^{\trop}(U_K^-)$ by Theorem \ref{thm:trop-iso-action}, this follows directly from \cite{Jan23}. More precisely, let $\Pi^\circ:=\Pi\setminus \VPi$ where $\VPi$ consists of the three vertices of $\Pi$. Then by \cite[Theorem 9.3 and Proposition 9.4]{Jan23} (see also \cite[Section 1.2]{Jan23}) we have that $\Sk^{\trop}(U_K^-)$ is the disjoint union of $G\cdot \Phi^-(\Pi^\circ)$ and these closed exceptional rays $\bR_{\leq - \frac{1}{2}}\cdot (q, p, p+q) $, $\bR_{\leq -  \frac{1}{2}}\cdot (p, p+q,q) $ or $\bR_{\leq -  \frac{1}{2}}\cdot (p+q, q, p) $  where $p,q$ are coprime nonnegative integers. Note that by the definition of $\Phi^-$ these closed exceptional rays are precisely the $\Phi^-$-images of closed $-$-exceptional rays $\oell_{p,q,i}^-$. Moreover, the union of the three $G$-orbits of vertices of $\Pi$ is precisely the endpoints of these closed $-$-exceptional rays. Thus the proof is finished. 
\end{proof}

\begin{prop}\label{prop:exc-lines-fg}
Every closed exceptional line $\oell_{p,q,i}$ is a geodesic line. In particular, we have $\oell_{p,q,i}\subset \cD^{\rfg}(U)$. Moreover, a valuation $v\in \oell_{p,q,i}$ is special if and only if $v\in \oell_{p,q,i}\setminus \ell_{p,q,i}$, i.e.\ it is one of the two endpoints $\ord_S$ and $v_{(p,q,3p+3q-2),i}'$.
\end{prop}

\begin{proof}
By \cite[Theorem 9.3 and Proposition 9.4]{Jan23}, every endpoint $v_{(p,q,3p+3q-2),i}'$ of an exceptional line belongs to the $G$-orbit of one of the three vertices of $\Pi$. Since coordinate permutations belong to $\Aut(U)$ and act transitively on these three vertices, all such endpoints belong to a single $\Aut(U)$-orbit. The common endpoint $\ord_S$ is fixed by $\Aut(U)$. Thus by Proposition \ref{prop:geod-indep} it suffices to prove the statements for a specific closed exceptional line  $\oell_{1,0,1}$. 

Let $(\cX, \ocS + \cH)\to \bA^1_t$ be the test configuration of $(\bP^3,\oS + H)$ induced by the valuation $v:=v_{(1,0,1),1}' = v_{(1,0,1),1}$ by Theorem \ref{thm:sa-tc}. Since $v$ is a vertex of $\Pi$ and hence a toric valuation on $\bP^3$ with $v(x) = v(z) = 0$ and $v(y)=v(w)=1$ in the projective coordinates $[x,y,z,w]$, we have $\cX = \bP^3 \times \bA^1$, $\cH = H \times \bA^1$, and $\ocS = V(xyz+wx^2+ wz^2+t^2 wy^2)$. Moreover, by adjunction and computations we have that the central fiber $\oS_0 = V(xyz+wx^2+wz^2) $ of $\ocS \to \bA^1$ is irreducible and slc, but not klt as it is not normal along $(x=z=0)$.

Next, by applying Proposition \ref{prop:deg-normal-cone} to $v_1:=v$ and $S:= \oS$ we have that there is a finitely generated geodesic line connecting $v$ and $\ord_{S}$. Since in the simplex $\{v_{\gamma, 1}'\}$ we can identify $v= v_{(1,0,1),1}'$ and $\ord_S = v_{(0,0,1),1}'$, the straight line  $\oell_{1,0,1}$ connecting them in this simplex is precisely this geodesic line by Proposition \ref{prop:geodesic-simplex}. Moreover, for any valuation in the open line
$\ell_{1,0,1}$, the central fiber of its induced test
configuration is isomorphic to $\fX_{(0,0)}$. By
Proposition~\ref{prop:deg-normal-cone}, for some $\varrho\in\bN$ the quotient
$\fX_{(0,0)}/\bmu_{\varrho}$
is a projective cone over $\oS$. Since $\oS$ is not normal,
this quotient is not normal. Hence $\fX_{(0,0)}$ is not
normal, and therefore not klt. It follows that no valuation
in the open line $\ell_{1,0,1}$ is special. On the other hand, both vertices $v$ and $\ord_S$ are special as they induce klt degenerations of $\bP^3$ to $\bP^3$ and the projective cone over $\oS$ with polarization $\cO_{\oS}(3)$ respectively. This finishes the proof of the statements for $\oell_{1,0,1}$ and hence for all $\oell_{p,q,i}$.
\end{proof}

\begin{proof}[Proof of Theorem \ref{thm:markov-special}]

(1) We first show that $G\cdot \Pi\subset \cD^{\SV, -}$. By Corollary \ref{cor:aut} we know that $\cD^{\SV, -}$ is $G$-invariant. Thus it suffices to show that $\Pi\subset \cD^{\SV,-}$. This is clear as $\Pi$ only consists of toric valuations on $X=\bP^3$ which are automatically special. 

Next, by Proposition \ref{prop:exc-rays} the complement $\cD^-\setminus G\cdot \Pi$ is the disjoint union of all open $-$-exceptional rays, which contain no special valuations by Proposition \ref{prop:exc-lines-fg}. Thus we have $\cD^{\SV, -}\subset G\cdot \Pi$.

Under $\Phi^-$, the fundamental-domain assertion is exactly
the ping-pong description of the central triangle in
\cite[Theorem 9.3]{Jan23}.

(2) This follows directly from (1) and Proposition \ref{prop:exc-rays}.

For the main statement, it suffices to show that $\cD^{\SV}(U)\cap \overline{\cD^+} = \{\ord_S\}$ and $\cD^{\SV, -}$ is path-connected. Path-connectedness follows from \cite[Theorem A]{Jan23} that $\Int(\cD^{\SV,-}) = G\cdot \Pi^{\circ}$ is homeomorphic to the hyperbolic plane $\bH^2$ which is path-connected, and the fact that $\Pi$ is path-connected. For a valuation $v\in \cD^{\SV}(U)\cap \overline{\cD^+}$, we can write $v = v_{\gamma,i}'$ where $1\leq i\leq 3 $,  $\gamma\in \sigma_0$ and $\gamma_3 \geq 3 \gamma_1+3\gamma_2$. 

If $\gamma_1$ and $\gamma_2$ are $\bQ$-linearly dependent, then we know that $v\in \oell_{p,q,i}$ where $p,q$ are coprime non-negative integers such that $p\gamma_2 = q\gamma_1$. Thus by Proposition \ref{prop:exc-lines-fg} we know that $v$ is either $\ord_S$ or $v_{(p,q,3p+3q-2),i}'$. The latter case cannot happen as $v_{(p,q,3p+3q-2),i}'\not\in \overline{\cD^+}$ by $3p+3q -2 < 3p + 3q$. 

If $\gamma_1$ and $\gamma_2$ are $\bQ$-linearly independent, then by Theorem \ref{thm:fg-multideg} and Definition \ref{def:geodesic} we know that $v$ is contained in a special geodesic simplex (line or triangle) denoted by $P_v$. After shrinking $P_v$, we may assume that it is contained
in a single cell. Hence Proposition~\ref{prop:geodesic-simplex} shows that $P_v$
is a rational line or triangle.
By assumption on $\gamma_1, \gamma_2$ we know that $P_v$ is not contained in any exceptional line. Thus it must intersect a closed $+$-exceptional ray at a point $v'$ as the union of all such rays is dense in $\overline{\cD^+}$ whose complement contains no rational line or triangle. Thus $v'$ is special which implies that either $v'=\ord_S$ or $v'\in \cD^-$ by Proposition \ref{prop:exc-lines-fg}. But both cases cannot happen as $P_v$ is very small so $\ord_S\not\in P_v$, and $v'$ lies in a closed $+$-exceptional ray which is disjoint from $\cD^-$. Thus the proof is finished. 
\end{proof}

\begin{proof}[Proof of Theorem \ref{thm:markov-fg}]
By Theorem \ref{thm:markov-special} and Proposition \ref{prop:exc-lines-fg} we know that $\cD^{\rfg}(U)$ contains $\cD^{\SV,-}$ and all closed exceptional lines $\oell_{p,q,i}$, and that $\cD^-\setminus \cD^{\SV,-}$ is contained in the union of all $\oell_{p,q,i}$. Thus $\cD^{\rfg}(U)$ contains $\cD^-$. To show the reverse containment, let $v_{\gamma,i}'\in \overline{\cD^+} = \cD(U)\setminus \cD^-$. Clearly, $v_{\gamma,i}'$ lies on a closed exceptional line if and only if $\gamma_1$ and $\gamma_2$ are linearly dependent over $\bQ$. Thus it suffices to show that $v_{\gamma,i}'$ is not finitely generated assuming $\gamma_1$ and $\gamma_2$ are linearly independent over $\bQ$.

Assume to the contrary that $v_{\gamma,i}'$ is finitely generated. After rescaling we may assume $\gamma_3 = 1$. If $\gamma_1,\gamma_2, 1$ are linearly independent over $\bQ$, then by \cite[Lemma 2.10]{LX18} and \cite[Lemma 4.4]{LXZ22} there exists a finitely generated geodesic triangle $P_\gamma\subset \overline{\cD^+}$ containing $v_{\gamma,i}'$. Thus we can choose some $v\in P_{\gamma}$ of rational rank $3$ such that $v$ is finitely generated. By Theorem \ref{thm:toric} we know that $v$ is special, i.e. $v\in \cD^{\SV}(U)\cap \overline{\cD^+}$. However, $\cD^{\SV}(U)\cap \overline{\cD^+}$ is contained in the union of all closed exceptional lines $\oell_{p,q,i}$, hence no valuation in $\cD^{\SV}(U)\cap \overline{\cD^+}$ has rational rank $3$, a contradiction. 

Next, we may assume that $\gamma_1, \gamma_2,1$ are linearly dependent over $\bQ$. Again by \cite[Lemma 2.10]{LX18} and \cite[Lemma 4.4]{LXZ22}, there exists a finitely generated geodesic line connecting $v_{\alpha,i}'$ and $v_{\beta,i}'$ that contains $v_{\gamma,i}'$ in the interior, and that the multi-degeneration $(\fX, \overline{\fS}+\fH)\to \bA^2$ induced by $v_{\alpha,i}'$ and $v_{\beta,i}'$ is  isotrivial over $\bA^2\setminus \bG_m^2$. Since $v_{\alpha,i}'$ and $v_{\beta,i}'$ are both divisorial, they belong to the union of all closed exceptional lines. In particular, there are finitely generated geodesic lines connecting $v_{\alpha,i}'$ and $v_{\beta,i'}$ to $\ord_S$ respectively. By Proposition \ref{prop:deg-normal-cone} we conclude that $\overline{\fS}_{(1,0)}$ and $\overline{\fS}_{(0,1)}$ are both irreducible, which implies that $\overline{\fS}_{(0,0)}$ is irreducible as well by the isotriviality of $\overline{\fS}$ over $\bA^2\setminus \bG_m^2$. Then applying Proposition \ref{prop:deg-normal-cone} again to $v_{\alpha,i}',v_{\beta,i'},\ord_S$ yields that they form a finitely generated geodesic triangle $P_{\alpha,\beta}$ which is contained in $\overline{\cD^+}$ by Proposition \ref{prop:geodesic-simplex}. As a result, there exists a valuation $v\in P_{\alpha,\beta}$ of rational rank $3$ that is finitely generated. By the same argument as in the previous paragraph we get a contradiction. Thus the proof is finished.
\end{proof}

\subsection{Topology of the special skeleton}
\label{sec:topology}
In this subsection, we study the topology of $\cD^{\SV}(U)$ and prove Theorems \ref{thm:intro-markov} and \ref{thm:counterex-LX24}.

\begin{thm}\label{thm:satake}
    There is a $G$-equivariant continuous bijection 
    \[
    \Psi: \cD^{\SV,-}\to \bH^2\cup \bP^1(\bQ)
    \]
    where the target is endowed with the Satake topology, and the $G$-action on the target is the $(\infty, \infty,\infty)$-triangle reflection group action. Moreover, the restriction $
\Psi^\circ:
\Int(\cD^{\SV,-})
\xrightarrow{\sim}
\bH^2$
is a $G$-equivariant homeomorphism.
\end{thm}

\begin{proof}
Recall that $\VPi$ is the set of vertices of $\Pi$ and $\Pi^\circ:=\Pi\setminus \VPi$.  We begin by recalling Jang's
$G$-equivariant homeomorphism between $G\cdot\Pi^\circ$ and $\bH^2$.

By Theorem \ref{thm:markov-special}, we have
$\cD^{\SV,-}=G\cdot\Pi$.
Moreover, by Proposition \ref{prop:exc-rays},
$G\cdot\Pi^\circ
=
\cD^-\setminus \bigcup_{p,q,i}\overline{\ell}^{\, -}_{p,q,i}.$
Hence $G\cdot\Pi^\circ$ is open in $\cD^-$. On the other hand, every point
of $G\cdot\VPi$ is the endpoint of an open exceptional ray contained in
$\cD^-\setminus\cD^{\SV,-}$, and therefore is not an interior point of
$\cD^{\SV,-}$. Thus we have $
\Int(\cD^{\SV,-})
=
G\cdot\Pi^\circ$.

Under the $G$-equivariant identification
$\Phi^-:\cD^-\xrightarrow{\sim}\Sk^{\trop}(U^-_K)$
of Theorem \ref{thm:trop-iso-action}, the subset $\Pi^\circ$ corresponds to the ping-pong table in \cite[Section~9]{Jan23}. By \cite[Theorem~9.3]{Jan23}, there is
an isomorphism
$\rho:G\xrightarrow{\sim}\langle r_1,r_2,r_3\rangle$
onto the $(\infty,\infty,\infty)$-triangle reflection group and a
$\rho$-equivariant homeomorphism
\[    \Psi^\circ:\Int(\cD^{\SV,-})=G\cdot\Pi^\circ\xrightarrow{\sim}\bH^2.
\]
Let
$T\subset \bH^2\cup \bP^1(\bQ)$
be the closed ideal triangle corresponding to $\Pi$, and set $T^\circ:=T\setminus \mathrm{Vert}(T)$.
More concretely, one may choose a homeomorphism $\Pi\xrightarrow{\sim}T$ matching the three
sides fixed by the Vieta involutions with the corresponding reflection
walls, and then extend it $G$-equivariantly.

Next, we extend $\Psi^\circ$ to the $G$-orbit of $\VPi$.
Let $v_1,v_2,v_3$ be the vertices of $\Pi$, and let
$\xi_1,\xi_2,\xi_3\in\bP^1(\bQ)$ be the corresponding ideal vertices
of $T$. 
We extend $\Psi^\circ$ by
\[
    \Psi(g\cdot v_i):=\rho(g)\cdot\xi_i \qquad \textrm{for }
     g\in G\textrm{ and } 1\le i\le3.
\]
This is well-defined. Indeed, if $v_i$ is the common fixed point of
$\tau_j$ and $\tau_k$, then the ping-pong description gives
$\Stab_G(v_i)=\langle\tau_j,\tau_k\rangle$,
while $\Stab_{\rho(G)}(\xi_i)=\langle r_j,r_k\rangle$.
These two stabilizers are identified by $\rho$. Thus we obtain a
$G$-equivariant map
\[
    \Psi:\cD^{\SV,-}=G\cdot\Pi
    \to \bH^2\cup\bP^1(\bQ).
\]

Next, we show that $\Psi$ is bijective. 
The restriction of $\Psi$ to $G\cdot\Pi^\circ$ is already a
homeomorphism onto $\bH^2$. It therefore remains only to consider the
vertices. In the hyperbolic model of \cite[Section~9]{Jan23}, the
$G$-translates of $T$ form the Farey ideal triangulation.
Hence its set of ideal vertices is precisely
$    \rho(G)\cdot\{\xi_1,\xi_2,\xi_3\}=\bP^1(\bQ)$.
Moreover, the equality of the stabilizers above shows that
\[
   \Psi|_{G\cdot \VPi}: G\cdot \VPi\to\bP^1(\bQ),
    \qquad g\cdot v_i\mapsto\rho(g)\cdot\xi_i,
\]
is injective. Hence $\Psi$ is a $G$-equivariant bijection.

Finally, it remains to prove continuity of $\Psi$ at the points of
$G\cdot\VPi$, since $\Psi^\circ$ is already a homeomorphism on
$G\cdot\Pi^\circ$. By $G$-equivariance, it suffices to consider a
vertex $v=v_1\in\VPi$. By Theorem \ref{thm:trop-iso-action} we have a $G$-equivariant homeomorphism $\Phi^-:\cD^-\xrightarrow{\sim}\Sk^{\trop}(U^-_K)$. Thus  we may equivalently work in the
tropical skeleton. After relabeling the coordinates, we may assume that
$v$ corresponds to $
(0,-\tfrac12,-\tfrac12)\in\Sk^{\trop}(U^-_K)$.
Then $\Stab_G(v)=\langle\tau_2,\tau_3\rangle$.
Let $\xi:=\Psi(v)\in\bP^1(\bQ)$ be the corresponding ideal vertex of
$T$.

Choose a sufficiently small horoball $\mathbb{B}\subset\bH^2$ based at $\xi$,
and set
$\mathbb{B}^*:=\mathbb{B}\cup\{\xi\}$.
Such sets form a neighborhood basis of $\xi$ for the Satake topology.
Moreover, after taking $\mathbb{B}$ sufficiently small, the cusp stabilizer
$\Stab_G(\xi)$ satisfies
$\mathbb{B}^*=\Stab_G(\xi)\cdot(\mathbb{B}^*\cap T)$.
By our choice of the homeomorphism $\Pi\to T$, the inverse image $W:=\Psi^{-1}(\mathbb{B}^*\cap T)\cap\Pi$
is a neighborhood of $v$ in $\Pi$. By equivariance, we have $\Stab_G(v)\cdot W\subset \Psi^{-1}(\mathbb{B}^*)$.
Thus it remains to show that $\Stab_G(v)\cdot W$ contains a neighborhood
of $v$ in $G\cdot\Pi$.

We verify this directly in the tropical skeleton.
Since $W$ is a neighborhood of $v$ in $\Pi$, there exists
$0<\epsilon<\frac12$ such that
$W_\epsilon:=\Pi\cap\{x_1>-\epsilon\}\subset W$.
The tropical Vieta formulas show that both $\tau_2$ and $\tau_3$
preserve the coordinate $x_1$. Hence
\[
\Stab_G(v)\cdot W_\epsilon
=
\bigl(\Stab_G(v)\cdot\Pi\bigr)\cap\{x_1>-\epsilon\}.
\]
Write $v_2=\left(-\frac12,-\frac12,0\right)$ and $
v_3=\left(-\frac12,0,-\frac12\right)$,
so that $\Pi=\operatorname{Conv}(v,v_2,v_3)$. On the set
\[
\Sigma_v:=\Sk^{\trop}(U^-_K)\cap\left\{x_1=-\tfrac12\right\},
\]
the involutions $\tau_2$ and $\tau_3$ act as the two reflections fixing 
$v_2$ and $v_3$, respectively. It follows that the
$\Stab_G(v)$-translates of the edge $\overline{v_2 v_3}$ cover the entire $\Sigma_v$. Thus 
$\Stab_G(v)\cdot\Pi$ is the join of $\{v\}$ and $\Sigma_v$.
Let $\ell_v^-:=\left\{(0,t,t)\mid t<-\frac12\right\}$
denote the open exceptional ray emanating from $v$. Then
\[
\bigl(\Stab_G(v)\cdot\Pi\bigr)\cup\ell_v^-
=
\Sk^{\trop}(U^-_K)\cap\left\{x_1\ge-\tfrac12\right\}.
\]
Therefore
\[
\bigl(\Stab_G(v)\cdot W_\epsilon\bigr)\cup\ell_v^-
=
\Sk^{\trop}(U^-_K)\cap\{x_1>-\epsilon\}.
\]
Since $G\cdot \Pi \subset \Sk^{\trop}(U_K^-)\setminus \ell_v^-$,
we conclude that
$\Stab_G(v)\cdot W_\epsilon$ contains the open neighborhood $(G\cdot\Pi)\cap\{x_1>-\epsilon\}$ of $v$ in $G\cdot\Pi$.
Hence $\Psi^{-1}(\mathbb{B}^*)$ contains an open neighborhood of $v$,
which proves that $\Psi$ is continuous at $v$. By $G$-equivariance,
$\Psi$ is continuous at every point of $G\cdot\VPi$.
\end{proof}

\begin{rem}
The continuous bijection $\Psi$ in Theorem~\ref{thm:satake} is
not a homeomorphism. Indeed, using the $G$-equivariant homeomorphism
$\Phi^-$, we work in $\Sk^{\trop}(U^-_K)$. Let $v=(0,-\tfrac12,-\tfrac12)\in\VPi$
and set $g:=\tau_2\tau_3\in\Stab_G(v)$. For $k\geq 3$, consider
$\ux_k:=(-\tfrac{1}{k},-\tfrac12,-\tfrac12+\tfrac{1}{k})\in\Pi^\circ$.
A direct computation from the tropical Vieta formulas gives
\[
g^m(\ux_k)
=
(-\tfrac{1}{k},-\tfrac12-\tfrac{2m}{k},
-\tfrac12-\tfrac{2m-1}{k})
\qquad \textrm{for }m\ge1.
\]
Thus we have $\ux_{k}\to v$ as $k\to\infty$, while
$g^k(\ux_{k})
\to
(0,-\tfrac52,-\tfrac52)\neq v$.
On the other hand, if $\xi:=\Psi(v)$, then
$\Psi(\ux_{k})\to\xi$. Since $g$ corresponds to a parabolic element
fixing $\xi$, and cusp neighborhoods of $\xi$ in the Satake topology
may be chosen $\Stab_G(\xi)$-invariant, we also have
$\Psi\bigl(g^k(\ux_{k})\bigr)
=
\rho(g)^k\Psi(\ux_{k})
\longrightarrow \xi$.
Thus $\Psi^{-1}$ is not continuous at $\xi$, and hence $\Psi$ is not a
homeomorphism. 
\end{rem}

\begin{thm}\label{thm:markov-not-locally-closed}
For any finite triangulation of $\cD(U)$, there exists an open simplex $C^\circ$ such that $
\cD^{\SV}(U)\cap C^\circ$
is not open in its closure in $C^\circ$. In particular,
$\cD^{\SV}(U)$ is not locally closed with respect to any finite triangulation
of $\cD(U)$.
\end{thm}

\begin{proof}
Fix a finite triangulation $\mathscr T$ of $\cD(U)$, and suppose, to the
contrary, that for every open simplex $C^\circ$ of $\mathscr T$, the set $\cD^{\SV}(U)\cap C^\circ$
is open in its closure in $C^\circ$. Denote by $\mathscr T^{(1)}$ the
$1$-skeleton of $\mathscr T$.

We first claim that $G\cdot\VPi\subset \mathscr T^{(1)}$.
Indeed, let $v\in G\cdot\VPi$. By the description of the exceptional
rays in Proposition~\ref{prop:exc-rays}, together with the local computation in the
proof of Theorem~\ref{thm:satake}, a sufficiently small neighborhood of $v$ in $\cD^-$ meets exactly one open exceptional ray $\ell_v^-$ emanating
from $v$. Hence a small open neighborhood of $v\in \cD^{\SV}$ is homeomorphic to
$\bR^2\setminus\bR_{>0}$, with the origin corresponding to $v$.
If $v$ belonged to an open $2$-simplex $C^\circ$ of $\mathscr T$, it
would follow that $\cD^{\SV}(U)\cap C^\circ$ is not open in its closure
in $C^\circ$, a contradiction. Thus $v\in\mathscr T^{(1)}$, proving the
claim.

Since $\mathscr T$ is finite, $\mathscr T^{(1)}$ is closed. On the other
hand, the accumulation set of $G\cdot\VPi$ is precisely
$\cD^0$ by Lemma \ref{lem:vertex-accumulation}.
Therefore $\cD^0\subset \mathscr T^{(1)}$.
Choose a point $p\in\cD^0$ lying in the relative interior of an edge of $\mathscr T$ and away from the vertices of $\mathscr T$. Since
$\cD^0\subset\mathscr T^{(1)}$ and $\mathscr T$ is finite, after shrinking to a sufficiently small neighborhood $W$ of $p$ we may assume
that
$W\cap\mathscr T^{(1)}=W\cap\cD^0$.
However, $p$ is an accumulation point of $G\cdot\VPi$. Hence $W$ contains a point $q\in G\cdot\VPi$. Since
$G\cdot\VPi\subset\mathscr T^{(1)}$ and $
G\cdot\VPi\cap\cD^0=\emptyset$, we know that $q \in (W\cap \mathscr{T}^{(1)})\setminus (W\cap \cD^0)$, which  contradicts the equality above.
Therefore, some open simplex $C^\circ$ of $\mathscr T$ satisfies that
$\cD^{\SV}(U)\cap C^\circ$ is not open in its closure in $C^\circ$.
\end{proof}

\begin{lem}\label{lem:vertex-accumulation}
The accumulation set of $G\cdot\VPi$ in $\cD(U)$ is precisely
$\cD^0$.
\end{lem}

\begin{proof}
By Proposition~\ref{prop:exc-rays}, the set $G\cdot\VPi$ consists of the endpoints
$[v'_{(p,q,\,3p+3q-2),i}]$ for $1\le i\le 3$,
where $p,q\in\bZ_{\geq 0}$ are coprime. After rescaling, these points can be written as
\[
\left[
v'_{\left(
\frac{p}{p+q},
\frac{q}{p+q},
3-\frac{2}{p+q}
\right),i}
\right].
\]
On the other hand, the $i$-th edge of $\cD^0$ is
\[
\left\{
\left[v'_{(t,1-t,3),i}\right]
\ \middle|\
0\leq t\leq 1
\right\}.
\]
Since the fractions $\frac{p}{p+q}$ with $p,q$ coprime are dense in
$[0,1]$, and $3-\frac{2}{p+q}\to 3$ as $p+q\to\infty$,
every point of $\cD^0$ is an accumulation point of $G\cdot\VPi$.

Conversely, let a sequence of distinct points of $G\cdot\VPi$
converge in $\cD(U)$. After passing to a subsequence, we may assume
that the index $i$ is fixed. Since there are only finitely many
coprime pairs $(p,q)$ with $p+q$ bounded, we must have
$p+q\to\infty$. After the above normalization, any limit point $[v_{\gamma,i}']$ satisfies $\gamma_3=3(\gamma_1+\gamma_2)$,
hence lies in $\cD^0$. Thus the accumulation set of
$G\cdot\VPi$ is exactly $\cD^0$.
\end{proof}

\begin{proof}[Proof of Theorem \ref{thm:intro-markov}]
This follows directly from Theorems \ref{thm:markov-special} and \ref{thm:satake}.
\end{proof}

\begin{proof}[Proof of Theorem~\ref{thm:counterex-LX24}]
The first statement is precisely Theorem~\ref{thm:markov-not-locally-closed}. We prove the local statement
by applying the cone construction as in \cite[Example~3.5]{LX24}.

Recall that $(X,D;0):=(\bP^3,\overline S+H;0)$ is the log CY-Fano compactification of $U$ considered above. Write
\[
F:=w\bigl(xyz+(x^2+y^2+z^2)w\bigr),
\]
so that $D=(F=0)$ and $\deg F=4$. Consider the affine cone
$Z:=\bA^4_{x,y,z,w}$ over $(X=\bP^3, \cO(1))$
with vertex $o$, and let
$D_Z:=(F=0)\subset Z$
be the cone over $D$.

For $u\in U^{\trop}=\LCP(X,D)$, define a valuation
$\widehat u=(\ord_o,u)$ on $Z$ by
\[
\widehat u\left(\sum_{m\geq 0}g_m\right)
:=
\min_{g_m\neq 0}\{m+u(g_m)\},
\]
where $g_m$ is homogeneous of degree $m$. Since $u$ is an lc place of $(X,D)$, we have
$A_X(u)=u(D)$,
and therefore
\[
A_Z(\widehat u)
=
4+A_X(u)
=
4+u(D)
=
\widehat u(F).
\]
Thus $\widehat u$ is an lc place of $(Z,D_Z)$ centered at $o$. 

Fix an integer $\ell>4$ and set $
\mathfrak a:=(F,\mathfrak m_o^\ell)$.
Choose sufficiently many general generators
$h_1,\ldots,h_N\in\mathfrak a$, and let
\[
D_Z':=\frac1N\sum_{j=1}^N(h_j=0).
\]
As in \cite[Example~3.5]{LX24}, $D_Z'$ is a $\bQ$-complement of the
smooth germ $o\in Z$ for which $o$ is the unique lc center. Moreover,
\[
\cD(Z,D_Z')
=
\left\{
\widehat u
\ \middle|\
u\in U^{\trop},\ A_X(u)\leq \ell-4
\right\}.
\]
This is because $\widehat u(\mathfrak a)
=
\min\{4+A_X(u),\ell\}$
whereas $A_Z(\widehat u)=4+A_X(u)$. On the other hand, if \(v\) is an lc place of \((Z,D'_Z)\), then
\[
A_Z(v)=v(\mathfrak a)\le v(F)\le A_Z(v), 
\]
because \((Z,D_Z)\) is lc. Hence equality holds throughout, so \(v\) is an lc place of \((Z,D_Z)\); the cone description then gives that $v$ and $\widehat u\) are the same after scaling.

The same cone construction identifies the Koll\'ar valuations among these
lc places. More precisely,
$\widehat u$ is a Koll\'ar valuation if and only if $
u$ is special. Indeed, the associated graded ring of $\widehat u$ is the
affine cone over the projective degeneration induced by $u$;  its $\Proj$ is precisely the degeneration used in the
definition of specialness. Hence
\[
\cD^{\rm KV}(Z,D_Z')
=
\{\ord_o\}
\cup
\left\{
\widehat u
\ \middle|\
[u]\in\cD^{\SV}(U),\
0<A_X(u)\leq\ell-4
\right\}.
\]
In other words, $\cD(Z,D_Z')$ is a truncated cone over $\cD(U)$ and
$\cD^{\rm KV}(Z,D_Z')$ is the corresponding truncated cone over
$\cD^{\SV}(U)$.

Suppose now that \cite[Conjecture~1.8(1)]{LX24} holds for $(Z,D_Z')$, and
let $\mathscr T$ be a rational triangulation satisfying the asserted
local closedness property. Choose
$0<c<\ell-4$
and consider the transverse slice
\[
\Sigma_c
:=
\{\widehat u\in\cD(Z,D_Z')\mid A_X(u)=c\}.
\]
Since every nontrivial ray in $U^{\trop}$ meets $\{A_X=c\}$ exactly once,
there are natural PL identifications
\[
\Sigma_c\cong\cD(U),
\qquad
\Sigma_c\cap\cD^{\rm KV}(Z,D_Z')
\cong\cD^{\SV}(U).
\]
Intersecting $\mathscr T$ with $\Sigma_c$ and subdividing the resulting
polyhedral decomposition gives a finite triangulation of $\cD(U)$. The
simplexwise local closedness property is preserved under taking this slice
and under subdivision. Hence for every open simplex $C^\circ$ of the
resulting triangulation, $\cD^{\SV}(U)\cap C^\circ$
would be open in its closure in $C^\circ$, contradicting
Theorem~\ref{thm:markov-not-locally-closed}.

Thus \cite[Conjecture~1.8(1)]{LX24} fails. Since
\cite[Conjecture~1.8(1')]{LX24} implies \cite[Conjecture~1.8(1)]{LX24},
it fails as well.
\end{proof}

\section{Discussions}\label{sec:discussions}

\subsection{Toward a non-Archimedean cone conjecture}

Let $U$ be an affine log CY variety. Let $U^{\SV,\mathrm{top}}\subset U^{\SV}$ denote the union of all
special geodesic cones of dimension $\dim U^{\trop}$, and denote its projectivization by $\cD^{\SV, \mathrm{top}}(U)\subset \cD^{\SV}(U)$.
We call them the top-dimensional special cone and the top-dimensional
special skeleton, respectively.

The examples in Sections~\ref{sec:surfaces} and~\ref{sec:Markov} suggest that the top-dimensional special skeleton may be
a natural domain on which the discrete dynamics of automorphisms become
polyhedral. In Examples~\ref{ex:nodal-cubic} and~\ref{ex:line-parabola}, as well as for the Markov cubic
complement threefold in Section \ref{sec:Markov}, the top-dimensional special skeleton is generated by finitely many geodesic
chambers up to the relevant automorphism group. On the other hand, Example~\ref{ex:cayley-cubic} suggests that sufficiently
wild automorphism dynamics may force the (top-dimensional) special skeleton to disappear:  every orbit on the dual complex is dense,
while $\cD^{\SV}(U)=\emptyset$. 

The following question is motivated by a possible non-Archimedean analogue
of the Kawamata--Morrison cone conjecture \cite{Mor93, Kaw97}, with the top-dimensional special skeleton
playing the role of the relevant cone.

\begin{que}\label{que:cone}
Assume that $U$ is an affine log CY variety of coregularity zero that admits no positive-dimensional torus action. Does the action of $\Aut(U)$ on $\cD^{\SV, \mathrm{top}}(U)$ admit a
rational polyhedral fundamental domain that is a finite union of special
geodesic simplices?
\end{que}

The coregularity-zero assumption ensures that $\cD(U)$ has the maximal
possible dimension $\dim U-1$, so that the question concerns
full-dimensional geodesic chambers.
There are two reasons for excluding torus actions from the question.
First, the algebraic torus itself behaves quite differently: for
$\bT=\bG_m^n$ one has
$\bT^{\SV}=\bT^{\rfg}=\bT^{\trop}\cong N_{\bR}$,
and the essential dynamics come from the linear $\GL(N)$-action; see Example \ref{ex:torus}.
Second, more generally, when $U$ admits a positive-dimensional torus
action, one can factor out the torus directions and
study a lower-dimensional torus-free reduction, which should capture the
essential discrete dynamics on the special skeleton.

The automorphism group need not be the optimal symmetry group for such a statement. In analogy with the usual Kawamata--Morrison cone conjecture, one may seek a larger group $\Adm(U)$ of admissible monodromy transformations arising from locally trivial families of affine log CY varieties; see \cite{GHK15, Li25}.
Such a group, if suitably defined, should act globally on the special
skeleton. This could be important, for instance, when $\Aut(U)$ is relatively small while
$\cD^{\SV,\mathrm{top}}(U)$ contains many special geodesic chambers. We leave the
definition of $\Adm(U)$ and the corresponding refinement of
Question~\ref{que:cone} to future work.


There is also a possible connection with Corti's conjectural picture for
toric specializations of klt Fano varieties \cite{Cor25}. Corti expects
such toric specializations to arise from torus charts of a cluster-type
mirror, with different charts related by mutations. Under the coregularity-zero assumption above,
top-dimensional special geodesic chambers come from toric special
degenerations, and it would be interesting to understand whether they
admit a similar mutation-theoretic description.

\subsection{Relation with cluster complexes}

More broadly, there is a natural mirror-symmetric interpretation of the special cone. In the expected mirror picture, integral tropical points $U^{\trop}(\bZ)$ of an
affine log CY variety $U$ index theta functions on its mirror
$U^\vee$. Thus one may view $U^{\SV}(\bZ)$ as singling out
a distinguished subset of theta functions on the mirror. It would be
interesting to understand the mirror-theoretic meaning of this subset.

When $U$ is a cluster variety, this suggests a more concrete
connection with the Fock--Goncharov cluster complex of the Fock--Goncharov 
dual cluster variety $U^\vee$; see \cite{GHKK18}. Indeed, we expect that the cluster complex of $U^\vee$ coincides with $U^{\SV, \mathrm{top}}$. This would also
relate the cluster modular group to the automorphism or monodromy
actions considered in the previous subsection. We leave a precise
formulation of this correspondence to future work.


\subsection{Finite generation and maximality}
Let $U$ be an affine log CY variety with a log CY-Fano compactification
$(X,D;\Delta)$. Choose $l\in \bN$ sufficiently divisible such that $L:=-l(K_X+\Delta)$
is Cartier, and set
\[
A:=\Gamma(U,\cO_U),\qquad R:=R(X,L).
\]
Although finite generation of a valuation is defined using the section
ring $R$, Theorem~\ref{thm:main} shows that it is intrinsic to $U$. It is
therefore natural to ask whether it can be detected directly from the
affine coordinate ring $A$.

\begin{conj}\label{conj:fg-affine}
    For any valuation $v\in U^{\trop}$, we have that $\gr_v R$ is finitely generated if and only if $\gr_v A$ is finitely generated.
\end{conj}

The forward implication of Conjecture \ref{conj:fg-affine} follows easily from the localization construction of
the Rees algebra, while the converse seems more difficult.




In view of Theorem~\ref{thm:maximality-intro}, it is natural to ask whether
maximality alone might already imply finite generation.

\begin{conj}\label{conj:max-fg}
Let $v\in U^{\trop}$ be a valuation. If $v$ is maximal with respect to $\preceq$, then $v$ is finitely generated.
\end{conj}

Together with Theorem~\ref{thm:maximality-intro}, Conjecture~\ref{conj:max-fg}
would characterize special valuations as precisely the maximal elements
of $U^{\trop}$. Via the cone construction, it is also closely related to
the finite generation problem for valuations computing log canonical
thresholds of graded sequences of ideals; see
\cite[Question~6.24]{Zhu25}.

\bibliography{ref}

\end{document}